%% file: rogers_compositionality.tex
\documentclass[a4paper,onecolumn, superscriptaddress,10pt, accepted=2022-05-06, issue=1, volume=5, shorttitle=papers/compositionality-5-1]{compositionalityarticle}
\pdfoutput=1

\usepackage[utf8]{inputenc}

\usepackage[bottom]{footmisc}

\usepackage[numbers,sort&compress]{natbib}
\usepackage{array}
\usepackage{verbatim}
\usepackage{subfig}
\usepackage{hyperref}
\usepackage{enumitem}
\setlist{nosep}
\usepackage{amsmath}
\usepackage{amssymb}
\usepackage{mathtools}
\usepackage{float}
\usepackage{tikz-cd}
\newcommand{\rfrak}{\mathfrak{r}}

\newcommand{\Bcal}{\mathcal{B}}
\newcommand{\Ccal}{\mathcal{C}}
\newcommand{\Dcal}{\mathcal{D}}
\newcommand{\Ecal}{\mathcal{E}}
\newcommand{\Fcal}{\mathcal{F}}
\newcommand{\Gcal}{\mathcal{G}}
\newcommand{\Ical}{\mathcal{I}}

\newcommand{\Pcal}{\mathcal{P}}
\newcommand{\Rcal}{\mathcal{R}}

\newcommand{\Tcal}{\mathcal{T}}

\newcommand{\Nbb}{\mathbb{N}}
\newcommand{\Qbb}{\mathbb{Q}}
\newcommand{\Rbb}{\mathbb{R}}
\newcommand{\Tbb}{\mathbb{T}}

\newcommand{\Zbb}{\mathbb{Z}}

\newcommand{\Set}{\mathbf{Set}}
\newcommand{\FinSet}{\mathbf{FinSet}}

\newcommand{\Top}{\mathbf{Top}}

\newcommand{\Hom}{\mathrm{Hom}}

\newcommand{\TOP}{\mathfrak{Top}}
\newcommand{\op}{^{\mathrm{op}}}

\newcommand{\too}{\twoheadrightarrow}

\DeclareMathOperator{\id}{id}

\DeclareMathOperator{\PSh}{PSh}
\DeclareMathOperator{\Sh}{Sh}
\DeclareMathOperator{\Geom}{Geom}
\DeclareMathOperator{\EssGeom}{EssGeom}

\DeclareMathOperator{\Ind}{Ind}
\DeclareMathOperator{\Cont}{Cont}

\DeclareMathOperator{\Sub}{Sub}

\DeclareMathOperator{\Aut}{Aut}
\DeclareMathOperator{\End}{End}

\usepackage{amsthm}
\newtheorem{thm}{Theorem}[section]
\newtheorem{prop}[thm]{Proposition}
\newtheorem{lemma}[thm]{Lemma}

\newtheorem{crly}[thm]{Corollary}
\newtheorem{schl}[thm]{Scholium}
\newtheorem{conj}{Conjecture}
\theoremstyle{definition}
\newtheorem{dfn}[thm]{Definition}
\newtheorem{xmpl}[thm]{Example}
\theoremstyle{remark}
\newtheorem{rmk}[thm]{Remark}

\begin{document}

\title{Toposes of Topological Monoid Actions}
\date{}
\author{Morgan Rogers}
\email{mrogers@uninsubria.it}
\orcid{0000-0002-0277-8217}
\affiliation{Universit\`a degli Studi dell{'}Insubria, Via Valleggio n. 11, 22100 Como, Italy}

\maketitle{}

\abstract{We demonstrate that categories of continuous actions of topological monoids on discrete spaces are Grothendieck toposes. We exhibit properties of these toposes, giving a solution to the corresponding Morita-equivalence problem. We characterize these toposes in terms of their canonical points. We identify natural classes of representatives with good topological properties, `powder monoids' and then `complete monoids', for the Morita-equivalence classes of topological monoids. Finally, we show that the construction of these toposes can be made ($2$-)functorial by considering geometric morphisms induced by continuous semigroup homomorphisms.}

\tableofcontents

\input{TTMA_Introduction}

\input{TTMA_Similar_Properties}

\input{TTMA_Monoids_with_Topologies}

\input{TTMA_Surjective_Point}

\input{TTMA_Semigroup_Homs}

\input{TTMA_Conclusion}

\bibliographystyle{plainnat}
\bibliography{classificationbibDOI}

\end{document}

%% file: TTMA_Introduction.tex
\section*{Introduction}

In a previous article \cite{TDMA}, we investigated properties of presheaf toposes of the form $\PSh(M) := [M\op,\Set]$ for a monoid $M$, whose objects are sets equipped with a right action of $M$. A natural direction to generalize this study is to view sets as discrete spaces and to consider the actions of a topological monoid on them. In order to analyze this case, we parallel the analogous categories for topological groups, which are well-studied.

For a topological group $(G,\tau)$, the category $\Cont(G,\tau)$ of continuous $G$-actions on discrete topological spaces is a Grothendieck topos. One way to prove this is to observe that there is a canonical geometric morphism $\PSh(G) \to \Cont(G,\tau)$ which is a surjection, see \cite[A4.2.4(a)]{Ele}. The inverse image functor of this morphism is the forgetful functor which sends a continuous $(G,\tau)$-set to its underlying $G$-set. The direct image functor is constructed explicitly by Mac Lane and Moerdijk in \cite[\S VII.3]{MLM}: it sends a $G$-set $X$ to the subset consisting of those elements whose `isotropy subgroup' is open; it follows that the counit of this morphism is monic and so (by \cite[A4.6.6]{Ele}, say) that the geometric morphism is moreover hyperconnected.

In this article we begin by extending these observations to categories of continuous actions of monoids. We take a rather classical approach at first: rather than considering genuine topological monoids (that is, monoids in the category of topological spaces), we consider endowing the underlying set of a monoid $M$ with an arbitrary topology $\tau \subseteq \Pcal(M)$. This approach is justified even in the group-theoretic setting: no part of the description of a continuous $(G,\tau)$-set relies on the fact that the topology $\tau$ makes $G$ a topological group, and we shall indeed see that the argument is valid even when this fails. \textit{A reader critical of this decision should be reassured by the fact that, as we shall eventually see in Theorem \ref{thm:tau} and Proposition \ref{prop:ctsx}, any `monoid with a topology' is in any case Morita-equivalent to a genuine topological monoid.}

\subsection*{Overview}

In Section \ref{ssec:necessary}, we exhibit the necessary data to establish that the forgetful functor from the category $\Cont(M,\tau)$ of continuous actions of a monoid with respect to an arbitrary topology $\tau$ to the topos $\PSh(M)$ is left exact and comonadic (Proposition \ref{prop:hyper}). The adjoint can be expressed using either clopen subsets of $(M,\tau)$ or open relations. From the existence of this adjunction we conclude that $\Cont(M,\tau)$ is an elementary topos, so that the forgetful functor is the inverse image of a hyperconnected geometric morphism, just as in the group case. In Section \ref{ssec:sgtrecap}, we recall theoretical results from our paper \cite{SGT} on supercompactly generated toposes, applying them in Section \ref{ssec:apply} to conclude that any topos of the form $\Cont(M,\tau)$ is moreover a supercompactly generated Grothendieck topos, which brings us to an intuitive Morita-equivalence result in terms of the (essentially small) category of continuous principal $M$-sets, Corollary \ref{crly:Morita}. Finally, in \ref{ssec:jcp} we show another property of the categories of principal continuous $M$-sets which has not yet been covered, indicating that our characterization of toposes of the form $\Cont(M,\tau)$ is not yet complete.

In Section \ref{sec:montop}, we examine the question of how much is recoverable about a topology $\tau$ on a monoid $M$ from the hyperconnected morphism $\PSh(M) \to \Cont(M,\tau)$. To do this, we construct the classical powerset $\Pcal(M)$ of $M$ as a right $M$-set in Section \ref{ssec:inverse}, and from this object recover in Section \ref{ssec:action} a canonical topology $\tilde{\tau}$, contained in $\tau$, making $(M,\tilde{\tau})$ a genuine topological monoid with an equivalent category of actions; given that Morita-equivalence for topological monoids is non-trivial, this is as good a result as we could have hoped for. In Section \ref{ssec:powder}, we show that we can further reduce this topological monoid to obtain a Hausdorff monoid $(\tilde{M},\tilde{\tau})$, still retaining the same topos of actions. The resulting class of representative topological monoids for toposes of the form $\Cont(M,\tau)$, which we call \textit{powder monoids}, have many special properties. In Section \ref{ssec:prodiscrete}, we show that this class includes, but is not limited to, the classes of prodiscrete monoids and nearly discrete groups; indeed, the reduction of a monoid to a powder monoid is analogous of the reduction of a group to a nearly discrete group in \cite[Example A2.1.6]{Ele}.

In Section \ref{sec:surjpt}, we consider the canonical surjective point of $\Cont(M,\tau)$, which is the composite of the canonical essential surjective point of $\PSh(M)$ and the hyperconnected morphism obtained in Section \ref{sec:properties}. Our aim is to characterize this class of toposes in terms of the existence of a point of this form, just as Caramello does for topological groups in \cite{TGT}. First, in Section \ref{ssec:EqRel}, we obtain a canonical small site for $\PSh(M)$ whose objects are right congruences, equivalent to the site of principal $M$-sets, and show that hyperconnected morphisms out of $\PSh(M)$ correspond to suitable subsites of this one. In particular, this provides a small site for $\Cont(M,\tau)$ (Scholium \ref{schl:Morita2}). By taking a limit indexed over such a site, we show in Section \ref{ssec:complete} that, analogously to the case of groups in \cite{TGT}, we can recover a presentation for the codomain of a hyperconnected geometric morphism out of $\PSh(M)$ as a topos of topological monoid actions. This presentation is obtained by topologizing the monoid of endomorphisms of the canonical point. Thus, the existence of a point factorizing as an essential surjection followed by a hyperconnected morphism characterizes this class of toposes (Theorem \ref{thm:characterization}). Moreover, the resulting \textit{complete monoids} are powder monoids (Proposition \ref{prop:Lpowder}), and any powder monoid presenting the same topos (equipped with the same canonical point) admits a dense injective monoid homomorphism to the canonical representative (Corollary \ref{crly:extend}). Paralleling the introduction of (algebraic) bases for topological groups, we show in Section \ref{ssec:base} that we can re-index the limit defining a complete monoid over a \textit{base of open congruences} in order to obtain a simpler expression for it and in certain cases deduce further properties. We briefly consider the topologies on the original monoid $M$ induced by hyperconnected morphisms out of $\PSh(M)$ in Section \ref{ssec:factor}.

Finally, since in \cite{TDMA} we saw that semigroup homomorphisms correspond to essential geometric morphisms between toposes of discrete monoid actions, in Section \ref{sec:homomorphism} we show that continuous semigroup homomorphisms between topological monoids induce geometric morphisms between the corresponding toposes of continuous actions (Lemma \ref{lem:cts}). As such, we show that $\Cont(-)$ defines a $2$-functor extending the presheaf construction for discrete monoids in \cite{TDMA}, which we may restrict to the class of complete monoids. In Section \ref{ssec:intrinsic} we record some intrinsic properties of the hyperconnected geometric morphism $\PSh(M) \to \Cont(M,\tau)$ when $M$ is a powder monoid or complete monoid, enabling us in Sections \ref{ssec:id} and \ref{ssec:monhom} to examine how, when a geometric morphism $g$ is induced by a continuous semigroup homomorphism $\phi$ between complete monoids, the properties of $g$ are reflected as properties of $\phi$. We show that the surjection--inclusion factorization of $g$ is canonically represented by the factorization of $\phi$ into a monoid homomorphism followed by an inclusion of a subsemigroup (Theorem \ref{thm:inccomplete}). Moreover, the hyperconnected--localic factorization of $g$ can be identified with the dense--closed factorization of $\phi$ (Theorem \ref{thm:locextend}). In both cases, the intermediate monoid is complete. Finally, in \ref{ssec:monads} we show that the classes of monoids we have been working with throughout assemble into reflective sub-($2$-)categories of the ($2$-)category of topological monoids.

In the conclusion, Section \ref{sec:conclusion}, we summarize the unresolved problems we have encountered along the way, address some questions about the constructiveness of our results including discussion of localic monoids, and suggest some future directions this research might proceed. Proposition \ref{prop:densegroup} hints at how properties of topological monoids lift to their toposes of actions.

While some knowledge of topos theory is (unsurprisingly) required for several parts of this paper, we believe that no more than a basic knowledge is needed to appreciate the key results we obtain. For readers experienced in topos theory, we note that we work exclusively with Grothendieck toposes over $\Set$, and some non-constructive arguments involving complementation are important in our developments. We discuss the resulting limitations on the generality of our results in Section \ref{ssec:localic}.

\subsection*{Acknowledgements}
The author would like to thank Olivia Caramello for her support, guidance and suggestions, and Riccardo Zanfa for his patient proof-reading.

This work was supported by INdAM and the Marie Sklodowska-Curie Actions as a part of the INdAM Doctoral Programme in Mathematics and/or Applications Cofunded by Marie Sklodowska-Curie Actions.

%% file: TTMA_Similar_Properties.tex
\section{Properties of categories of continuous monoid actions}
\label{sec:properties}

\subsection{Necessary clopens}
\label{ssec:necessary}

Throughout, we will refer to pairs $(M,\tau)$ where $M$ is a monoid (in $\Set$) and $\tau \subseteq \Pcal(M)$ is a topology on it. The multiplication on $M$ will be largely left implicit, but for expressing inverse images we shall denote it by $\mu$. There is no assumption here that $\tau$ makes $\mu$ continuous; when it makes makes $\mu$ continuous in its first (resp. second) argument, we say that \textit{$\tau$ makes the multiplication of $M$ left (resp. right) continuous}.

\begin{rmk}
As mentioned in the introduction, the motivation for using these `monoids with topologies' rather than genuine topological monoids is that the definition of continuous $M$-set to follow applies without modification to this larger class of objects, and because it shall turn out to be useful to think of the topology as equipped (rather than intrinsic) structure. We reassure the reader that we shall eventually be able to reduce any `monoid with a topology' to a Morita-equivalent topological monoid.
\end{rmk}

Consider a (right) $M$-set, expressed in the form of a set $X$ equipped with a right action $\alpha: X \times M \to X$ subject to the usual conditions. We say this is an \textbf{$(M,\tau)$-set} if the action $\alpha$ is continuous when $X \times M$ is endowed with the product topology of the discrete topology on $X$ and the topology $\tau$ on $M$. An $(M,\tau)$-set will be referred to simply as a \textbf{continuous $M$-set} when the topology $\tau$ is understood.

We begin by exhibiting necessary and sufficient conditions for an $M$-set to be continuous.
\begin{lemma}
\label{lem:Inx}
Let $M$ be a monoid equipped with a topology $\tau$ and $X$ an $M$-set. Then $X$ is an $(M,\tau)$-set if and only if for each $x \in X$ and $p \in M$, the set
\[\Ical_x^p := \{m \in M \mid xm = xp\}  \]
is open in $\tau$. We call the collection of all such $\Ical_x^p$ the \textbf{necessary clopens} for $X$.
\end{lemma}
\begin{proof}
For continuity we precisely require that for each open subset $U \subseteq X$, its preimage under the action is open. Since $X$ is discrete, without loss of generality we may assume $U = \{x'\}$ for some $x' \in X$. A subset of $X \times M$ is open if and only if its intersection with each open subspace of the form $\{x\} \times M$ with $x \in X$ is open.

Thus we require $\{m \in M \mid xm = x'\}$ to be open for each pair $x, x' \in X$. However, if $x' \neq xp$ for every $p \in M$, the corresponding set is empty and so automatically open. Otherwise, $x' = xp$ for some $p$, which gives the result.
\end{proof}

To justify the name `necessary clopens' rather than merely `necessary opens', note that for each fixed $x$, the sets $\Ical_x^p$ partition $M$, so $\Ical_x^p$ being open for every $p$ forces each such set to also be closed. 

An $M$-set $X$ being continuous requires the `stabilizer submonoids' $\Ical_x^1$ to be both open and closed for every $x$. When $M$ is a topological group we know that this condition is actually sufficient, since the other subsets in the partition are simply the right cosets of $\Ical_x^1$ (which are open because a topological group acts on itself by homeomorphisms) but this is not the case for monoids in general.

While necessary clopens are the most direct generalization of (the right cosets of) the stabilizer subgroups for the action of a group on a set, we can avoid the additional need to index over these by working with equivalence relations:
\begin{crly}
\label{crly:Rnx}
Let $M$ be a monoid equipped with a topology $\tau$ and $X$ an $M$-set. Then $X$ is an $(M,\tau)$-set if and only if for each $x \in X$, the equivalence relation
\[\rfrak_x := \{(p,q) \in M \times M \mid xp = xq\}  \]
is open in the product topology $\tau \times \tau$ on $M \times M$.
\end{crly}
\begin{proof}
If $X$ is an $(M,\tau)$-set, by Lemma \ref{lem:Inx} we have $\Ical_x^n \in \tau$ for each $n \in M$, and $\rfrak_x$ is precisely $\bigcup_{n \in M} \Ical_x^n \times \Ical_x^n$, so is open in $\tau \times \tau$. On the other hand, if $\rfrak_x$ is open, for fixed $p\in M$ each $(p,q) \in \rfrak_x$ is contained in an open rectangle $U_q \times V_q \subseteq \rfrak_x$. Thus $\Ical_x^p = \bigcup_{(p,q)\in \rfrak_x}V_q$ is open, as required.
\end{proof}

Observe that when $M$ is a group, $\rfrak_x$ is the relation that partitions $M$ into the right cosets of the stabilizer subgroup of $x$.

Working concretely with an equivalence relation from Corollary \ref{crly:Rnx} is equivalent to working with all of the clopens in a partition at once. For each result to follow we can therefore give an expression in terms of either the necessary clopens or the open relations.

\begin{prop}
\label{prop:hyper}
Suppose a monoid $M$ is equipped with a topology $\tau$. Then the forgetful functor $V : \Cont(M,\tau) \to \PSh(M)$ is left exact and comonadic; its right adjoint $R$ sends an $M$-set $X$ to:
\begin{align*}
R(X) & := \{x \in X \mid \forall p,q \in M, \, \Ical^p_{xq} \in \tau \} \\
& = \{x \in X \mid \forall q \in M, \, \rfrak_{xq} \in \tau \times \tau \}.
\end{align*}
Moreover, if $\tau$ makes the multiplication of $M$ left continuous then the expression for $R(X)$ simplifies to
\begin{align*}
R(X) & := \{x \in X \mid \forall p \in M, \, \Ical^p_x \in \tau \}\\
& = \{x \in X \mid \rfrak_{x} \in \tau \times \tau\}.
\end{align*}
\end{prop}
\begin{proof}
The definition ensures that $R(X)$ is closed under the action of $M$, since for any $x \in R(X)$ and $q \in M$, $\Ical_{xq}^p$ is open for every $p \in M$ by assumption, ensuring $xq \in R(X)$. Taking $q = 1$ for each $x \in R(X)$ demonstrates (by Lemma \ref{lem:Inx}) that $R(X)$ is a continuous $M$-set.

The inclusion $R(X) \hookrightarrow X$ is the universal morphism from a continuous $M$-set into $X$. Indeed, suppose $f : Y \to X$ is an $M$-set homomorphism with $Y$ a continuous $M$-set. Given $m \in \Ical^p_{f(y)q}$, there is an inclusion of subsets $\Ical_{yq}^m \subseteq \Ical_{f(y)q}^m$ since each $m' \in \Ical_{yq}^m$ has $f(y)qm = f(yqm) = f(yqm') = f(y)qm'$. So every $\Ical_{f(y)q}^p$ is open and the image of $f$ is contained in $R(X)$. It follows that $X \mapsto R(X)$ is a right adjoint for the forgetful functor, as required.

Since $V$ is full and faithful, it is conservative. A finite limit of discrete spaces is discrete, so a finite limit of continuous $(M,\tau)$-sets is precisely the limit of the corresponding $M$-sets. Thus $V$ is left exact, in particular preserving all equalizers. By any version of the comonadicity theorem, it follows that $V$ is comonadic.

Finally, observe that for $x \in R(X)$, $p,q \in M$, we have:
\[\Ical_{xq}^p = \{m \in M \mid xqm = xqp\} = \{m \in M \mid qm \in \Ical_{x}^{qp}\} = q^*(\Ical_{x}^{qp}),\]
where $q^*$ is the inverse image of multiplication on the left by $q$ (which shall be described in more detail in Section \ref{ssec:inverse}). Thus if $\tau$ makes multiplication by $q$ continuous then $\Ical_{xq}^p$ is open whenever $\Ical_{x}^{qp}$ is, whence we obtain the simplified expressions.
\end{proof}

We call $R(X)$ the subset of \textbf{continuous elements} of $X$ with respect to $\tau$ (even when multiplication is not left continuous with respect to $\tau$).

\begin{crly}
\label{crly:topos}
$\Cont(M,\tau)$ is an elementary topos.
\end{crly}
\begin{proof}
We have shown that $\Cont(M,\tau)$ is equivalent to the category of algebras for a cartesian comonad on $\PSh(M)$, which by \cite[Theorem A4.2.1]{Ele} makes $\Cont(M,\tau)$ an elementary topos.
\end{proof}

\begin{rmk}
\label{rmk:sgrp}
One might wonder what can be said of the continuous actions of a \textbf{semigroup} endowed with a topology. In \cite[Section 2]{TDMA} we observed that an action of a semigroup $S$ extends canonically to an action of the monoid $S_1$ obtained by adjoining a unit element (which must act as the identity). Similarly, given a topology on $S$, we may extend it to a topology on $S_1$ with an equivalent category of actions by making the singleton consisting of the adjoined unit an open subset, and extending this to a topology by taking unions with the existing opens. Thus once again, no generality is lost by considering only monoids equipped with topologies rather than arbitrary semigroups. 
\end{rmk}

\subsection{Recap on supercompactly generated toposes}
\label{ssec:sgtrecap}

Here we recall results which are proved in our previous theoretical paper, \cite{SGT}. That paper was motivated by this one, and we shall see in the next subsection how the general results we collect here apply to the special case of toposes of topological monoid actions. For each statement, we give the relevant reference to that paper; it should be noted that while the majority of the results quoted here are original to that paper, the definitions and some basic results appear elsewhere. For example, Definition \ref{dfn:scompact} can be found in \cite[Remark D3.3.10]{Ele}, and Lemma \ref{lem:pshf} appears as a comment before Proposition 4.3 of Bridge's thesis \cite{TAC}.

\begin{dfn}[{\cite[Definition 1.1.1]{SGT}}]
\label{dfn:scompact}
An object $C$ of a category $\Ecal$ is \textbf{supercompact} if any jointly epic family $\{A_i \to C \mid i \in I\}$ contains an epimorphism.
\end{dfn}

\begin{lemma}[{\cite[Lemma 1.2.1]{SGT}}]
\label{lem:screp}
Let $\Ecal \simeq \Sh(\Ccal,J)$ be a Grothendieck topos of sheaves on a small cite $(\Ccal,J)$. Then the supercompact objects of $\Ecal$ are quotients of the representable sheaves $\ell(C)$ for $C \in \Ccal$. In particular, the category of all supercompact objects is essentially small.
\end{lemma}

\begin{dfn}[{\cite[Definition 1.2.2]{SGT}}]
We say a topos is \textbf{supercompactly generated} if its (essential) set of supercompact objects is separating. We write $\Ccal_s$ for the category of supercompact objects.
\end{dfn}

In the special case of presheaf toposes, the representable presheaves are themselves supercompact objects, whence we conclude that:
\begin{lemma}[{\cite[Proposition 1.2.4(iii)]{SGT}}]
\label{lem:pshf}
Any presheaf topos is a supercompactly generated topos with enough points.
\end{lemma}

\begin{dfn}[{\cite[\S 1.8]{SGT}}]
\label{dfn:hype}
Recall that a geometric morphism $f:\Fcal \to \Ecal$ is \textbf{hyperconnected} if its inverse image functor is full and faithful (so expresses $\Ecal$ as a coreflective subcategory of $\Fcal$), and moreover $\Ecal$ is closed in $\Fcal$ under subobjects and quotients.
\end{dfn}

See \cite[\S A4.6]{Ele} for some background on hyperconnected (and localic) morphisms. The main result of interest to us presently is the following:

\begin{prop}[{\cite[Theorem 1.8.5]{SGT}}]
\label{prop:hype3}
Let $f:\Fcal \to \Ecal$ be a hyperconnected geometric morphism between elementary toposes. If $\Fcal$ is:
\begin{itemize}
 	\item a Grothendieck topos, or
 	\item a Grothendieck topos with enough points, or
 	\item a two-valued topos (having exactly two subterminal objects), or
 	\item a supercompactly-generated topos,
\end{itemize}
then so is $\Ecal$.
\end{prop}

\begin{rmk}
\label{rmk:atom}
The final point in Proposition \ref{prop:hype3} can be expanded as follows. Suppose that $P$ is a property of objects of a topos which descends along epimorphisms, in the sense that given an epimorphism $A \too B$, if $A$ satisfies $P$ then $B$ must also; several such properties appear in \cite[\S 4]{SCGI}, including the property of being an atom (having no non-trivial subobjects). Suppose moreover that objects with property $P$ are preserved and reflected by the inverse image of a hyperconnected geometric morphism $f:\Fcal \to \Ecal$. Then if $\Fcal$ has a separating set of objects with property $P$, so does $\Ecal$; explicitly, the latter set may be taken to be the collection of quotients of objects in the separating set for $\Fcal$ which lie in $\Ecal$. We apply this observation to atomic toposes in Proposition \ref{prop:densegroup} later on.
\end{rmk}

In order to understand supercompactly generated toposes, we studied their full subcategories of supercompact objects. We present the properties of these categories $\Ccal_s$ in a different order here than we did in \cite{SGT}, starting with a categorical characterization of them as `reductive' categories, and then presenting their more general properties.

\begin{dfn}[{\cite[Definition 1.3.3]{SGT}}]
\label{dfn:funnel}
A small indexing category $\Dcal$ is a \textbf{funnel} if it has a weakly terminal object. A \textbf{funneling diagram} in an arbitrary category $\Ccal$ is a functor $F: \Dcal \to \Ccal$ with $\Dcal$ a funnel. We shall denote the image of the weakly terminal object of $\Dcal$ by $D$; for example:
\[\begin{tikzcd}[row sep = small]
A_i \ar[dr, "f_i", shift left] \ar[dr, "f'_i"', shift right] & \\
\vdots & D. \\
A_j \ar[ur, "f_j", shift left] \ar[ur, "f'_j"', shift right] &
\end{tikzcd}\]
The colimit of $F$, if it exists, is an object $C$ of $\Ccal$ equipped with an epimorphism $f: F(D) \too C$ through which all legs of the colimit cone factor. Such a colimit will be called \textbf{funneling}. Any epimorphism obtainable in this way for some funneling diagram $F$ is a \textbf{strict epimorphism}, equivalently expressible as the universal coequalizer of all pairs it coequalizes. Notably, these include isomorphisms and regular epimorphisms.
\end{dfn}

\begin{dfn}[{\cite[Definition 2.1.1]{SGT}}]
\label{dfn:stable}
Let $\Ccal$ be a small category. We say a class $\Tcal$ of morphisms in $\Ccal$ is called \textbf{stable} if it satisfies the following three conditions:
\begin{enumerate}
	\item $\Tcal$ contains all identities;
	\item $\Tcal$ is closed under composition;
	\item For any $f : C \to D$ in $\Tcal$ and any morphism $g$ in $\Ccal$ with codomain $D$, there exists a commutative square
	\[\begin{tikzcd}
		A \ar[r, "f'"] \ar[d, "g'"'] & B \ar[d, "g"]\\
		C \ar[r, "f"'] & D
	\end{tikzcd}\]
	in $\Ccal$ with $f' \in \Tcal$.
\end{enumerate}
\end{dfn}

As we remarked in \cite{SGT}, this definition is precisely what is needed for the class $J_{\Tcal}$ of sieves containing $\Tcal$-morphisms to be a well-defined Grothendieck topology on $\Ccal$; we employ such a Grothendieck topology in Proposition \ref{prop:algbase} below.

\begin{dfn}[{\cite[Definitions 2.4.2, 2.4.10]{SGT}}]
\label{dfn:prereg}
A category $\Ccal$ is called \textbf{reductive} if it has all funneling colimits and its class of strict epimorphisms is stable. We call the Grothendieck topology generated by the class of strict epimorphisms the \textbf{reductive topology}, and denote it $J_r$.

A reductive category $\Ccal$ is called \textbf{effectual} if, for every funneling diagram $F: \Dcal \to \Ccal$ with colimit expressed by $\lambda: F(D_0) \too C_0$ and any object $C$ in $\Ccal$ admitting morphisms $g_1, g_2:C \rightrightarrows F(D_0)$ such that $\lambda \circ g_1 = \lambda \circ g_2$, there is a strict epimorphism $t: C' \too C$ such that $g_1 \circ t$ and $g_2 \circ t$ lie in the same connected component of $(C' \downarrow F)$.
\end{dfn}

\begin{prop}[{\cite[Theorem 2.4.12, Corollary 1.4.5]{SGT}}]
\label{prop:correspond}
Up to equivalence, there is a one-to-one correspondence between supercompactly generated toposes and effectual, reductive categories. The correspondence sends a topos to its category of supercompact objects and an effectual, reductive category to the topos of sheaves for the reductive topology on that category.

In particular, if $\Ecal$ and $\Ecal'$ are supercompactly generated toposes and $\Ccal_s$, $\Ccal'_s$ are their respective categories of supercompact objects, then $\Ecal \simeq \Ecal$ if and only if $\Ccal_s \simeq \Ccal_s$.
\end{prop}

\begin{thm}[{\cite[Corollary 1.3.10, Lemma 1.3.12, Corollary 1.3.14, Lemma 1.5.5, Scholium 1.5.6, Proposition 1.8.1, Corollary 1.8.2]{SGT}}]
\label{thm:Cs}
Let $\Ccal_s$ be the category of supercompact objects in a supercompactly generated, two-valued Grothendieck topos $\Ecal$. Then beyond $\Ccal_s$ being reductive and effectual, we have that:
\begin{enumerate}
	\item All monomorphisms in $\Ccal_s$ are regular and coincide with those in $\Ecal$;
	\item All epimorphisms in $\Ccal_s$ are strict and coincide with those in $\Ecal$;
	\item The classes of epimorphisms and monomorphisms in $\Ccal_s$ form an orthogonal factorization system;
	\item $\Ccal_s$ has a terminal object $1$, and every object is well-supported;
	\item $\Ccal_s$ has cokernels, which is to say pushouts along the unique morphism to the terminal object:
	\[\begin{tikzcd}
		A \ar[d,"!_A"',two heads] \ar[r, "f"] \ar[dr, phantom, "\lrcorner", very near start] & B \ar[d, two heads] \\
		1 \ar[r] & B/f.
	\end{tikzcd}\]
\end{enumerate}
\end{thm}

\subsection{Corollaries for toposes of topological monoid actions}
\label{ssec:apply}

In light of Corollary \ref{crly:topos}, we see that the adjunction $(V \dashv R)$ is a hyperconnected geometric morphism $\PSh(M) \to \Cont(M,\tau)$. Recalling from \cite{TDMA} that $\PSh(M)$ is two-valued, we may apply Lemma \ref{lem:pshf} and Proposition \ref{prop:hype3} to conclude that:

\begin{crly}
\label{crly:conttop}
Any topos of the form $\Cont(M,\tau)$ is a supercompactly generated, two-valued Grothendieck topos with enough points.
\end{crly}

The fact that $\PSh(M)$ is supercompactly generated is implicitly important in Hemelaer's work in \cite{TGRM}: when identifying those toposes of $G$-equivariant sheaves on a space $X$ which are equivalent to one of the form $\PSh(M)$, they arrive at the definition of a \textit{minimal basis}, which corresponds to a base of supercompact open sets.

Having reached Corollary \ref{crly:conttop} abstractly, it may not be immediately obvious what the supercompact objects are in this case.

\begin{dfn}
We shall call an object $N$ in $\PSh(M)$ a \textbf{principal}\footnote{Some readers might prefer the term \textit{cyclic}.} \textbf{right $M$-set} if it is a quotient of $M$, in that there exists an epimorphism $M \too N$. Such an $M$-set is generated by a single element, the image of $1 \in M$ under the given epimorphism. Similarly, given a topology $\tau$ on $M$, we say an $(M,\tau)$-set $N$ is \textbf{principal} if $V(N)$ is a principal right $M$-set.
\end{dfn}

\begin{prop}
\label{prop:prince}
The supercompact objects of $\Cont(M,\tau)$ are precisely the principal $M$-sets. As such, these form an effectual, reductive category, with all the properties of Theorem \ref{thm:Cs}.
\end{prop}
\begin{proof}
Clearly this is true in $\PSh(M)$, since by definition the principal $M$-sets are exactly the quotients of the representable $M$-set $M$. It follows that the supercompact objects of $\Cont(M,\tau)$ are the continuous principal $M$-sets, since the inverse image of a hyperconnected morphism preserves and reflects supercompact objects.
\end{proof}

By an observation made in \cite[following Lemma 1.21]{SGT}, we recover the intuitive fact that every $(M,\tau)$-set is the union of its principal sub-$M$-sets. Indeed, this fact could have been proved directly; we felt that the route via abstraction highlighted that this property was not particular to toposes of monoid actions. From Proposition \ref{prop:correspond}, we have:

\begin{crly}
\label{crly:Morita}
Let $\Ccal_s$ be the category of continuous principal $(M,\tau)$-sets. Then there is an equivalence $\Cont(M,\tau) \simeq \Sh(\Ccal_s,J_r)$. In particular, topological monoids $(M,\tau)$ and $(M',\tau')$ are \textbf{Morita equivalent}, which is to say that $\Cont(M,\tau) \simeq \Cont(M',\tau')$, if and only if they have equivalent categories of continuous principal $M$-sets.
\end{crly}

\begin{xmpl}
\label{xmpl:infeq}
This result can be practically applied. For example, it shows that any monoid endowed with a topology for which there are infinitely many distinct isomorphism classes of continuous principal actions cannot be Morita-equivalent to any finite monoid. Of course, when the monoids involved are large enough, even the categories of principal actions can be hard to work with, so some alternative ways of generating Morita equivalences are desirable; we shall see some in subsequent sections.
\end{xmpl}

\begin{xmpl}
\label{xmpl:zero}
To present a more categorical example, we recall that a \textbf{zero element} of a monoid $M$ is an element $z \in M$ such that $mz = z = zm$ for all $m \in M$.

Let $(M,\tau)$ be a topological monoid with a zero element and $(M',\tau')$ another topological monoid. Then if $\Cont(M,\tau) \simeq \Cont(M',\tau')$, it must be that every principal $(M',\tau')$-set has a unique fixed point, since this is true in $\PSh{M}$ and the category of principal $(M,\tau)$-sets is a full subcategory containing $1$. In particular, if $M'$ is a group and $M$ is as above, then $\Cont(M,\tau) \simeq \Cont(M',\tau')$ if and only if both $\tau$ and $\tau'$ are indiscrete topologies.
\end{xmpl}

In Section \ref{ssec:EqRel}, we shall provide an alternative presentation of the site $\Ccal_s$ of continuous principal $M$-sets in terms of right congruences.

\begin{rmk}
\label{rmk:compact}
In \cite{SGT}, we also treat the broader class of \textbf{compactly generated toposes}. Without going into extraneous detail, the compact objects of $\Cont(M,\tau)$ are the \textit{finitely generated continuous $M$-sets}, and the category of these provides a larger site presenting $\Cont(M,\tau)$, as well as an alternative Morita equivalence condition. We felt that there was not sufficient added theoretical value to cover this perspective in detail in this paper.
\end{rmk}

A feature of hyperconnected morphisms which was not covered in \cite{SGT} is that they provide a way to compute exponential objects in the codomain topos using those in the domain topos.

\begin{lemma}
\label{lem:expo}
Let $h:\Fcal \to \Ecal$ be a (hyper)connected geometric morphism and let $X$, $Y$ be objects of $\Ecal$. Then the exponential object $Y^X$ in $\Ecal$ can be computed as $h_*\left(h^*(Y)^{h^*(X)}\right)$.
\end{lemma}
\begin{proof}
We check the universal property:
\begin{align*}
& \Hom_{\Ecal}\left(Z,h_*\left(h^*(Y)^{h^*(X)}\right)\right) \\ \cong \, &
\Hom_{\Fcal}(h^*(Z),h^*(Y)^{h^*(X)})\\ \cong \, &
\Hom_{\Fcal}(h^*(Z) \times h^*(X),h^*(Y))\\ \cong \, &
\Hom_{\Ecal}(Z \times X,Y),
\end{align*}
where the latter isomorphism is obtained from full faithfulness and preservation of products by $h^*$.
\end{proof}

\begin{crly}
\label{crly:expo}
Let $X$, $Y$ be $(M,\tau)$-sets. Then the exponential object $Y^X$ in $\Cont(M,\tau)$ is $R\left(\Hom_{\PSh(M)}(M \times V(X),V(Y))\right)$, which consists of the continuous elements of the exponential object $V(Y)^{V(X)}$ in $\PSh(M)$.
\end{crly}
\begin{proof}
Applying Lemma \ref{lem:expo}, it suffices to compute $V(Y)^{V(X)}$ in $\PSh(M)$. The underlying set is given by $\Hom_{\PSh(M)}(M,Q^P) \cong \Hom_{\PSh(M)}(M \times P,Q)$, by the universal property of exponentials. $M$ acts by multiplication in the first component, so that given $h:M \times P \to Q$, $h \cdot m$ is the mapping $(n,p) \mapsto h(mn,p)$.
\end{proof}

\subsection{The joint covering property}
\label{ssec:jcp}

One might wonder if the properties of $\Cont(M,\tau)$ identified in Corollary \ref{crly:conttop} are enough to characterize toposes of this form. For comparison, in the work of Caramello in \cite{TGT}, it is shown that a topos is equivalent to the topos of actions of a topological group if and only if it is an atomic, two-valued topos admitting a special surjective point\footnote{The inverse image of this point is an extension of the $J_{at}$-flat functor represented by a $\Ccal$-universal and $\Ccal$-ultrahomogeneous object $u$ in $\mathrm{Ind}$-$\Ccal$; see \cite[Theorem 3.5]{TGT}.}. These conditions look a lot like the properties in Corollary \ref{crly:conttop}, except we have replaced `atomic' by `supercompactly generated' and have weakened the existence of a special point to the mere existence of \textit{enough} points.

Of course, we also know that toposes of the form $\Cont(M,\tau)$ have a special surjective point, obtained as the composite of the canonical point of $\PSh(M)$ and the hyperconnected morphism $\PSh(M) \to \Cont(M,\tau)$. Here we observe an additional property of categories of principal $M$-sets and an example of a topos having all of the properties of Corollary \ref{crly:conttop} but whose category of supercompact objects fails to have this additional property.

\begin{dfn}
\label{dfn:jcp}
We say a small category $\Ccal$ has the \textbf{joint covering property} if for any pair of objects $A,B$ of $\Ccal$ there exists an object $N$ of $\Ccal$ admitting epimorphisms to $A$ and $B$.
\end{dfn}

If $\Ccal$ is a poset, the joint covering property is equivalent to $\Ccal$ having a lower bound for any pair of elements. If $\Ccal$ has binary products, it corresponds to the property that the projection maps from any binary product should be epimorphisms. The category of non-empty sets has this property; more generally, the category of well-supported objects of a topos always has this property. In contrast, any non-trivial category with a strict initial object must fail to have the joint covering property.

\begin{lemma}
\label{lem:Mjcp}
Consider the topos $\PSh(M)$; let $\Ccal_s$ be its subcategory of supercompact objects. Then $\Ccal_s$ has the joint covering property.
\end{lemma}
\begin{proof}
Given principal $M$-sets $N_1,N_2$ with generators $n_1,n_2$, consider the product $N_1 \times N_2$. The principal sub-$M$-set $N$ of this product generated by $(n_1,n_2)$ clearly admits the desired epimorphisms to $N_1$ and $N_2$.
\end{proof}

By applying a topological argument, we could directly extend the proof of Lemma \ref{lem:Mjcp} to the corresponding result for $\Cont(M,\tau)$. However, in the spirit of Proposition \ref{prop:hype3}, we once again give a more general argument for hyperconnected morphisms.

\begin{prop}
\label{prop:hypejcp}
Let $\Fcal$ be a topos and $\Ccal'_s$ its subcategory of supercompact objects. Suppose $\Ccal'_s$ has the joint covering property and $f:\Fcal \to \Ecal$ is a hyperconnected geometric morphism. Then the corresponding subcategory $\Ccal_s$ of $\Ecal$ also has the joint covering property.
\end{prop}
\begin{proof}
Since $f$ is hyperconnected, $\Ecal$ is closed in $\Fcal$ under products and subobjects. Given a joint cover with domain $X$ in $\Ccal'_s$ of a pair of objects $A,B$ in $\Ccal_s$, we may take the image of the canonical morphism $X \to A \times B$ to obtain a joint cover which is a subobject of the product, and hence also lies in $\Ccal_s$. Note that since the functor $\Ccal_s \to \Ccal'_s$ is full and faithful, we do not need to worry whether epimorphisms in $\Ccal'_s$ coincide with those in $\Fcal$: any epimorphism in $\Ccal'_s$ will also be one in $\Ccal_s$ (see Remark \ref{rmk:strict} below).
\end{proof}

\begin{crly}
The category of principal $(M,\tau)$-sets in $\Cont(M,\tau)$ has the joint covering property. 
\end{crly}

\begin{xmpl}
\label{xmpl:nonjcp}
At this point we can present an example of a two-valued, supercompactly generated topos with enough points which is not equivalent to $\Cont(M,\tau)$ for any topological monoid $(M,\tau)$. Consider the following category, $\Ccal$:
\[\begin{tikzcd}
	X \ar[r, shift left, two heads] \ar[loop left] &
	1 \ar[l, shift left, hook] \ar[r, shift left, hook] &
	Y \ar[l, shift left, two heads] \ar[loop right, "{,}"]
\end{tikzcd}\]
where identity morphisms are omitted and the outside loops are the idempotent endomorphisms whose splitting gives the terminal object. We can check directly that this is a reductive category: there are relatively few colimits that need to be checked, and the only non-trivial non-identity strict epimorphisms (marked with double-headed arrows) are stable thanks to the following rectangles: 
\[\begin{tikzcd}
	X \ar[d, equal] \ar[r, two heads] & 1 \ar[r, hook] &
	Y \ar[d, two heads] \\
	X \ar[rr, two heads] & & 1
\end{tikzcd}
\hspace{10pt}
\begin{tikzcd}
	Y \ar[d, equal] \ar[r, two heads] & 1 \ar[r, hook] &
	X \ar[d, two heads] \\
	Y \ar[rr, two heads] & & 1.
\end{tikzcd}\]
Since all of the epimorphisms split, the reductive topology coincides with the trivial topology, as noted in \cite[Remark 2.4]{SGT}.

Therefore, let $\Ecal$ be the presheaf topos $\Sh(\Ccal,J_r) \simeq \PSh(\Ccal)$, which is supercompactly generated and, being a presheaf topos, has enough points. One can compute the subterminal objects of this topos directly to verify that this topos is two-valued.

We can also compute directly that the category $\Ccal_s$ of supercompact objects of $\Ecal$ is equivalent to $\Ccal$ (so $\Ccal$ is an effective reductive category), and so does not have the joint covering property. Thus $\Ecal$ is not equivalent to a topos of the form $\Cont(M,\tau)$.

For a family of related examples, we can let $M$ and $M'$ be non-trivial monoids each having a zero element (see Example \ref{xmpl:zero} above). Then their idempotent-completions each have a terminal object; we may construct a category $\Ccal$ by gluing these idempotent completions along their respective terminal objects. The category of presheaves on this category will have the properties of Corollary \ref{crly:conttop}, but $\Ccal_s$ will not have the joint covering property (because there can be no joint covering of $M$ and $M'$). The above is the case where $M = M'$ is the two-element monoid with both elements idempotent.
\end{xmpl}

\begin{rmk}
\label{rmk:strict}
The category of supercompact objects in $\Cont(M,\tau)$ has the even more restrictive property that the covering morphisms in Definition \ref{dfn:jcp} may be chosen to be \textit{strict} epimorphisms. In a general supercompactly generated topos, a morphism in $\Ccal_s$ is epimorphic in the ambient topos if and only if it is a strict epimorphism in $\Ccal_s$, by \cite[Corollary 1.15]{SGT}. Incidentally, this `strict joint covering property' for supercompact objects implies two-valuedness of a supercompactly generated topos. The ordinary joint covering property does not have this implication, since the category of supercompact objects in the topos of presheaves on any meet semi-lattice has the joint covering property, and any non-trivial such topos is not two-valued.
\end{rmk} 

Even including the joint covering property to the list of properties derived previously, it is not clear whether we obtain a complete characterization of toposes of the form $\Cont(M,\tau)$, since there is no canonical way of reconstructing a topological monoid given only the reductive category of principal $(M,\tau)$-sets and no additional data (such as their underlying sets). In particular, we have not yet arrived at a complete answer to the question of when a supercompactly generated, two-valued Grothendieck topos $\Ecal$ is equivalent to one of the form $\Cont(M,\tau)$. We shall return to this question in Section \ref{sec:surjpt}.

%% file: TTMA_Monoids_with_Topologies.tex
\section{Monoids with topologies}
\label{sec:montop}

In this section we examine the extent to which the topology on the monoid $(M,\tau)$ can be recovered from the hyperconnected geometric morphism $\PSh(M) \to \Cont(M,\tau)$.

\subsection{Powersets and inverse image actions}
\label{ssec:inverse}

If $M$ acts on a set $X$ on the \textit{left}, then $M$ has a corresponding \textit{right} action on its powerset $\Pcal(X)$ via the `inverse image' action, $A \mapsto g^*(A) = \{x \in X \mid gx \in A\}$; it is easily checked that $(gh)^* = h^*g^*$. Note that if $M$ is a group then $g^*$ is simply (element-wise) left multiplication by $g^{-1}$.

If $t : X \to Y$ is a homomorphism of left $M$-sets, so $t(g \cdot x) = g\cdot t(x)$ for every $x \in X$, then we can define $t^{-1} : \Pcal(Y) \to \Pcal(X)$ sending $B$ to $t^{-1}(B)$, since
\begin{align*}
g^*(t^{-1}(B)) & = \{x \in X \mid g \cdot x \in t^{-1}(B)\}\\
&= \{x \in X\mid t(g \cdot x) \in B\}\\
&= \{x \in X\mid g \cdot t(x) \in B\}\\
&= \{x \in X\mid t(x) \in g^*(B)\} = t^{-1}(g^*(B)).
\end{align*}
Thus we obtain a functor $\Pcal : [M,\Set]\op \to [M\op,\Set]$, which is self-adjoint: the dual functor $\Pcal\op: [M\op,\Set] \to [M, \Set]\op$ is left adjoint to $\Pcal$. This adjunction is, by construction, a lifting of the powerset adjunction on $\Set$ along the forgetful functor from $[M\op,\Set]$, in the sense that the following diagram commutes:
\[\begin{tikzcd}
\Set \ar[d,bend right,"\Pcal\op"'] \ar[d, phantom, "\dashv"] & {[M\op,\Set]} \ar[l,"U"'] \ar[d,bend right,"\Pcal\op"'] \ar[d, phantom, "\dashv"]\\
\Set\op \ar[u,bend right,"\Pcal"'] & {[M,\Set]}\op \ar[l,"U"'] \ar[u,bend right,"\Pcal"'].
\end{tikzcd}\]
The purpose of introducing this adjunction is to identify some special $M$-sets. First and foremost, the action of $M$ on itself by left multiplication gives a canonical right $M$-action on $\Pcal(M)$ which (even \textit{a priori}) seems a good starting point from which to recover a topology.

In our previous work \cite{TDMA}, we were able to identify a representing monoid $M$ for $\PSh(M)$ as the representing object for the forgetful functor $U$ in the diagram above. We can do something very similar here:
\begin{lemma}
\label{lem:represent}
$\Pcal(M)$ represents the composite functor $\Pcal\op \circ U : \PSh(M) \to \Set\op$. In particular, it is uniquely determined as an object of $\PSh(M)$ by the choice of representing monoid $M$.
\end{lemma}
\begin{proof}
Passing around the square and applying Yoneda, we obtain natural isomorphisms:
\[\Hom_{\PSh(M)}(X,\Pcal(M)) \cong \Hom_{[M,\Set]}(M,\Pcal\op(X)) \cong U(\Pcal\op(X)) \cong \Pcal\op(U(X)),\]
as required. 
\end{proof}
We can in fact deduce that this composite functor is comonadic, so that $\PSh(M)$ is comonadic over $\Set\op$, but since the existing tools for comparing toposes with cotoposes (beyond those used to show the existence of colimits in toposes) are not well-developed to the author's knowledge, we shall take a different route to derive further properties of $\Pcal(M)$.

Note that the two-element set $2$ represents $\Pcal: \Set \to \Set\op$. By passing through the available adjunctions, we find that for every right $M$-set $X$,
\begin{align*}
\Hom_{\PSh(M)}(X,\Pcal(M)) & \cong \Pcal\op(U(X))\\
&\cong \Hom_{\Set}(U(X),2)\\
&\cong \Hom_{\PSh(M)}(X,\Hom_{\Set}(M,2)).
\end{align*}
That is, $\Pcal(M) \cong \Hom_{\Set}(M,2)$ as right $M$-sets, which is clear at the level of underlying sets, but the fact that the actions coincide was not apparent \textit{a priori}.

\begin{rmk}
Localic geometric morphisms over a topos $\Ecal$ correspond to internal locales in $\Ecal$, by \cite[Lemma 1.2]{factorizationI}, say. The correspondence sends a morphism $f: \Fcal \to \Ecal$ to the internal locale $f_*(\Omega_{\Fcal})$, where $\Omega_{\Fcal}$ is the subobject classifier of $\Fcal$.

Recalling that $2$ is the subobject classifier for $\Set$, we have just shown that $\Pcal(M)$ is (the underlying object of) the internal locale corresponding to the canonical point of $\PSh(M)$; this provides another way to deduce the second statement in Lemma \ref{lem:represent}, and endows $\Pcal(M)$ with a canonical order relation (which coincides with the usual inclusion ordering).
\end{rmk}

\begin{lemma}
\label{lem:Pcal}
$\Pcal(M)$ is an internal Boolean algebra in $\PSh(M)$. In particular, it has a distinguished non-trivial automorphism, complementation. There are exactly two morphisms $1 \to \Pcal(M)$. Also, $\Pcal(M)$ has the subobject classifier $\Omega$ of $\PSh(M)$ as a subobject. Finally, $\Pcal(M)$ is a coseparator. 
\end{lemma}
\begin{proof}
The structure of a Boolean algebra involves only finite limits, so Boolean algebras are preserved by both direct and inverse image functors; thus $\Pcal(M)$ inherits the Boolean algebra structure from $2$. The two morphisms $1 \to \Pcal(M)$ correspond to the empty set and the whole of $M$; these are the only two since composing the canonical point with the global sections morphism must give the identity geometric morphism on $\Set$, which means $\Gamma(\Pcal(M)) = \Hom_{\PSh(M)}(1,\Pcal(M)) \cong 2$.

The usual argument showing that the category of coalgebras for a left exact comonad is a topos (see \cite[\S V.8]{MLM}) exhibits the subobject classifier as an equalizer of two endomorphisms of the free coalgebra on the subobject classifier; this free algebra is precisely $\Pcal(M)$. More specifically, the endomorphisms are the identity and the morphism sending a subset $A$ to those $m \in M$ for which $m^*(A) = M$. From these expressions we recover the fact that $\Omega \hookrightarrow \Pcal(M)$ is the collection of right ideals of $M$. Since the subobject classifier of a topos is always injective, we in fact can conclude that $\Omega$ is a retract of $\Pcal(M)$; a canonical retraction map sends a subset of $M$ to the right ideal it generates.

Finally, the functor $\Pcal \circ U$ is a composite of faithful functors so it is faithful, meaning its representing object must be a coseparator.
\end{proof}

Having established some key properties of $\Pcal(M)$ as an object of $\PSh(M)$, we examine how the necessary clopens from Lemma \ref{lem:Inx} behave as elements of $\Pcal(M)$.

\begin{lemma}
\label{lem:InA}
Given $A \in \Pcal(M)$, $p \in M$, we have $\Ical_{A}^{p} = \Ical_{M \backslash A}^{p}$; moreover,
\[\Ical_{A}^{p} \subseteq
\begin{cases}
A & \text{if } p \in A\\
M \backslash A & \text{if } p \notin A.
\end{cases}\]
\end{lemma}
\begin{proof}
By definition, $\Ical_{A}^{p} = \{m \in M \mid m^*(A) = p^*(A)\}$. Since inverse images respect complementation, we have $m^*(A) = p^*(A)$ if and only if $m^*(M \backslash A) = p^*(M \backslash A)$, and hence $\Ical_{A}^{p} = \Ical_{M \backslash A}^{p}$, as claimed.

Now, without loss of generality, suppose $p \in A$, else we may exchange $A$ and $M \backslash A$. Then $1 \in p^*(A)$. Given $m \in \Ical_{A}^{p}$, it follows that $1 \in m^*(A)$ which forces $m \in A$. Thus $\Ical_{A}^{p} \subseteq A$.
\end{proof}

\begin{lemma}
\label{lem:InA2}
Suppose $X$ is any $M$-set, $x \in X$ and $p \in M$. Let $A = \Ical_{x}^{p} \in \Pcal(M)$. Then for any $p' \in A$, the inclusion in Lemma \ref{lem:InA} holds with equality: $\Ical_{A}^{p'} = A$.
\end{lemma}
\begin{proof}
Suppose $m \in A$ so that $xm = xp = xp'$. Then $m^*(A) = \{m' \in M \mid xmm' = xp\} = \{m' \in M \mid xp'm' = xp\} = p'{}^*(A)$, so $m \in \Ical_{A}^{p'}$. This proves the reverse inclusion to that in Lemma \ref{lem:InA}.
\end{proof}

Note that the complement of $A$ in Lemma \ref{lem:InA2} may split into multiple sets of the form $\Ical_{A}^{p}$ for $p \notin A$, but we at least retain that $\Ical_{x}^{p} \subseteq \Ical_{A}^{p}$ for each $p$.

\subsection{Action topologies}
\label{ssec:action}

We have by now developed sufficient tools to reconstruct a topology from the hyperconnected morphism $\PSh(M) \to \Cont(M,\tau)$.

\begin{thm}
\label{thm:tau}
Suppose $M$ is a monoid equipped with a topology $\tau$, and $V,R$ are as in Proposition \ref{prop:hyper}. Consider $\Pcal(M)$ equipped with the inverse image action. Then the underlying set of
\[T := VR(\Pcal(M)) = \{A \subseteq M \mid \forall p, q \in M, \, \Ical_{q^*(A)}^{p} \in \tau \}\]
is a base of clopen sets for a topology $\tilde{\tau} \subseteq \tau$ such that $\Cont(M,\tilde{\tau}) = \Cont(M,\tau)$ as sub-categories of $\PSh(M)$. Moreover, $\tilde{\tau}$ is the coarsest topology on $M$ with this property.
\end{thm}
\begin{proof}
We extracted the expression for $T$ from the construction of $R$ in Proposition \ref{prop:hyper}. By Lemma \ref{lem:InA}, every $A \subseteq M$ is a union over its elements $t$ of the sets $\Ical_{A}^{t}$, so if $A \in T$ then $A$ is necessarily open. Similarly, Lemma \ref{lem:InA} guarantees that $M \backslash A \in T$ whenever $A \in T$, since $\Ical_{M \backslash A}^{p} = \Ical_{A}^{p}$ and $q^*(M \backslash A) = M \backslash q^*(A)$. It follows that each $A \in T$ is clopen with respect to $\tau$.

To show that $T$ is a base for a topology it suffices to show that $A \cap B$ is in $T$ whenever $A$ and $B$ are. Directly,
\begin{equation}
\label{eq:intersection}
\Ical_{q^*(A \cap B)}^{p} = \{m \in M \mid (qm)^*(A) \cap (qm)^*(B) = (qp)^*(A) \cap (qp)^*(B)\};
\end{equation}
if $p'$ is any element of this set, then by inspection $\Ical_{q^*(A)}^{p'} \cap \Ical_{q^*(B)}^{p'} \subseteq \Ical_{q^*(A \cap B)}^{p'} = \Ical_{q^*(A \cap B)}^{p}$ is an open neighbourhood of $p'$ contained in it, ensuring that the latter is open. We conclude $A \cap B \in T$, as required.

If $X$ is an $M$-set which is continuous with respect to the generated topology $\tilde{\tau}$, then $\Ical_{x}^{p} \in \tilde{\tau} \subseteq \tau$ for every $x \in X$, $p \in M$ so $X$ is continuous with respect to $\tau$.

Conversely, if $X$ is continuous with respect to $\tau$, so $\Ical_{x}^{p} \in \tau$ for all $x \in X$ and $p \in M$, we want to show that each $\Ical_{x}^{p} \in \tilde{\tau}$. Writing $A = \Ical_{x}^{p}$, this is equivalent to showing that $\Ical_{q^*(A)}^{p} \in \tau$ for each $p, q \in M$. Given $p_1 \in \Ical_{q^*(A)}^{p}$, consider the open set $\Ical_{xq}^{p_1}$. We have $p_2 \in \Ical_{xq}^{p_1}$ if and only if $xqp_2 = xqp_1$. Consequently,
\[p_2^*q^*(A) = \{m \in M \mid xqp_2m = xp\} = \{m \in M \mid xqp_1m = xp\} = p_1^*q^*(A).\]
But $\Ical_{q^*(A)}^{p_1}$ is precisely $\{m \in M \mid m^*q^*(A) = p_1^*q^*(A)\}$ so we conclude $\Ical_{xq}^{p_1} \subseteq \Ical_{q^*(A)}^{p_1} = \Ical_{q^*(A)}^{p}$, and hence the latter is open as required.

Finally, to show that $\tilde{\tau}$ is the coarsest such topology, suppose $\tau'$ is some topology on $M$ such that any $M$-set $X$ is continuous with respect to $\tau'$ if and only if it is continuous with respect to $\tau$. Then the respective inclusions of $\Cont(M,\tau)$ and $\Cont(M,\tau')$ into $\PSh(M)$ are isomorphic. Thus $T$ is computed in the same way with respect to either topology, and by repeating the above argument, we have $\tilde{\tau} \subseteq \tau'$, as claimed.
\end{proof}

\begin{rmk}
Note that the caveat `as subcategories of $\PSh(M)$' in Theorem \ref{thm:tau} likely cannot be removed in full generality, since a sufficiently large monoid could admit two topologies with distinct categories of continuous $M$-sets which happen to be equivalent. We have not constructed such an example, since in this paper we are primarily interested in examining $\Cont(M,\tau)$ as a topos under $\PSh(M)$.
\end{rmk}

\begin{dfn}
\label{dfn:action}
The topology $\tilde{\tau}$ derived in Theorem \ref{thm:tau} will be called the (right) \textbf{action topology induced by $\tau$}. By the final statement of Theorem \ref{thm:tau}, the construction of $\tilde{\tau}$ is idempotent (see Lemma \ref{lem:G3} below for a deeper exploration of this). As such, we say $\tau$ is an \textbf{action topology} if $\tilde{\tau} = \tau$.
\end{dfn}

We will continue to employ the notation $T:=VR(\Pcal(M))$ for the Boolean algebra of necessary clopens when the topology $\tau$ is understood. Rather than considering the full action topology $\tilde{\tau}$, it will sometimes be more convenient to work directly with $T$, since this is an object residing in the toposes we are studying.

\begin{schl}\label{schl:Tseparator}\hspace{-0.2cm}\footnote{Taking after Johnstone in \cite[pp. xiv, footnote 7]{Ele}, we call a result a `scholium' if it is a consequence of preceding proofs, as opposed to a `corollary' which is a consequence of preceding result statements.} Considering the Boolean algebra $T$ as an object of $\Cont(M,\tau)$, it inherits all of the properties we observed in $\Pcal(M)$ in Lemmas \ref{lem:represent} and \ref{lem:Pcal}: it represents $\Pcal\op \circ U \circ V: \Cont(M,\tau) \to \Set\op$, is a complete internal Boolean algebra with exactly two global sections, and is a coseparator which contains the subobject classifier of $\Cont(M,\tau)$ as an (order-inheriting) subobject. Explicitly, the subobject classifier of $\Cont(M,\tau)$ consists of the left ideals of $M$ lying in $T$.
\end{schl}
\begin{proof}
For the first part, we extend the proof of Lemma \ref{lem:represent} with the observation that
\[\Hom_{\Cont(M,\tau)}(X,R(\Pcal(M))) \cong \Hom_{\PSh(M)}(V(X),\Pcal(M)),\]
where $R(\Pcal(M))$ is $T$ viewed as an object of $\Cont(M,\tau)$.

For the second part, all of the arguments in the proof of Lemma \ref{lem:Pcal} carry over with $T$ in place of $\Pcal(M)$.
\end{proof}

\begin{schl}
\label{schl:Tsuff}
Let $X$ be a right $M$-set continuous with respect to $\tau$, and let $T$ be as in Theorem \ref{thm:tau}. Then for every $x \in X$, $p \in M$, we have $\Ical_{x}^{p} \in T$. In particular, we do not need to generate $\tilde{\tau}$ in order to verify continuity.
\end{schl}
\begin{proof}
Consider the construction in the proof of Theorem \ref{thm:tau}; in it, we showed that $\Ical_{q^*(\Ical_{x}^{p})}^{p'} \in \tau$ for each $x \in X$ and $p, p', q \in M$. But this is exactly the condition needed for $\Ical_{x}^{p}$ to be in $T$, since it ensures that the action of $M$ on $\Pcal(M)$ is continuous on the sub-$M$-set generated by $\Ical_{x}^{p}$.
\end{proof}

\begin{schl}
\label{schl:base}
The clopen sets of the form $\Ical_A^p$ for $A \in T$ (or more generally, the necessary clopens of all $(M,\tau)$-sets) also form a base for $\tilde{\tau}$.
\end{schl}
\begin{proof}
Given $a \in A$, we have $a \in \Ical_A^a \subseteq A$, so each member of $T$ is a union of members of the given form, as required.
\end{proof}

\subsection{Powder monoids}
\label{ssec:powder}

Action topologies have much more convenient properties than arbitrary topologies. Most notably:
\begin{prop}
\label{prop:ctsx}
The multiplication on $M$ is continuous with respect to $\tilde{\tau}$ for any starting topology $\tau$.
\end{prop}
\begin{proof}
Given $A \in \tilde{\tau}$ and a pair $(a,b) \in \mu^{-1}(A)$, we have $a \in \Ical_{A}^{a}$ and $b \in a^*(A)$ by inspection. Since the inverse image action commutes with arbitrary unions and the generating set $T$ of $\tilde{\tau}$ is closed under the action of $M$ on $\Pcal(M)$, we deduce that $a^*(A) \in \tilde{\tau}$. Given any $m \in \Ical_{A}^{a}$, $n \in a^*(A)$, we have $m^*(A) = a^*(A)$ and hence $n \in m^*(A)$ and $mn \in A$. Thus $\Ical_{A}^{a} \times a^*(A) \subseteq \mu^{-1}(A)$. It follows that $\mu^{-1}(A) \in \tilde{\tau} \times \tilde{\tau}$, as required.
\end{proof}

Thus, almost miraculously, $(M,\tilde{\tau})$ is a topological monoid. That is, from the perspective of continuous actions on discrete sets, there is no loss in generality in assuming that the topology on the monoid makes its multiplication continuous, which shall come as a relief to the modern algebraist.

Passing to the action topology also sheds the extraneous local richness of the original topology which the discrete sets being acted on are oblivious of.
\begin{lemma}
\label{lem:connected}
Let $(M,\tau)$ be a locally connected topological monoid. Then $\tilde{\tau}$ is generated by the connected components with respect to $\tau$. In particular, if $(M,\tau)$ is connected, $(M,\tilde{\tau})$ is indiscrete.
\end{lemma}
\begin{proof}
If $\tau$ makes $M$ locally connected, the connected components of $M$ are clopen. For each $x \in M$ let $C_x$ be the connected component containing $x$. We claim each $C_x$ is a member of $\tilde{\tau}$. Indeed, given $p \in M$, $p^*(C_x)$ is clopen (since $(M,\tau)$ \textit{is} a topological monoid here), and hence is some union of connected components. Observe that $\Ical_{C_x}^{p} = \{m \in M \mid m^*(C_x) = p^*(C_x)\}$ contains $C_p$: multiplication on the right is continuous, so whenever $py \in C_x$, we have $my \in C_x$ for every $m$ in $C_p$. It follows that $\Ical_{C_x}^{p}$ is a union over $q \in \Ical_{C_x}^{p}$ of components $C_q$, and so is open. Thus $C_x \in \tilde{\tau}$ and since these are the minimal clopen sets we are done.

If $(M,\tau)$ is connected, the only clopen subsets of $M$ are $\emptyset$ and $M$, so $\tilde{\tau}$ contains only these.
\end{proof}

Lemma \ref{lem:connected} means that, for example, $\Rbb$ with its usual topology goes from being Hausdorff (or even stronger, normal) to being indiscrete upon passing to $\tilde{\tau}$. On the other hand, other properties of a topology $\tau$ are preserved by passing to $\tilde{\tau}$. For example:
\begin{lemma}
\label{lem:compactau}
Suppose $\tau$ is a compact topology on a monoid $M$. Then $\tilde{\tau}$, being a coarser topology than $\tau$, is compact too.
\end{lemma}

By definition, action topologies are \textbf{zero-dimensional}, since they have a base of clopen sets. See \cite[Section I.4, Figure 9]{Counter} for a helpful diagram of how this property interacts with some basic separation properties. We list some of them here:
\begin{lemma}
\label{lem:T0T2}
Suppose $\tau = \tilde{\tau}$ is an action topology and $(M,\tau)$ is Kolmogorov (satisfies the $T_0$ separation axiom). Then $(M,\tau)$ has the properties (listed in order of decreasing strength) of being totally separated and regular, totally disconnected and Urysohn, and Hausdorff ($T_2$).
\end{lemma}

\begin{dfn}
\label{dfn:powder}
A topological monoid $(M,\tau)$ which is $T_0$ and such that $\tau$ is a right action topology (that is, such that $\tau$ has a basis of clopen sets $U$ such that $\Ical_U^p = \{q \mid q^*(U) = p^*(U)\} \in \tau$ for every $p$) shall be called a (right) \textbf{powder monoid}; the name is motivated the separation properties exhibited in Lemma \ref{lem:T0T2}. We have avoided the name `action monoid' since it conflicts with the terminology `monoid actions'.
\end{dfn}

Note that there is implicit asymmetry in Definition \ref{dfn:powder}, and indeed we may define a \textit{left} powder monoid to be a topological monoid $(M,\tau)$ which is Hausdorff and such that $\tau$ is a right action topology on $M\op$. We shall discuss these in more detail in Sections \ref{ssec:monads} and \ref{ssec:moremonoid}. For the time being, we write `powder monoid' to mean `right powder monoids' unless otherwise stated.

\begin{rmk}
\label{rmk:Q}
The properties of Lemma \ref{lem:T0T2} do not characterize powder monoids. For example, $\Qbb$ with its usual topology is a $T_0$ and zero-dimensional topological group, but we find that, just like $\Rbb$, its corresponding action topology is trivial.
\end{rmk}

\begin{thm}
\label{thm:Hausd}
Given a monoid with an arbitrary topology $(M,\tau)$, there is a canonical (right) powder monoid, which we shall by an abuse of notation denote by $(\tilde{M},\tilde{\tau})$, such that $\Cont(M,\tau) \simeq \Cont(\tilde{M},\tilde{\tau})$ and the canonical points of these toposes coincide. 
\end{thm}
\begin{proof}
We first construct the action topology $\tilde{\tau}$ corresponding to $\tau$ from Theorem \ref{thm:tau}. We shall show that the equivalence relation on $M$ relating points which are topologically indistinguishable with respect to $\tilde{\tau}$ is a two-sided congruence on $M$.

Suppose $m_1,m'_1$ and $m_2,m'_2$ are two pairs of topologically indistinguishable points in the sense that every open set of $\tilde{\tau}$ containing $m_i$ also contains $m'_i$ and vice versa. Then given an open set $U \in \tilde{\tau}$ containing $m_1m_2$, we have $m_2 \in m_1^*(U)$ and hence $m'_2 \in m_1^*(U)$. Moreover, $m_1 \in \Ical_U^{m_1}$ so $m'_1 \in \Ical_U^{m_1}$, which is to say that $m_1^*(U) = {m'_1}^*(U)$ and so $m'_1m'_2 \in U$.

Moreover, the actions of any pair of topologically indistinguishable points $m,m'$ of $(M,\tilde{\tau})$ on any $(M,\tilde{\tau})$-set are forced to be equal: if we had $xm \neq xm'$ we would have $\Ical_x^m$ containing $m$ but not $m'$ and therefore not open. Thus the continuous actions restrict to the quotient $\tilde{M}$ of $M$ by this relation.

The monoid $\tilde{M}$ inherits its topology from $(M,\tilde{\tau})$; we abuse notation and call the inherited topology $\tilde{\tau}$ too, since the frames of open sets of the two topologies are isomorphic. It is easily checked that $\tilde{\tau}$ is still an action topology on $\tilde{M}$, and $(\tilde{M},\tilde{\tau})$ is Kolmogorov by construction, as required.

Since we have not modified the forgetful functor (the underlying sets of the actions remain the same), this construction depends only on $M$ and the canonical point of $\Cont(M,\tau)$.
\end{proof}

While convenient, it is unavoidable that the construction of Theorem \ref{thm:Hausd} relies on our original representing monoid $(M,\tau)$. In Section \ref{ssec:complete}, we shall construct a representing powder monoid for a topos of the form $\Cont(M,\tau)$, in general different from the one constructed above, which depends only on the canonical point.

\subsection{Prodiscrete monoids, nearly discrete groups}
\label{ssec:prodiscrete}

There are plenty of nontrivial examples of powder monoids.

\begin{dfn}
Recall that a \textbf{prodiscrete monoid} is a topological monoid $(M,\tau)$ obtained as a (projective) limit of discrete monoids,
\[M = \varprojlim_{i\in I} M_i,\]
with $\tau$ the coarsest topology making each projection map continuous, which has a base of opens of the form $\pi_i^{-1}(\{m_i\})$ with $m_i \in M_i$. Often the limit is taken to be filtered or such that all of the monoid homomorphisms involved are surjections, but we do not require these restrictions.
\end{dfn}

\begin{xmpl}
\label{xmpl:idemclosed}
For those readers unfamiliar with prodiscrete monoids, we construct an example now which will be useful later. Consider the `truncated addition' monoids $N_{a,1}$, indexed by integers, $a \geq 0$ consisting of the integers $\{0,\dotsc,a\}$ equipped with the operation $\mu(p,q) = \min\{p+q,a\}$. For each $a \leq a'$ we have a surjective monoid homomorphism $N_{a',1} \too N_{a,1}$. The projective limit of the resulting sequence of monoids can be identified with $\Nbb \cup \{\infty\}$, equipped with addition extended in the obvious way, and the topology on it coincides with the one-point (Alexandrov) compactification of $\Nbb$ as a discrete topological space.
\end{xmpl}

\begin{prop}
\label{prop:prodisc}
Any prodiscrete monoid is a (right) powder monoid.
\end{prop}
\begin{proof}
It suffices to show that open sets of the form $U = \pi_i^{-1}(A)$ with $A \subseteq M_i$ are continuous elements of the topology. Indeed, given $\alpha = (a_i)_{i \in I} \in M$,
\begin{align*}
\Ical_U^{\alpha} &=
\{\beta = (b_i)_{i \in I} \in M \mid \alpha^*(U) = \beta^*(U)\}\\
&= \{\beta \in M \mid \forall c \in M_i, \, a_ic \in A \Leftrightarrow b_ic \in A \}\\
&\supseteq \{\beta \in M \mid b_i = a_i\} = \pi_i^{-1}(\{a_i\})
\end{align*}
contains an open neighbourhood of $\alpha$ and hence is open in the prodiscrete topology $\tau$, and by a similar argument, for any $\alpha' = (a'_i)_{i \in I}$,
\[\alpha'{}^*(U) = \pi_i^{-1}({a'_i}^*(A))\]
is of the same form, so is open in $\tau$ as required.
\end{proof}

Given the motivation of the present work, it is natural to wonder what happens when we apply the construction of Theorem \ref{thm:Hausd} to groups.

\begin{dfn}
A topological group is said to be \textbf{nearly discrete} if the intersection of all open subgroups contains only the identity element; see Johnstone \cite[Example A2.1.6]{Ele} or Caramello \cite[comments following Proposition 2.4]{TGT}.
\end{dfn}

\begin{lemma}
\label{lem:powdergrp}
A topological group is a (right) powder monoid if and only if it is nearly discrete and has a neighbourhood base of open subgroups at the identity; we accordingly call such groups \textbf{powder groups}.
\end{lemma}
\begin{proof}
Given such a topological group $(G,\tau)$, $g \in G$ and a neighbourhood $U \in \tau$ of $g$, since multiplication by any element of $G$ on either side preserves opens, we may without loss of generality suppose that $g = 1 \in U$, and so that $U$ is an open subgroup.

Then $\Ical_U^1$, being the set of $h \in G$ such that $h^{-1}U = U$, contains (and so is equal to) $U$ and in particular is open in $G$, ensuring that $U$ is in the action topology corresponding to $\tau$, whence $\tau$ is an action topology and hence $(G,\tau)$, being Hausdorff (since given any element distinct from the identity we can find an open subgroup which does not contain it), is a powder monoid.

Conversely, a powder group has a basis of the identity consisting of the isotropy subgroups $\Ical_U^1$ with $U$ varying over the open neighbourhoods of the identity, and being Hausdorff forces such a group to be nearly discrete.
\end{proof}

\begin{xmpl}
\label{xmpl:notpro}
Lemma \ref{lem:powdergrp} allows us to present an example of a powder monoid which is not a prodiscrete monoid. Consider the group of automorphisms of $\Nbb$ with the stabilizers of finite subsets defined to be open subgroups (as suggested in \cite[Example A2.1.6]{Ele}). Any prodiscrete group is the limit of its quotients by normal open subgroups, but this group has no proper open normal subgroups. Incidentally, this nearly discrete group is one of the many powder monoids representing the Schanuel topos, see Example \ref{xmpl:Schanuel} below.
\end{xmpl}

Of course, the class of prodiscrete groups forms a subclass of the class of powder groups. An even more refined class is the following:
\begin{dfn}
A group is \textbf{profinite} if it is (expressible as) a directed projective limit of finite groups. They can alternatively be characterized as compact, totally disconnected groups; see \cite[Proposition 1.1.7]{Profinite}.
\end{dfn}

\begin{schl}
\label{schl:Hausgrp}
If $(G,\tau)$ is a group with an arbitrary topology, then the corresponding powder monoid $(\tilde{G},\tilde{\tau})$ is a powder group. If $(G,\tau)$ is compact, then $(\tilde{G},\tilde{\tau})$ is profinite.
\end{schl}
\begin{proof}
The proof that the equivalence relation identifying $\tilde{\tau}$-indistinguishable points respects multiplication also demonstrates that the equivalence class containing the inverse $g^{-1}$ of an element $g$ is an inverse for the equivalence class of $g$ in the quotient $\tilde{M}$. Thus $(\tilde{G},\tilde{\tau})$ is a group, and hence a powder group, as claimed. Adding Lemma \ref{lem:compactau} gives the profinite result.
\end{proof}

Thus applying Theorem \ref{thm:Hausd} to any topological group brings us naturally to the class of groups described by Johnstone in \cite[Example A2.1.6]{Ele} as canonical representives for toposes of topological group actions, which are precisely the powder groups. This is not the end of the story, however: see Remark \ref{rmk:completegrp}.

%% file: TTMA_Surjective_Point.tex
\section{The canonical surjective point}
\label{sec:surjpt}

In Section \ref{sec:properties}, we made extensive use of the hyperconnected morphism $f: \PSh(M) \to \Cont(M,\tau)$ which we constructed to demonstrate that $\Cont(M,\tau)$ is a topos. In this section, we show that the existence of a hyperconnected morphism with domain $\PSh(M)$ entirely characterizes toposes of this form.

From \cite{TDMA}, we know that the (discrete) monoids presenting $\PSh(M)$ correspond to the surjective essential points of that topos. There may be multiple possible presentations, but there is a unique one for each point. Without loss of generality, we fix a presentation of $\PSh(M)$ and its corresponding canonical point (whose inverse image functor is the forgetful functor), and examine equivalent presentations in the next section.

Suppose we are given a topos $\Ecal$ having a point $p$ which factorizes through the canonical point of $\PSh(M)$ via a hyperconnected morphism $h$:
\begin{equation}
\label{eqn:p}
\begin{tikzcd}
\Set \ar[r, bend left=50, "- \times M"] \ar[r, bend right=50, "\Hom_{\Set}(M{,}-)"']
\ar[r, phantom, shift left=6, "\bot", near end] \ar[r, phantom, shift right=4, "\bot", near end]
& {\PSh(M)} \ar[l, "U"', near start] \ar[r, shift right = 2, "h_*"'] \ar[r, phantom, "\bot"]
& \Ecal, 
\ar[l, shift right = 2, "h^*"']
\end{tikzcd}
\end{equation}
where $U$ is the usual forgetful functor. This point is surjective and localic. It follows from Propositions \ref{prop:hype3} and \ref{prop:hypejcp} that $\Ecal$ is a supercompactly generated, two-valued Grothendieck topos whose category of supercompact objects has the joint covering property.

In order to have a complete picture of what $\Ecal$ can look like, our first task is to classify the hyperconnected morphisms under $\PSh(M)$. For this task, we re-express the canonical site of principal $M$-sets in terms of relations on the monoid.

\subsection{Equivariant relations and congruences}
\label{ssec:EqRel}

We can use the powerset adjunction of Section \ref{ssec:inverse} to construct another canonical object of $\PSh(M)$. Since $M \times M$ is naturally equipped with a left $M$-action, we obtain an inverse image action on $\Pcal(M \times M)$, the set of all relations on $M$ \textit{viewed as a set}. However, this contains more relations than we need!

When we consider $M \times M$ as an object of $\PSh(M)$, its subobjects (which correspond to the relations on $M$ \textit{as an $M$-set}) are its sub-right-$M$-sets. We call such subobjects (right) \textbf{equivariant relations} on $M$, because they are the relations $r$ with the property that $(p,q) \in r$ implies $(pm,qm) \in r$ for every $m \in M$.\footnote{We have chosen to consistently use a lower case $r$ for relations to avoid a clash of notation with the right adjoint functor $R$ constructed earlier.}

\begin{lemma}
\label{lem:relations}
The sets of equivariant relations, reflexive relations, symmetric relations and transitive relations are sub-$M$-sets of $\Pcal(M\times M)$ with the inverse image action, each inheriting the ordering from $\Pcal(M \times M)$. Thus their intersection, the collection $\Rcal$ of \textbf{right congruences}, is an ordered sub-$M$-set of $\Pcal(M \times M)$.
\end{lemma}
\begin{proof}
Equivariance of a relation is unaffected by composition on the left; that is, if $r$ is an equivariant relation, then so is $k^*(r)$, so these form a sub-$M$-set.

The diagonal relation $\Delta : M \hookrightarrow M \times M$ clearly satisfies $k^*(\Delta) \supseteq \Delta$ for every $k \in M$, and the inverse image action preserves containment, so reflexive relations form a sub-$M$-set. 

Given a symmetric relation $r \hookrightarrow M \times M$ and $(m,m') \in k^*(r)$, we have $(km, km') \in r$ and hence $(km', km) \in r$ and $(m',m) \in k^*(r)$, so symmetric relations form a sub-$M$-set.

By a similar argument, if $(m,m'), (m',m'') \in k^*(r)$ then $(m,m'') \in k^*(r)$, so transitive relations form a sub-$M$-set.
\end{proof}

\begin{rmk}
Let $\Omega$ be the subobject classifier of $\PSh(M)$. The \textbf{internal power-object} of $M$, the exponential object $\Omega^M$, has an underlying set which coincides with the collection of equivariant relations. Indeed, the elements of $\Omega^M$ correspond to elements of
\[\Hom_{\PSh(M)}(M,\Omega^M) \cong \Hom_{\PSh(M)}(M \times M,\Omega) \cong \Sub_{\PSh(M)}(M\times M).\]
However, the action of $M$ on $\Omega^M$ is by inverse images \textit{only in the first component}, so $\Omega^M$ does not coincide with the first sub-$M$-set of $\Pcal(M \times M)$ described in Lemma \ref{lem:relations}. In fact, for $M$ a non-trivial monoid, the subsets of $\Omega^M$ on the reflexive, symmetric or transitive relations are typically not even sub-$M$-sets.
\end{rmk}

Now observe that the relations $\rfrak_x$ of Corollary \ref{crly:Rnx} are always right congruences, so we can examine their behaviour as elements of $\Rcal$, just as we considered the behaviour of necessary clopens as elements of $\Pcal(M)$ in the last section.

\begin{lemma}
\label{lem:rinrfrak}
Let $r \in \Rcal$ and $p \in M$. Then, as elements of $\Rcal$ with the restriction of the action and ordering from $\Pcal(M \times M)$, we have $r \subseteq \rfrak_r$ and $p^*(\rfrak_r) = \rfrak_{p^*(r)}$. That is, $\rfrak_{(-)}$ is an order-increasing $M$-set endomorphism of $\Rcal$.
\end{lemma}
\begin{proof}
We must show that $(x,y) \in r$ implies that $x^*(r) = y^*(r)$. Indeed, since $r$ is equivariant, for any $(p,q) \in M \times M$, $(xp,yp)$ and $(xq,yq)$ are in $r$, whence $(xp,xq) \in r$ if and only if $(yp,yq) \in r$, which gives the desired equality. Both $p^*(\rfrak_r)$ and $\rfrak_{p^*(r)}$ are equal to the set $\{(x,y) \in M \times M \mid x^*(p^*(r)) = y^*(p^*(r))\}$.
\end{proof}

\begin{rmk}
\label{rmk:order}
While $\rfrak$ may be order-increasing, there is no reason for it to be order-preserving: $r \subseteq s$ does not imply $\rfrak_r \subseteq \rfrak_s$ in general, since given $(m,n)$ such that $(mu,mv) \in r$ if and only if $(nu,nv) \in r$, we could still have $(mu,mv) \in s$ with $(nu,nv) \notin s$ if the former pair of elements are related in $s$ but not $r$. However, by expanding the definitions we at least find that $\rfrak_{r \cap s} \supseteq \rfrak_r \cap \rfrak_s$, in analogy with \eqref{eq:intersection}.
\end{rmk}

\begin{rmk}
\label{rmk:tauhat}
Suppose $M$ is a monoid equipped with a topology $\tau$, and $V,R$ are as in Proposition \ref{prop:hyper}. Consider $\Rcal$ equipped with the inverse image action. In light of Theorem \ref{thm:tau}, we might consider the continuous right congruences, which is to say those lying in $\Tcal:= VR(\Rcal)$, and construct a topology $\hat{\tau}$ on $M$ by taking the equivalence classes with respect to congruences lying in $\Tcal$ as a base of clopen sets. We might then wonder if $\Cont(M,\tau) \simeq \Cont(M,\hat{\tau})$; this turns out not to be the case in general.

It follows from Remark \ref{rmk:order} that the base of $\hat{\tau}$ is closed under intersections. We can also show that $\tilde{\tau} \subseteq \hat{\tau}$. Let $A \in T$ be a clopen set in $\tilde{\tau}$ and consider the congruence $\rfrak_A$, which we know to be open in $\tau \times \tau$. Then the inclusion $\rfrak_A \subseteq \rfrak_{\rfrak_A}$ from Lemma \ref{lem:rinrfrak} ensures that the latter is also in $\tau \times \tau$, and similarly stability under the inverse image action is guaranteed, so the equivalence classes $\Ical_A^p$ (and by extension $A$) are open in $\hat{\tau}$.

However, there is no reason that the opposite inclusion should hold, or even that $\hat{\tau}$ should be contained in $\tau$, since the congruence classes of $r$ need not be in $\tau$ when those of $\rfrak_r$ are. Indeed, consider a monoid $M$ with two distinct right-absorbing elements $x,y$, so that $xm = x$ and $ym = y$ for all $m \in M$. Let $\tau$ be the topology on $M$ generated by asserting that every singleton except $\{x\}$ and $\{y\}$ is open, and also $\{x,y\}$ is open. Then the diagonal relation $\Delta: M \to M \times M$ has $\rfrak_{\Delta} = \{(p,q) \mid p^*(\Delta) = q^*(\Delta)\}$, which is open in $\tau \times \tau$ since $x^*(\Delta) = M \times M = y^*(\Delta)$, so $(x,y)$ and $(y,x) \in \rfrak_{\Delta}$. As such, $\hat{\tau}$ is the discrete topology in this case. In particular, while continuity with respect to $\tau$ requires $x$ and $y$ to act identically, continuity with respect to $\hat{\tau}$ does not, so $\Cont(M,\tau) \not\simeq \Cont(M,\hat{\tau})$.

Moreover, the construction of $\hat{\tau}$ is not obviously idempotent, and we have not even been able to prove that $(M,\hat{\tau})$ is a topological monoid (recall Proposition \ref{prop:ctsx}), let alone that $\hat{\tau}$ is an action topology.
\end{rmk}

Instead, the reason we introduced the object of right congruences was to provide a canonical indexing of principal $M$-sets.

\begin{lemma}
\label{lem:prel}
The quotients of $M$ in $\PSh(M)$ correspond precisely to the right congruences on $M$ (which are internal equivalence relations on $M$ in this topos).
\end{lemma}
\begin{proof}
Any right congruence $r$ on $M$ gives a quotient $M \too M/r$. Conversely, given a quotient $q:M \too N$, let $r := \rfrak_{q(1)}$; then we clearly have $N \cong M/r$, and this operation is an inverse to the preceding one by inspection.
\end{proof}

Note that we specifically refer to quotients of $M$ in Lemma \ref{lem:prel}, rather than principal $M$-sets, since distinct congruences can give isomorphic principal $M$-sets: each generator of a principal $M$-set presents it as a quotient in a distinct way. Nonetheless, we do get all principal $M$-sets at least once in this way.

Our next task is to recover the categorical structure on these objects. For a start, the natural ordering on the collection of congruences $\Rcal$ is reflected in the subcategory $\Ccal_s$ of supercompact objects of $\Ecal$. Indeed, if $r \subseteq r'$ there is a corresponding quotient map $M/r \too M/r'$.

\begin{lemma}
\label{lem:factor}
Let $\Ccal$ be the full subcategory of $\PSh(M)$ on objects of the form $M/r$. Then any morphism $g: M/r_1 \to M/r_2$ in $\Ccal$ factors uniquely as a quotient map of the form described above followed by an inclusion of the form $M/m^*(r_2) \hookrightarrow M/r_2$.
\end{lemma}
\begin{proof}
Consider the canonical generator $[1]$ of $M/r_1$. Let $m$ be any representative of $g([1])$. Then the image part of $g$ is precisely the inclusion described in the statement of the Lemma, and the factoring quotient map is of the desired form.

For uniqueness, note that epimorphisms are orthogonal to monomorphisms in $\Ccal$ by Theorem \ref{thm:Cs}.3, so any such factorization has an isomorphic intermediate object. But the isomorphism commutes with the quotient maps, thus preserving the canonical generator. Hence the corresponding congruences are equal.
\end{proof}

As such, we extend the partial order $\Rcal$ to a category $\underline{\Rcal}$ as follows.

\begin{dfn}
The objects of $\underline{\Rcal}$ are the right congruences in $\Rcal$. A morphism $r_1 \to r_2$ is an equivalence class $[m]$ of $r_2$ such that $r_1 \subseteq m^*(r_2)$, and composition is given by multiplication in $M$, which is easily seen to be compatible with the containment condition.
\end{dfn}

That this is well-defined follows from Lemma \ref{lem:rinrfrak}, since $(m,m') \in r$ implies $m^*(r) = {m'}^*(r)$. By inspection of Lemma \ref{lem:factor}, the resulting category is isomorphic to the full subcategory $\Ccal$ appearing there, and \textit{we identify $\underline{\Rcal}$ with that category}. In particular, since that category is equivalent to the category $\Ccal_s$ of all principal $M$-sets, we have the following.

\begin{crly}
\label{crly:Rsite}
The topos $\PSh(M)$ is equivalent to $\Sh(\underline{\Rcal},J_r)$. 
\end{crly}

It is worth noting that this expression for $\PSh(M)$ is essentially independent of the representing monoid $M$, since for any other choice of $M$, we obtain the same category $\underline{\Rcal}$ up to equivalence (being equivalent to the category $\Ccal_s$ of supercompact objects of $\PSh(M)$). We adapt the presentation of Corollary \ref{crly:Rsite} in order to obtain a canonical site for a topos $\Ecal$ admitting a hyperconnected morphism from $\PSh(M)$.

\begin{prop}
\label{prop:prince2}
Let $h : \PSh(M) \to \Ecal$ be a hyperconnected geometric morphism. Let $\underline{\Rcal}_h$ be the full subcategory of $\underline{\Rcal}$ on the right congruences $r$ such that the principal $M$-set $M/r$ lies in $\Ecal$ (up to isomorphism). Then $\underline{\Rcal}_h$ is non-empty and closed under subobjects and quotients, has the joint covering property, and $\Ecal \simeq \Sh(\underline{\Rcal}_h,J_r)$.

Conversely, non-empty subcategories of $\underline{\Rcal}$ which are closed under subobjects, quotients and joint covers correspond bijectively with the (equivalence classes of) hyperconnected morphisms under $\PSh(M)$.
\end{prop}
\begin{proof}
Clearly $\underline{\Rcal}_h$ always contains the maximal equivalence relation, since $\Ecal$ contains the terminal object, while being closed under quotients and subobjects is a consequence of the fact that this is true of $\Ecal$, and that monomorphisms and epimorphisms in $\underline{\Rcal}$ and $\underline{\Rcal}_h$ coincide with those in $\PSh(M)$ and $\Ecal$ by Theorem \ref{thm:Cs} (using the fact that $\underline{\Rcal}$ and $\underline{\Rcal}_h$ can be identified up to equivalence with the respective categories of supercompact objects).

Similarly, using the proof of Proposition \ref{prop:hypejcp} we conclude that $\underline{\Rcal}_h$ must inherit joint covers from $\underline{\Rcal}$.

By Proposition \ref{prop:correspond}, we recover $\Ecal \simeq \Sh(\underline{\Rcal}_h,J_r)$, as claimed.

Conversely, given a subcategory $\underline{\Rcal}'$ of $\underline{\Rcal}$ satisfying the given conditions, consider the full subcategory $\Ecal$ of $\PSh(M)$ on those objects which are colimits of objects in $\underline{\Rcal}'$.

$\underline{\Rcal}'$ contains the terminal object, so $\Ecal$ also does. A product of $M$-sets in $\Ecal$ has elements generating $M$-sets corresponding to (quotients of) joint covers in $\underline{\Rcal}'$, and an equalizer in $\PSh(M)$ of $M$-sets in $\Ecal$ is in particular a sub-$M$-set of one in $\Ecal$, so is covered by principal $M$-sets in $\underline{\Rcal}'$. As such, $\Ecal$ is closed under finite limits, the embedding of $\Ecal$ into $\PSh(M)$ is a left exact coreflection, and this embedding is the inverse image functor of a hyperconnected geometric morphism $\PSh(M) \to \Ecal$, as required.
\end{proof}

We may use the site resulting from Proposition \ref{prop:prince2} in the special case when $\Ecal = \Cont(M,\tau)$ to recover a small (rather than merely essentially small) site for $\Cont(M,\tau)$. Given a topological monoid $(M,\tau)$, we write $\underline{\Rcal}_{\tau}$ for $\underline{\Rcal}_h$, where $h:\PSh(M) \to \Cont(M,\tau)$ is the canonical hyperconnected morphism. We shall put this to use in Proposition \ref{prop:densegroup} below.
\begin{schl}
\label{schl:Morita2}
Given a topological monoid $(M,\tau)$, the category $\underline{\Rcal}_{\tau}$ is equivalent to the category $\Ccal_s$ of supercompact objects in $\Cont(M,\tau)$. In particular, another topological monoid $(M',\tau')$ is Morita equivalent to $(M,\tau)$ if and only if $\underline{\Rcal}_{\tau} \simeq \underline{\Rcal}_{\tau'}$.
\end{schl}
\begin{proof}
We see from the proof of Proposition \ref{prop:prince2} that the observations leading up to Corollary \ref{crly:Rsite} also apply in $\Cont(M,\tau)$. Thus we have the desired result, and the Morita equivalence statement follows by Corollary \ref{crly:Morita}.
\end{proof}

Note that we can restrict the ordering on $\Rcal$ to the objects of $\underline{\Rcal}_h$ or $\underline{\Rcal}'$, or equivalently recover that ordering by considering only morphisms indexed by $1 \in M$. We denote the resulting posets by $\Rcal_h$ and $\Rcal'$ respectively, extending this convention where needed.

\begin{rmk}
\label{rmk:Mequivariant}
In Proposition \ref{prop:prince2}, if we instead consider only the ordered sets, we can characterize the sub-poset $\Rcal_h$ as an \textit{$M$-equivariant filter} in $\Rcal$: a subset which is non-empty, upward closed, downward directed and closed under the inverse image action. These conditions are easier to verify in practice, so we use them occasionally in examples to follow.
\end{rmk}

\begin{lemma}
\label{lem:collect}
Let $M$ be a topological monoid, $\underline{\Rcal}$ the category of right congruences on $M$, and $\underline{\Rcal}'$ a subcategory of $\underline{\Rcal}$ satisfying the conditions of Proposition \ref{prop:prince2}. Given a topology $\tau$ on $M$, $M/r$ is continuous with respect to $\tau$ for every $r \in \underline{\Rcal}'$ if and only if every $r$ is open in $\tau \times \tau$.
\end{lemma}
\begin{proof}
For $M/r$ to be continuous with respect to $\tau$, it is certainly necessary that $r$ be open, since $r = \rfrak_{[1]}$. Conversely, given $[m] \in M/r$, $\rfrak_{[m]} = m^*(r) \in \underline{\Rcal}'$, so if all of the right congruences in $\underline{\Rcal}'$ are open then all elements of $M/r$ are continuous.
\end{proof}

In light of Lemma \ref{lem:collect}, we call the objects of $\underline{\Rcal}_h$ the \textbf{open congruences of $h$}. This name will take on further significance in Section \ref{ssec:complete} below.

\begin{xmpl}
\label{xmpl:grprln}
Suppose $M$ is a group. Then the right congruences on $M$ are precisely the partitions of $M$ into right cosets of a subgroup of $M$. As such, we can identify the objects of $\underline{\Rcal}$ with these subgroups, and form a category of \textit{open subgroups} corresponding to a given hyperconnected geometric morphism instead.
\end{xmpl}

We are now in a position to show that not every topos admitting a hyperconnected morphism from $\PSh(M)$ is of the form $\Cont(M,\tau)$ for some topology $\tau$ on $M$.

\begin{xmpl}
\label{xmpl:N+}
Consider the monoid $\Nbb$ of non-negative integers with addition. It is easily shown that the proper principal $\Nbb$-sets are indexed by pairs $(a,b)$ of non-negative integers with $b \geq 1$, consisting of the set $\{0, 1, \dotsc, a + b - 1\}$ acted on by addition such that sums greater than $a + b - 1$ are reduced modulo $b$ into the interval $[a, a + b - 1]$; we write $N_{a,b}$ for this $M$-set; the reader should recognize the $M$-sets $N_{a,1}$ of Example \ref{xmpl:idemclosed} amongst these.

We have an epimorphism $N_{a,b} \too N_{a',b'}$ if and only if $a' \leq a$ and $b'\mid b$, and a monomorphism $N_{a,b} \hookrightarrow N_{a',b'}$ if and only if $a \leq a'$ and $b = b'$ (so all monos split). The joint cover of $N_{a,b}$ and $N_{a',b'}$ is $N_{\max \{a,a'\},\mathrm{lcm}(b,b')}$. Note that $N_{a,b}$ is always finite, so that $\Nbb$ itself is the only infinite principal $\Nbb$-set.

By the above, the collection of right congruences corresponding to finite $\Nbb$-sets is therefore an upward-closed, downward-directed set closed under the inverse image action. By Remark \ref{rmk:Mequivariant}, there is a corresponding hyperconnected morphism $h : \PSh(\Nbb) \to \Ecal$, where the subcategory $\Ecal$ of $\PSh(\Nbb)$ consists of the $\Nbb$-sets all of whose elements generate finite subsets under the action. This could be compared to the adjunction between the category of abelian groups and the category of torsion abelian groups.

For $a \geq 0$, the necessary clopens of $N_{a,1}$ are $\{\{0\},\{1\}, \dotsc, \{a - 1\}, [a,1)\}$. Thus the only topology on $M$ making all of the $N_{a,1}$ continuous is the discrete topology, but $\Ecal \not\simeq \PSh(\Nbb)$: not only is the stated hyperconnected morphism not an equivalence, but it is clear that the subcategories of supercompact objects of these toposes cannot be equivalent, since no supercompact object of $\Ecal$ admits epimorphisms to all of the others. By computing the action topologies corresponding to coarser topologies on $\Nbb$, we can verify by a similar argument that $\Ecal \not\simeq \Cont(\Nbb,\tau)$ for any topology $\tau$, as claimed.

On the other hand, since $\Nbb$ is a commutative monoid, each $N_{a,b}$ is canonically a monoid, and it is easy to see that $\Ecal$ is the topos of continuous actions of the \textbf{profinite completion of $\Nbb$}, which is the profinite monoid obtained as the inverse limit of the $N_{a,b}$ along the epimorphic maps between them.
\end{xmpl}

The argument in Example \ref{xmpl:N+} is somewhat overzealous; even if we had found that $\Ecal \simeq \Cont(\Nbb,\tau)$ for some $\tau$, the important conclusion is that this hyperconnected morphism fails to \textit{express} $\Ecal$ as $\Cont(\Nbb,\tau)$. By extending the argument of Theorem \ref{thm:tau}, we shall show in Theorem \ref{thm:factor} that for any hyperconnected morphism out of $\PSh(M)$ there is a canonical coarsest topology $\tau$ such that this morphism factorizes through the hyperconnected morphism $\PSh(M) \to \Cont(M,\tau)$; in Example \ref{xmpl:N+}, this topology happens to be the discrete topology.

However, the fact that we were able to find a representing topological monoid for $\Ecal$ in the end turns out to be a general fact, as we might have hoped, and the latter part of Example \ref{xmpl:N+} suggests how it can be constructed. We perform this construction in general in Proposition \ref{prop:lim} and conclude the proof that it gives a representing monoid with Theorem \ref{thm:characterization}.

\subsection{Endomorphisms of the canonical point}
\label{ssec:complete}

Let $p$ be the point of the topos $\Ecal$ factorized in \eqref{eqn:p}. Consider the endomorphisms of $p$; that is, the monoid of natural transformations $\alpha: p^* \Rightarrow p^*$. Since we know $\Ecal$ is supercompactly generated, any such endomorphism is determined by its components on the subcategory $\underline{\Rcal}_h$ of $\Ecal$. That is, each $\alpha$ consists of a collection of endomorphisms $\alpha_r: p^*(M/r) \to p^*(M/r)$ for $r \in \underline{\Rcal}_h$ satisfying naturality conditions relative to morphisms in $\underline{\Rcal}_h$.

The following argument replicates the proof of \cite[Theorem 5.7]{ATGT}, but without the need for a subsequent restriction to automorphisms. For consistency, we consider the opposite of the endomorphism monoid, acting on the right, so that composition is from left to right (this precludes the need to dualize subsequently).

\begin{prop}
\label{prop:lim}
Let $h:\PSh(M) \to \Ecal$ be a hyperconnected geometric morphism, let $\underline{\Rcal}_{h}$ be the corresponding category of right congruences on $M$ described in Proposition \ref{prop:prince2}, and let $\Rcal_h$ be the underlying order. Then $\End(p^*)\op$ can be identified with the limit
\begin{equation}
\label{eq:limit}
L := \varprojlim_{r \in \Rcal_{h}}U(M/r)
\end{equation}
in $\Set$. Explicitly, the elements are tuples $\alpha = ([a_r])_{r \in \Rcal_h}$ with each $[a_r] \in M/r$, represented by $a_r \in M$, such that whenever $r \subseteq r'$, $[a_r] = [a_{r'}]$ in $M/r'$. The composite of $\alpha$ and $\beta = ([b_r])_{r \in \Rcal_h}$ is $\alpha\beta = ([a_r b_{a_r^*(r)}])_{r \in \Rcal_h}$. 
\end{prop}
\begin{proof}
Consider the factorization of morphisms in $\underline{\Rcal}_h$ from  Lemma \ref{lem:factor}, which corresponds to the factorization of a morphism $m:r_1 \to r_2$ as,
\[ \begin{tikzcd} 
r_1 \ar[r, "{[1]}", two heads] & m^*(r_2) \ar[r, "{[m]}", hook] & r_2.
\end{tikzcd} \]
The naturality conditions for a natural transformation $p^*|_{\underline{\Rcal}_h} \Rightarrow p^*|_{\underline{\Rcal}_h}$ can be reduced to naturality along factors of the form $[1]: r \too r'$ and $[m]: m^*(r) \hookrightarrow r$.

First, consider $[m]: m^*(r) \hookrightarrow r$, corresponding to the inclusion $M/m^*(r) \hookrightarrow M/r$. Naturality forces $([m])\alpha_{r}$ to be the image of $([1])\alpha_{m^*(r)}$ under this inclusion. Thus, any endomorphism $\alpha$ of $p^*$ is determined by its values on the equivalence classes represented by the identity of $M$. Write $[a_r] \in U(M/r)$ for $([1])\alpha_r$; then $([m])\alpha_r = [m a_{m^*(r)}]$.

Now consider $[1]: r \too r'$, corresponding to the quotient map $M/r \too M/r'$. This forces $[a_r] \mapsto [a_{r'}]$. Thus we may identify each endomorphism $\alpha$ with an element of the stated limit; these observations also determine the composition.

Conversely, any element of the limit defines an endomorphism, since we have shown that collections satisfying these conditions are guaranteed to be natural.
\end{proof}

Let $L$ be the monoid defined in Proposition \ref{prop:lim}. For each $r \in \underline{\Rcal}_h$, we write $\pi_{r}:L \to U(M/r)$ for the universal projection map. Being expressed as a limit of discrete sets, $L$ is canonically equipped with a prodiscrete topology $\rho$ making it a topological monoid. The basic opens for this topology form a genuine base of open sets:

\begin{lemma}
\label{lem:basics}
The collection of open sets of the form $\pi_r^{-1}(\{[m]\})$ generating $\rho$ are closed under (non-empty) finite intersection.
\end{lemma}
\begin{proof}
The empty intersection is all of $L = \pi^{-1}_{L \times L}(\{[1]\})$. Thus it suffices to consider binary intersections.

Given opens $\pi_{r_1}^{-1}(\{[m_1]\})$ and $\pi_{r_2}^{-1}(\{[m_2]\})$ having non-empty intersection, let $r = r_1 \cap r_2$. Since the intersection of the open sets in non-empty, there is some element $\alpha \in L$ whose component $[a_r]$ at $M/r$ maps under the canonical quotient maps to $[m_1], [m_2]$ respectively. Taking $[m]=[a_r]$ provides the desired expression for the intersection.
\end{proof}

Thus we need only concern ourselves with basic opens when checking continuity.

\begin{lemma}
\label{lem:Mdense}
Let $M$ and $L$ be as above. There is a canonical monoid homomorphism $u: M \to L$ sending $m$ to the endomorphism $\alpha_m$ represented at every principal $M$-set $M/r$ by $[m]$. Its image is dense in $(L,\rho)$.
\end{lemma}
\begin{proof}
The stated definition does give a well-defined endomorphism for each $m$, since the canonical quotient morphisms preserve representatives by definition. To see that this is a monoid homomorphism, simply observe from the expression for composition in Proposition \ref{prop:lim} that the component of $\alpha_m \cdot \alpha_{m'}$ at $M/r$ is always represented by $mm'$.

Now observe that for each $m \in M$, the basic opens of the form $\pi_r^{-1}(\{[m]\})$ contain $\alpha_m$, so the image of $M$ intersects every open set by Lemma \ref{lem:basics} and has dense image, as claimed.
\end{proof}

The monoid homomorphism $u$ is the key to demonstrating that we can in fact present $\Ecal$ as $\Cont(L,\rho)$. Indeed, note that it induces a geometric morphism $q: \PSh(M) \to \Cont(L,\rho)$, whose inverse image is the restriction of $L$-actions along the homomorphism $\alpha_{(-)}$. We will explore the geometric morphisms induced by general continuous (monoid and) semigroup homomorphisms in Section \ref{sec:homomorphism}.

\begin{prop}
\label{prop:representation}
Let $h:\PSh(M) \to \Ecal$ be hyperconnected, and let $(L,\rho)$ be the endomorphism monoid constructed above. The geometric morphism $q:\PSh(M) \to \Cont(L,\rho)$ induced by the continuous dense homomorphism $u: M \to L$ of Lemma \ref{lem:Mdense} is hyperconnected, and the principal $M$-sets lying in $\Cont(L,\rho)$ are precisely those lying in the topos $\Ecal$ from which $L$ was defined. That is, $\Ecal \simeq \Cont(L,\rho)$.
\end{prop}
\begin{proof}
Let $\underline{\Rcal}_h$ be the category of right congruences defined in Section \ref{ssec:EqRel}. There is a canonical right action of $L$ on each $M/r$ in $\Rcal_h$: if $\alpha = ([a_r])_{r \in \Rcal_h}$, then $\alpha$ acts on $[m] \in M/r'$ by `projected multiplication', sending it to $[ma_{m^*(r')}]$.

We verify that this action is well-defined and continuous with respect to $\rho$. If $\beta = ([b_r])_{r \in \Rcal_h}$ is another element of $L$, then acting by $\alpha$ and then $\beta$ gives $[m] \mapsto [ma_{m^*(r')}] \mapsto [ma_{m^*(r')}b_{(ma_{m^*(r')})^*(r')}]$, which indeed is equal to the action of $\alpha\beta$, so $M/r'$ is a right $L$-set.

For continuity of the action on $M/r'$ (with canonical generator $n$, say), consider
\begin{align*}
\Ical_{n}^\alpha &= \{\beta = ([b_r])_{r \in \Rcal_h} \mid [a_{r'}] = [b_{r'}]\}\\
&= \pi_{r'}^{-1}(\{[a_{r'}]\}).
\end{align*}
By definition of the induced prodiscrete topology, this is open in $L$, as required. By Lemma \ref{lem:collect}, since $r'$ was arbitrary, this is sufficient to conclude that the actions of $L$ on the principal $M$-sets lying in $\Ecal$ are continuous with respect to $\rho$. Moreover, it is clear that these are principal $L$-sets, since we can obtain all of $M/r'$ by applying the elements $\alpha_m$ to the generator of $M/r'$.

Now $q^*$ returns the $L$-sets defined above to the principal $M$-sets they extended. Thus it remains only to show that $q^*$ is full and faithful and then that any principal $(L,\rho)$-set is a quotient of (and hence equal to) one of the principal $M$-sets, since arbitrary colimits are computed in both $\PSh(M)$ and $\Cont(L,\rho)$ at the level of underlying sets, whence considering the principal objects is sufficient.

It is immediate that $q^*$ is faithful since the underlying function of an $L$-set homomorphism is unaffected by applying $q^*$. To show that any $M$-set homomorphism $g:P \to Q$ between $(L,\rho)$-sets $P$ and $Q$ is an $(L,\rho)$-set homomorphism, consider an element $p \in P$ and $\alpha \in L$. We have that $\alpha_m \in \Ical_p^\alpha$ for some $m \in M$ by density of $M$. Thus $g(p \cdot \alpha) = g(p \cdot \alpha_m) = g(p) \cdot \alpha_m = g(p) \cdot \alpha$, where the final equality is by continuity of the action of $L$ on $Q$. Thus $q^*$ is full, as expected.

Given a quotient $L \too K$ in $\PSh(L)$ with canonical generator $k$, observe that to be continuous with respect to $\rho$ it must be that for each $\alpha \in L$, $\Ical_{k}^\alpha$ is open in $\rho$, so must contain an open set (and hence a basic open set) around $\alpha$, say $\pi_{r'}^{-1}(\{[a_{r'}]\})$. But by the equation above, this is precisely $\Ical_n^\alpha$, where $n$ is the canonical generator of $M/r'$. Thus $K$ must be a quotient of $M/r'$ as an $L$-set, and hence as an $M$-set, as required.
\end{proof}

We can summarize the results obtained so far in this section with the following theorem.
\begin{thm}
\label{thm:characterization}
A topos is equivalent to one of the form $\Cont(M,\tau)$ if and only if it has a surjective point which factors as an essential surjection followed by a hyperconnected geometric morphism. Moreover, every topological monoid is canonically Morita equivalent to a monoid endowed with a prodiscrete topology.
\end{thm}

\begin{rmk}
We have had to be careful with the language used in Theorem \ref{thm:characterization}: $(L,\rho)$ is \textit{not} in general a prodiscrete monoid, since the principal $M$-sets are not in general equipped with the structure of monoids. Even so, $\rho$ is always a `good' topology in the following sense.
\end{rmk}

\begin{prop}
\label{prop:Lpowder}
Given a monoid $M$ and a hyperconnected geometric morphism $h:\PSh(M) \to \Ecal$, the corresponding topological monoid $(L,\rho)$ is a powder monoid.
\end{prop}
\begin{proof}
Clearly the basic opens at fixed $r \in \Rcal_h$ partition $L$, so they are clopen; it therefore suffices to check that the basic open sets are continuous elements of $\Pcal(L)$ in $\PSh(L)$.

Using the established notation for the components of $\alpha$ and $\beta$,
\begin{align*}
\alpha^*(\pi_r^{-1}(\{[m]\}))
&= \{\beta \in L\op \mid [a_r b_{a_r^*(r)}] = [m] \in M/r \}\\
&= \{\beta \in L\op \mid b_{a_r^*(r)} \in a_r^*(\Ical_n^m) \}\\
&= \begin{cases}
\pi_{a_r^*(r)}^{-1}(\{[m']\}) & \text{if } \exists m' \in a_r^*(\Ical_n^m)\\
\emptyset & \text{otherwise.}
\end{cases}
\end{align*}
This shows that the inverse image action preserves basic opens, and consequently we need only consider:
\begin{align*}
\Ical_{\pi_r^{-1}(\{[m]\})}^{\alpha}
&= \{\beta \in L\op \mid \beta^*(\pi_r^{-1}(\{[m]\})) = \alpha^*(\pi_r^{-1}(\{[m]\})) \}\\
& \supseteq \pi_r^{-1}(\{[a_r]\}),
\end{align*}
by inspection of the fact that $\alpha^*(\pi_r^{-1}(\{[m]\}))$ depends only on $a_r$. Thus the basic opens are continuous elements; since an action topology has at most as many opens as the topology it is derived from, it follows that $\tilde{\rho}=\rho$ as claimed.

To see that $(L,\rho)$ is Hausdorff we simply note that two points are equal if and only if they are equal in every component (by the definition of $L$ as a limit) and points which differ in any component are separated by basic opens from the corresponding projection maps.
\end{proof}

\begin{schl}
\label{schl:completion}
The topological monoid $(L,\rho)$ constructed in Proposition \ref{prop:lim} from a hyperconnected geometric morphism $h: \PSh(M) \to \Ecal$ depends only on the point $\Set \to \PSh(M) \to \Ecal$ of $\Ecal$ which, after composing with the equivalence $\Ecal \simeq \Cont(L,\rho)$, is naturally isomorphic to the canonical point of the latter topos.
\end{schl}
\begin{proof}
While we can identify $L$ with the limit described in Proposition \ref{prop:lim}, it remains the opposite of the monoid of endomorphisms of the stated point, so is independent of $M$. Since $\rho$ is an action topology by Proposition \ref{prop:Lpowder}, it is uniquely determined by the topos $\Cont(L,\rho)$ after establishing $L$ by Theorem \ref{thm:tau}, which completes the proof.
\end{proof}

\begin{crly}
\label{crly:extend}
Suppose that $(M,\tau)$ is a powder monoid and let $(L,\rho)$ be the opposite of the topological monoid of endomorphisms of the corresponding canonical point of $\Cont(M,\tau)$. Then the monoid homomorphism $u: M \to L$ of Lemma \ref{lem:Mdense} is injective and continuous with respect to $\tau$ and $\rho$.
\end{crly}
\begin{proof}
Let $(M,\tau)$ be a powder monoid and $m, m' \in M$. Suppose $\alpha_m = \alpha_{m'}$, and let $U \in T$ with $m \in U$, whence the same is true for $\Ical_U^m$. Consider the principal sub-$M$-set of $T$ generated by $U$, say $M/r$. Consider $\pi_r^{-1}(\{[m]\})$ and $\pi_r^{-1}(\{[m']\})$; since $\alpha_m = \alpha_{m'}$, these open sets must be equal in $L$, which is to say that $[m] = [m']$ in $M/r$, and hence $m' \in U$. As such, $m,m'$ are topologically indistinguishable and hence equal in $(M,\tau)$, and so $u$ is injective as claimed.

To demonstrate continuity, consider a basic open $U' := \pi_r^{-1}(\{[a]\})$ in $L$. Then
\[u^{-1}(U') = \{m \mid \alpha_m \in U' \} = \{m \mid [m] = [a] \text{ in } M/r\} \supseteq \Ical_n^a,\]
where $n$ is the canonical generator of $M/r$. Since $M/r$ is a continuous $M$-set by assumption, $\Ical_n^a$ is open in $(M,\tau)$, as required.
\end{proof}

\begin{xmpl}
\label{xmpl:Z+}
Consider the following example, to be contrasted with Example \ref{xmpl:N+}. Let $\Zbb$ be the group of integers under addition. Taking the topology $\tau$ on $\Zbb$ in which the subgroups $n\Zbb \subseteq \Zbb$ and their cosets are open for each $n\neq 0$, we find that the continuous principal $\Zbb$-sets are precisely the finite cyclic groups $\Zbb/n\Zbb$ with $n > 0$; we thus obtain the topos $\Cont(\Zbb,\tau)$ of torsion $\Zbb$-sets. The topology $\tau$ is not discrete since every neighbourhood of $0$ is infinite, but it is nearly discrete and hence $(\Zbb,\tau)$ is a powder group. However, the monoid obtained from Proposition \ref{prop:lim} is the \textit{profinite completion} of the integers. Thus even for a powder monoid, the comparison map may fail to be an isomorphism.
\end{xmpl}

\begin{dfn}
\label{dfn:cpltmon}
In light of Corollary \ref{crly:extend} and Example \ref{xmpl:Z+}, we say that a monoid is (right) \textbf{complete} if the comparison morphism $u:(M,\tau) \to (L,\rho)$ is an isomorphism of topological monoids.
\end{dfn}

We shall see in Section \ref{ssec:monads} that complete monoids form a reflective subcategory of the ($2$-)category of monoids and also of the ($2$-)category of powder monoids; the comparison homomorphism $u$ is the unit of these reflections.

\begin{crly}
\label{crly:Mlim}
Let $(M,\tau)$ be a powder monoid and $h:\PSh(M) \to \Cont(M,\tau)$ the canonical hyperconnected morphism. Consider the poset $\Rcal_h$ as a subcategory of $\PSh(M)$. Then $M$ is complete if and only if it is the limit of $\Rcal_h$ in $\PSh(M)$.
\end{crly}
\begin{proof}
Since limits in $\PSh(M)$ are computed from their underlying sets, this follows from Proposition \ref{prop:lim} and Corollary \ref{crly:extend}, after observing that the action of $M$ expressed as the limit is respected by these morphisms by definition.
\end{proof}

\begin{rmk}
\label{rmk:completegrp}
It must be stressed that our notion of completeness for powder monoids does not quite coincide with that of completeness for groups described in Caramello's paper \cite[\S 2.3]{TGT}. This is clear from a comparison between our Proposition \ref{prop:lim} and Caramello and Lafforgue's construction in \cite[Proposition 5.7]{ATGT}, since they begin by constructing the topological monoid of endomorphisms of the canonical point as we do (which is complete in our sense), but then restrict to the subgroup of automorphisms in order to obtain the representing topological group (which is complete in their sense), and it is an important fact that this gives a genuinely different representation in general; see Example \ref{xmpl:notpro2} below.

For consistency, we say a powder group is \textbf{complete} if it is isomorphic to the topological subgroup of units of the corresponding complete monoid; this is true to Caramello and Lafforgue's terminology, and coincides with ours whenever the complete monoid happens to be a group.
\end{rmk}

\begin{xmpl}
\label{xmpl:notpro2}
The Schanuel topos is typically presented as the topos of sheaves for the atomic topology on the opposite of the category of finite sets and injective functions, $\Ecal := \Sh(\FinSet_{\mathrm{mono}}\op, J_{\mathrm{at}})$. It is the classifying topos for the theory of infinite decidable objects, so that in particular its category of $\Set$-valued points is equivalent to the category of infinite sets and injective functions. We employed it as an example of a supercompactly generated topos which is not a presheaf topos in \cite[Example 2.40]{SGT}.

Thanks to the work of Caramello, \cite[\S 6.3]{TGT}, we know that any infinite set $X$ provides, via its corresponding point, an equivalence $\Ecal \simeq \Cont(\Aut(X),\tau_{\mathrm{fin}})$, where $\tau_{\mathrm{fin}}$ is the topology generated from the base of subgroups stabilizing finite subsets. The topological group of automorphisms of $\Nbb$ considered in Example \ref{xmpl:notpro} is one such representation. In particular, \textit{all points of the Schanuel topos are of the form required by Theorem \ref{thm:characterization}}, and the corresponding complete monoids are simply $(\End_{\mathrm{mono}}(X),\tau_{\mathrm{fin}})$, where this time $\tau_{\mathrm{fin}}$ has basic open sets consisting of subsets of the form
\[\{f \in \End_{\mathrm{mono}}(X) \mid f(x_1) = y_1, \dotsc, f(x_k) = y_k\}\]
for each finite set of pairs of elements $(x_i,y_i) \in X \times X$. 

In computing these complete monoids, we have the advantage of being able to inspect the endomorphisms of objects in the category of models for the theory classified by the topos, since these can be identified with endomorphisms of the corresponding points. Computing one of these complete monoids with \eqref{eq:limit} and identifying the result with a monoid of injective endomorphisms is a more demanding, albeit instructive, exercise. Understanding the theories classified by toposes of topological monoid actions will thus lead to deeper insights into their completions (see Section \ref{ssec:TSGT} below).
\end{xmpl}

\subsection{Bases of congruences}
\label{ssec:base}

While we made abstract use of the limit expression \eqref{eq:limit} in proofs in the last subsection, in practice it will be more convenient to compute $L$ after re-indexing the limit over a smaller collection of right congruences. 

\begin{dfn}
\label{dfn:basecong}
Given a hyperconnected morphism $h:\PSh(M) \to \Ecal$, a collection of right congruences $\Rcal' \subseteq \Rcal_h$ is called a \textbf{base of open congruences for $h$} if for every $r \in \Rcal_h$ there is some $r' \in \Rcal'$ with $r' \subseteq r$. A base of open congruences is precisely an initial subcategory of the poset $\Rcal_h$, and as such we can replace \eqref{eq:limit} with
\begin{equation}
\label{eq:limit2}
L := \varprojlim_{r \in \Rcal'} U(M/r),
\end{equation}
where the morphisms are simply inclusions of congruences. Moreover, the expression for the prodiscrete topology on $L$ restricts to this re-indexing, thanks to the fact that the basic opens for this topology coming from the projection maps along the omitted relations are necessarily unions of opens coming from any initial collection of relations.
\end{dfn}

\begin{xmpl}
\label{xmpl:algbase}
Suppose $(M,\tau)$ is a topological group. Recall from \cite[following Lemma 2.1]{TGT} that an \textit{algebraic base} for $(M,\tau)$ is a neighbourhood base of the identity $\Bcal$, consisting of open subgroups, such that for any $H,K \in \Bcal$ there exists $P \in \Bcal$ with $P \subseteq H \cap K$ and for any $g \in M$, there exists $Q \in \Bcal$ with $Q \subseteq g^{-1}Hg$.

Suppose we are given an algebraic base $\Bcal$ of open subgroups of $(M,\tau)$. In accordance with Example \ref{xmpl:grprln}, we can identify the open congruences for the canonical hyperconnected morphism $h:\PSh(M) \to \Cont(M,\tau)$ with open subgroups. Every such open subgroup of $M$ must contain one belonging to $\Bcal$, whence the congruences corresponding to groups in $\Bcal$ form a base of open congruences on $(M,\tau)$. Conversely, any base of open congruences gives an algebraic base for $(M,\tau)$. 

The limit \eqref{eq:limit2} corresponds to the $\Bcal$-indexed limit expression for the monoid of endomorphisms presented in \cite[Proposition 5.7(i)]{ATGT}.
\end{xmpl}

\begin{prop}
\label{prop:algbase}
Let $h:\PSh(M) \to \Ecal$ be a hyperconnected morphism. Suppose $\Rcal' \subseteq \Rcal_h$ is a base of open congruences for $h$. Suppose further that we extend $\Rcal'$ to a subcategory of $\underline{\Rcal}' \subseteq \underline{\Rcal}_h$ such that given any $r,r' \in \underline{\Rcal}'$ with $r \subseteq m^*(r')$, $\underline{\Rcal}'$ contains a span of morphisms,
\begin{equation}
\label{eq:span}
\begin{tikzcd}
r & \ar[l, "{[1]}"', two heads] r'' \ar[r, "{[m]}"] & r'.
\end{tikzcd} 	
\end{equation}
For example, it suffices that $\underline{\Rcal}'$ be a full subcategory. Then the morphisms indexed by $[1]$ in $\underline{\Rcal}'$ form a stable class $\Tcal$ (in the sense of Definition \ref{dfn:stable}), and $(\underline{\Rcal}',J_{\Tcal})$ is a dense subsite of $(\underline{\Rcal}_h,J_r)$, which is to say that
\[ \Ecal \simeq \Sh(\underline{\Rcal}',J_{\Tcal}).\]
\end{prop}
\begin{proof}
Given $[1]:r_1 \too r$ and $[m]:r_2 \to r$ in $\underline{\Rcal}'$, stability of strict epimorphisms in $\underline{\Rcal}_h$ provides a square there,
\[ \begin{tikzcd}
r' \ar[d, "{[m']}"'] \ar[r, "{[1]}", two heads] &
r_2 \ar[d, "{[m]}"] \\
r_1 \ar[r, "{[1]}"', two heads] & r;
\end{tikzcd}\]
without loss of generality we may assume $r' \in \underline{\Rcal}'$ since any $r'$ is covered by a member of $\underline{\Rcal}'$. Then we may construct spans on the upper and left-hand sides using \eqref{eq:span} in order to produce a similar square all of whose morphisms lie in $\underline{\Rcal}'$. Moreover, the morphisms indexed by $1$ are precisely the morphisms inherited from $\underline{\Rcal}_h$ which generate covering families, whence we see that $\underline{\Rcal}'$ meets the definition of dense subsite required to apply the Comparison Lemma. This allows us to deduce the stated presentation of $\Ecal$.
\end{proof}

We can use open bases of congruences to address the question of when the completion $(L,\rho)$ of a powder monoid is (isomorphic to) a prodiscrete monoid.

\begin{crly}
\label{crly:prodisc}
Suppose that $h:\PSh(M) \to \Ecal$ is hyperconnected. The corresponding complete monoid $(L,\rho)$ is discrete if and only if there exists a base of open congruences $\Rcal' \subseteq \Rcal_h$ with $\Rcal'$ finite. More generally, $(L,\rho)$ is prodiscrete if and only if there exists a base of open congruences $\Rcal' \subseteq \Rcal_h$ where each $r \in \Rcal'$ is a two-sided congruence. In the latter case, if $M$ is also a group, or $M/r$ is a group for each $r \in \Rcal'$, then so is the resulting prodiscrete monoid.

In particular, if $M$ is finite or commutative, so that any right congruence $r$ on $M$ is also a left congruence, then the codomain of any hyperconnected morphism out of $\PSh(M)$ is equivalent to the topos of continuous actions of a finite discrete monoid or commutative prodiscrete monoid respectively.
\end{crly}
\begin{proof}
As observed in Definition \ref{dfn:basecong}, given any base of open congruences $\Rcal'$, we may express $L$ as the limit \eqref{eq:limit2} over congruences in $\Rcal'$. Note that this limit is directed. When $\Rcal'$ is finite, the induced topology is a finite product of discrete topologies, so it is discrete, and there must be an initial congruence in this case by directedness. Conversely, if $L$ is discrete, then $\Cont(L,\rho) = \PSh(L)$. By construction, $\underline{\Rcal}_h$ is equivalent to the category of quotients of $L$ in this topos, which contains a generating object, namely $L$ itself. As such, there is some relation $r^*$ in $\underline{\Rcal}_h$ such that the equivalence $\Ecal \simeq \PSh(L)$ identifies $M/r^*$ with $L$. Then $\{r^*\}$ is a finite base of open congruences for $h$, as required.  

Now suppose instead that each $r \in \Rcal'$ is also a left congruence. Then the quotients $M/r$ are naturally equipped with a multiplication operation compatible with the multiplication from $M$, and the topological monoid $(L,\rho)$ constructed in Proposition \ref{prop:lim} is their limit as discrete monoids in the category of topological monoids, hence is a prodiscrete monoid. Conversely, if $(L,\rho)$ is prodiscrete, it can be defined as a limit of its discrete quotients, which are quotients of $L$ by a two-sided congruences. The restriction of such a congruence along $u$ is also a two-sided congruence on $M$; the collection of such congruences gives the desired base of open congruences.

If $M$ is a group and $\Rcal'$ consists of two-sided congruences, then the quotients $M/r$ for $r \in \Rcal'$ are also groups, whence $(L,\rho)$ is a prodiscrete group, fulfilling the claim regarding groups.
\end{proof}

\begin{xmpl}
As an example application of Corollary \ref{crly:prodisc} on a monoid which is neither commutative nor finite, consider the (non-commutative) monoid $M$ obtained from the non-negative integers $\Nbb$ with addition by freely adjoining a left absorbing element $l$. The elements of $M$ are of the form $(1,n)$ or $(l,n)$ with $n \in \Nbb$, and multiplication is defined by $(1,m)(l,n) = (l,m)(l,n) = (l,n)$, $(l,m)(1,n) = (l,m+n)$ and $(1,m)(1,n) = (1,m+n)$.

We can define a right congruence $r$ on $M$ which identifies all elements of the form $(l,n)$ and has all other equivalence classes being singletons. Then $(1,m)^*(r) = r$ for every $m$ and $(l,m)^*(r) = M \times M$, whence the collection of all equivalence relations containing $r$ is an $M$-equivariant filter in $\Rcal$, and we have a corresponding hyperconnected morphism $\PSh(M) \to \Cont(L,\rho)$.

Since $\{r\}$ is initial in $\Rcal_h$, Corollary \ref{crly:prodisc} informs us that $\rho$ is the discrete topology. Indeed, we find that $L \cong \Nbb \cup \{\infty\}$ with extended addition, and $\Cont(L,\rho) = \PSh(L)$. It is interesting to note that, $L$ being commutative, any further hyperconnected geometric morphism lands in the topos of actions of a prodiscrete monoid, such as the topologization of $\Nbb \cup \{\infty\}$ seen in Example \ref{xmpl:idemclosed}.
\end{xmpl}

\subsection{Factorizing topologies}
\label{ssec:factor}

Having made it this far, we would be remiss not to initiate an investigation of when a hyperconnected geometric morphism $\PSh(M) \to \Ecal$ actually \textit{does} express $\Ecal$ as $\Cont(M,\tau)$ for some topology $\tau$ on $M$. Our first result in this direction is a strengthening of Theorem \ref{thm:tau}.
\begin{thm}
\label{thm:factor}
Let $h: \PSh(M)\to \Ecal$ be a hyperconnected geometric morphism. Consider $\Pcal(M) \in \PSh(M)$ with the inverse image action corresponding to left multiplication. Write $T := h^*h_*(\Pcal(M)) \hookrightarrow \Pcal(M)$. Then (the underlying set of) $T$ is a base of clopen sets for the coarsest topology $\tau_h$ on $M$ such that $h$ factors through the canonical morphism $\PSh(M) \to \Cont(M,\tau_h)$. That is, toposes constructed from topologies on $M$ are universal amongst toposes admitting a hyperconnected morphism from $\PSh(M)$.
\end{thm}
\begin{proof}
By assumption, the counit at $\Pcal(M)$ is monic, so $T$ is indeed a subobject of $\Pcal(M)$. Further, $T$ must be a sub-Boolean-algebra of $\Pcal(M)$, so it is closed under complementation and finite intersections. Let $\tau_h$ be its closure in $\Pcal(M)$ under arbitrary unions. By Proposition \ref{prop:prince2}, to show that all $M$-sets lying in $\Ecal$ also lie in $\Cont(M,\tau_h)$, it suffices to show that the principal $M$-sets belonging to $\Ecal$ are continuous with respect to $\tau_h$.

Suppose $r \in \underline{\Rcal}_h$, and let $p \in M$. We must show that, for $[1] \in M/r$, $\Ical_{[1]}^p \in U(T)$, or equivalently that the corresponding morphism $\ulcorner \Ical_{[1]}^{p} \urcorner: M \to \Pcal(M)$ factors through the inclusion $T \hookrightarrow \Pcal(M)$. We define a morphism $i^p : M/r \to \Pcal(M)$ by $[q] \mapsto q^*(\Ical_{[1]}^p)$. To see that this is well-defined, note that if $(q, q') \in r$, then
\[q^*(\Ical_{[1]}^p)= \{m \in M \mid [qm] = [p]\} = \{m \in M \mid [q'm] = [p]\} = {q'}^*(\Ical_{[1]}^p).\]
Hence the following diagram commutes:
\begin{equation}
\begin{tikzcd}
M \ar[r, "\ulcorner \Ical_{[1]}^{p} \urcorner"] \ar[d, two heads] & \Pcal(M) \\
M/r \ar[ur,"i^p"] \ar[r,dotted] & T \ar[u,hook]
\end{tikzcd}
\label{eq:factorize}
\end{equation}
Since $\Ecal$ is coreflective and $M/r \in \Ecal$, $i^p$ must further factor through the inclusion $T \hookrightarrow \Pcal(M)$. Thus we are done: $\Ical_{[1]}^{p} \in T$, as required.

It follows that $h$ factors through the morphism $\PSh(M) \to \Cont(M, \tau_h)$ as claimed, and that $\tau_h$ is an action topology. Moreover, if $h$ factors through $\Cont(M,\tau')$ for any other topology $\tau'$, then $T$ must be continuous with respect to $\tau'$, and hence $\tau_h \subseteq \tau'$, as claimed.
\end{proof}

Note that the fact that $\Pcal(M)$ is an internal Boolean algebra ensures that for any object $N$, the set $\Hom_{\PSh(M)}(N,\Pcal(M))$ inherits the structure of a Boolean algebra. When $N = M/r$ for some $r \in \underline{\Rcal}_h$, a morphism $a: M/r \to \Pcal(M)$ is determined by the image of the generator $[1]$, and for any $(p,p') \in r$ must satisfy
\[ p^*(a([1])) = a([p]) = a([p']) = {p'}^*(a([1])).\]
In particular, by considering whether $1 \in p^*(a([1]))$, we see that $p \in a([1])$ if and only if $\Ical_{[1]}^p \subseteq a([1])$. Thus the morphisms $i^p$ in the proof above are actually \textit{atoms} in the Boolean algebra $\Hom_{\PSh(M)}(M/r,\Pcal(M))$, since they have precisely two lower bounds, themselves and the trivial map sending every element of $M/r$ to $\emptyset \in \Pcal(M)$.

\begin{schl}
\label{schl:factors}
A hyperconnected morphism $h:\PSh(M)\to \Ecal$ expresses the topos $\Ecal$ as $\Cont(M,\tau)$ for some topology $\tau$ on $M$ if and only if whenever the image of each atom in  $\Hom_{\PSh(M)}(M/r,\Pcal(M))$ lies in $\Ecal$, we have $M/r$ in $\Ecal$.
\end{schl}
\begin{proof}
Reconstruct the diagram \eqref{eq:factorize} for a right congruence $r$ which is open with respect to $\tau \times \tau$, and let $M/r_p$ be the image of $i^p$:
\begin{equation}
\label{eq:image}
\begin{tikzcd}
M \ar[r, "\ulcorner \Ical_{[1]}^{p} \urcorner"] \ar[d, two heads] & \Pcal(M) \\
M/r \ar[ur,"i^p"] \ar[r, two heads] & M/r_p \ar[u,hook].
\end{tikzcd}
\end{equation}
Since $\Ecal$ is closed under quotients, if $M/r$ is in $\Ecal$ then so are the $M/r_p$ for every $p \in M$.

Conversely, since $\Ecal$ is closed under subobjects, the inclusion $M/r_p \hookrightarrow \Pcal(M)$ factors through $T$ if and only if $M/r_p$ lies in $\Ecal$; that is, $\Ical_{[1]}^p$ is in the topology induced by $h$ if and only if $M/r_p$ lies in $\Ecal$. Thus if $M/r_p$ lies in $\Ecal$ for every $p \in M$, this forces $M/r$ to be continuous.
\end{proof}

Another way of interpreting Scholium \ref{schl:factors} is as a necessary and sufficient condition for $(M,\tau_h)$ to be Morita-equivalent to the complete monoid representing $\Ecal$.

\begin{xmpl}
\label{xmpl:surjtop}
Any surjective monoid homomorphism $\phi: M \to M'$ induces a hyperconnected geometric morphism $f: \PSh(M) \to \PSh(M')$; see Proposition \ref{prop:essgeom} below. The corresponding filter of $M$-equivariant relations is simply the collection of relations containing $r_{\phi} := \{(m,n) \mid \phi(m) = \phi(n)\}$. As such, it suffices to check the conditions of Scholium \ref{schl:factors} for $M/r_{\phi}$.

Suppose $M/r$ is such that for every $p \in M$, the relation $r_p$ from \eqref{eq:image} contains $r_{\phi}$. Then given $(m,n) \in r_{\phi}$, consider $r_m = \{(p,p') \mid p^*(\Ical_{[1]}^m) = p'{}^*(\Ical_{[1]}^m) \}$, where $[1]$ is the generator for $M/r$. By assumption, $(m,n) \in r_m$, whence $m^*(\Ical_{[1]}^m) = n^*(\Ical_{[1]}^m)$ and $n \in \Ical_{[1]}^m$, which is to say that $(m,n) \in r$, as required. So $T = f^*f_*(P(M))$ generates a topology $\tau_f$ on $M$ such that $\Cont(M,\tau_f) \simeq \PSh(M')$ via $f$.

Of course, we can calculate $T$ directly as the topology whose open sets are the equivalence classes of $r_{\phi}$, and this coincides with $\tau$: since $f$ is essential, $f^*f_*(\Pcal(M))$ is a complete Boolean algebra in $\Set$ (both $f^*$ and $f_*$ preserve small $\Set$-limits). 
\end{xmpl}

While Example \ref{xmpl:surjtop} illustrates that Scholium \ref{schl:factors} provides a workable necessary and sufficient condition, it does not illuminate precisely which toposes arise in this way. See Conjecture \ref{conj:characterization} below for a more detailed discussion of this issue.

%% file: TTMA_Semigroup_Homs.tex
\section{Semigroup homomorphisms}
\label{sec:homomorphism}

In this section we functorialize the results obtained so far by examining how homomorphisms lift to geometric morphisms. In light of the results in \cite{TDMA}, it is sensible to consider as morphisms between topological monoids not only (continuous) monoid homomorphisms, but also (continuous) semigroup homomorphisms, which correspond to essential geometric morphisms between the corresponding presheaf toposes.

\subsection{Restricting essential geometric morphisms}

Let $\phi : M \to M'$ be a semigroup homomorphism between monoids. Recall that $\phi$ induces a functor $\check{\phi}: \check{M} \to \check{M}'$ between the idempotent-completions of the monoids, and hence induces an essential geometric morphism $f: \PSh(M) \to \PSh(M')$.

Factorizing $f$, we have the following result, which is explored further in forthcoming work with Jens Hemelaer \cite{EDMA}:
\begin{prop}
\label{prop:essgeom}
Let $f: \PSh(M) \to \PSh(M')$ be an essential geometric morphism induced by a monoid homomorphism $\phi$, and let $e := \phi(1)$. Then the surjection--inclusion factorization of $f$ is canonically represented by the factorization of $\phi$ as a monoid homomorphism $M \to eM'e$ followed by an inclusion of subsemigroups $eM'e \hookrightarrow M'$. Meanwhile, the hyperconnected--localic factorization of $f$ is canonically represented by the factorization of $\phi$ as surjective monoid homomorphism followed by an injective semigroup homomorphism.
\end{prop}
\begin{proof}
These results are proved by considering the factorization of $\check{\phi}$ corresponding to the surjection--inclusion and hyperconnected--localic factorizations of $f$, which can be found in \cite[Examples 4.2.7(b), 4.2.12(b), 4.6.2(c) and 4.6.9]{Ele}. We find in both cases that the intermediate category is the idempotent completion of the monoid indicated in the statement, whence these factors reduce to the stated semigroup homomorphisms.
\end{proof}

Now consider topologies $\tau$, $\tau'$ on $M$, $M'$ respectively. Then we may consider the square
\begin{equation}
\label{eq:ctshom}
\begin{tikzcd}
{\PSh(M)} \ar[r, shift right = 4, "f_*"'] \ar[r, shift right = 2, phantom, "\bot"] \ar[r, shift left = 2, phantom, "\bot"] \ar[r, shift left=4, "f_!"] \ar[d, shift left = 2, "R"] &
{\PSh(M')} \ar[l, "f^*"'{very near start, inner sep = 0pt}] \ar[d, shift left = 2, "R'"] \\
{\Cont(M{,}\tau)} \ar[u, shift left = 2, hook, "V"] \ar[u, phantom, "\dashv"]&
{\Cont(M'{,}\tau')}  \ar[l, "Rf^*V'"] \ar[u, shift left = 2, hook, "V'"] \ar[u, phantom, "\dashv"],
\end{tikzcd}
\end{equation}
where across the top we have the essential geometric morphism $f$ induced by $\phi$, whose inverse image is induced by tensoring with the left-$M'$-right-$M$-set $M'\phi(1)$ (which coincides with restriction along $\phi$ when $\phi$ is a monoid homomorphism). This situation bears a strong resemblance to that involved in describing morphisms or comorphisms of sites, where the vertical morphisms are inclusions directed upwards rather than hyperconnected morphisms directed downwards. Accordingly, under suitable hypotheses, the lower horizontal map becomes the inverse image functor of a geometric morphism.

\begin{lemma}
\label{lem:cts}
Let $\phi: M \to M'$, $\tau$, $\tau'$ and $f$ be as above. Then the following are equivalent:
\begin{enumerate}
	\item $f^*: \PSh(M') \to \PSh(M)$ maps every $(M',\tau')$-set to an $(M,\tau)$-set.
	\item $\phi$ is continuous with respect to $\tau$ and $\tilde{\tau}'$.
	\item The composite functor $Rf^*V'$ (is left exact and) has a right adjoint $G$ satisfying $GR \cong R'f_*$, which is to say that $f$ restricts along the functors $V,V'$ to a geometric morphism $(G \dashv Rf^*V'): \Cont(M,\tau) \to \Cont(M',\tau')$ making the square commute.
\end{enumerate}
\end{lemma}
\begin{proof}
($1 \Leftrightarrow 2$) The precomposition functor $f^*$ maps every $(M',\tau')$-set to an $(M,\tau)$-set if and only if for each $X \in \Cont(M',\tau')$, we have $\Ical_x^{p} \in \tau$ for every $x \in f^*(X)$, $p \in M$. By definition of the action of $M$ on $f^*(X)$, we have $\Ical_x^{p} = \{m \in M \mid x \phi(p) = x \phi(m)\} = \phi^{-1}(\Ical_x^{\phi(p)})$; thus (1) is equivalent to $\phi^{-1}$ preserving the openness of the necessary clopens, which lie in $\tilde{\tau}'$.

Given any $U' \in \tilde{\tau}'$, we may express $U'$ as a union of necessary clopens $\Ical_x^{p'}$, and $\phi^{-1}(U')$ is the corresponding union of $\phi^{-1}(\Ical_x^{p'})$. Each such clopen $\Ical_x^{p'}$ either intersects with the image of $\phi$ and so is of the form $\Ical_x^{\phi(p)}$, or does not and so has empty inverse image. It follows that $\phi$ reflecting openness of the $\Ical_x^{\phi(p)}$ is equivalent to $\phi$ being continuous with respect to $\tilde{\tau}'$ and $\tau$, as required.

($1 \Leftrightarrow 3$) If $f^*$ maps every $(M',\tau')$-set to an $(M,\tau)$-set, then composing $R$ with $f^*V'$ does not affect the underlying set of the image. That is, $Rf^*V'$ preserves finite limits and arbitrary colimits since $f^*V'$ does and these are computed in $\Cont(M,\tau)$ just as in $\PSh(M)$, making $Rf^*V'$ the inverse image of a geometric morphism by the special adjoint functor theorem. Write $G$ for the direct image. It is immediate that $V(Rf^*V') \cong f^*V'$, which means that the corresponding square of right adjoints commutes up to isomorphism and we have $GR' \cong R'f_*$.

Conversely, given a right adjoint $G$ to $Rf^*V'$ satisfying the given identity we must have $f^*V' \cong VRf^*V'$, which ensures that $f^*$ sends every $(M',\tau')$-set to an $(M,\tau)$-set.
\end{proof}

Thus, since $\phi$ being continuous with respect to $\tau$ and $\tau'$ is in general strictly stronger than condition (2) of Lemma \ref{lem:cts}, we obtain a functorialization of the $\Cont(-)$ construction from the ($1$-)category of topological monoids and continuous semigroup homomorphisms to the ($1$-)category of Grothendieck toposes and geometric morphisms. Let us reintroduce the $2$-morphisms between semigroup homomorphisms.

\begin{dfn}
Recall from \cite[Definition 6.2]{TDMA} that a \textbf{conjugation} $\alpha:\phi \Rightarrow \psi$ between semigroup homomorphisms $\phi,\psi:M \to M'$ is an element $\alpha \in M'$ such that $\alpha \phi(1) = \alpha = \psi(1) \alpha$ and for every $m \in M$, $\alpha \phi(m) = \psi(m) \alpha$.
\end{dfn}

Conjugations correspond bijectively and contravariantly with the natural transformations between the essential geometric morphisms corresponding to $\phi$ and $\psi$, by \cite[Theorem 6.5]{TDMA}. Since $\Cont(M,\tau)$ is a full subcategory of $\PSh(M)$, any natural transformation $\alpha: f^* \Rightarrow g^*$ restricts along $R$ to give a natural transformation $Rf^*V' \Rightarrow Rg^*V'$; this is shown to be a general property of any connected geometric morphism in Proposition \ref{prop:conncoff}. Thus, we conclude:

\begin{thm}
\label{thm:Cont}
The construction $\Cont(-)$ is a $2$-functor from the $2$-category of topological monoids, continuous semigroup homomorphisms and conjugations to the $2$-category of Grothendieck toposes, geometric morphisms and natural transformations. We may restrict the codomain of this $2$-functor to those Grothendieck toposes satisfying the condition of Theorem \ref{thm:characterization} to make it essentially surjective on objects. Since every such topos has a representative which is a complete monoid, we may also restrict the domain to the class of complete monoids without changing this fact.
\end{thm}

We make some further comments about this $2$-functor before Example \ref{xmpl:Schanuel}.

One aim of the remainder of this section is to investigate the surjection--inclusion and hyperconnected--localic factorizations of a geometric morphism corresponding to a continuous semigroup homomorphism, to extend Proposition \ref{prop:essgeom} to topological monoids. We ultimately show in Theorems \ref{thm:surjinc} and \ref{thm:hypeloc} that these factorizations restrict along $\Cont(-)$.

\subsection{Intrinsic properties of geometric morphisms}
\label{ssec:intrinsic}

For reference in the rest of the section, we shall use the following notation for the square of geometric morphisms induced by a continuous semigroup homomorphism $\phi: (M,\tau) \to (M',\tau')$ thanks to Lemma \ref{lem:cts}:
\begin{equation}
\label{eq:square}
\begin{tikzcd}
{\PSh(M)} \ar[r, "f"] \ar[d, "h"'] &
{\PSh(M')} \ar[d, "h'"] \\
{\Cont(M{,}\tau)} \ar[r, "g"'] &
{\Cont(M'{,}\tau')},
\end{tikzcd}	
\end{equation}
where $h$ and $h'$ are hyperconnected and $f$ is essential; we could alternatively have denoted $g$ by $\Cont(\phi)$ in accordance with Theorem \ref{thm:Cont}, but the shorter notation will make some of the results below clearer. In order to understand the relationships between $f$ and $g$, we shall exploit intrinsic properties of the geometric morphisms $h$ and $h'$ as $1$-morphisms in the $2$-category $\TOP$ of Grothendieck toposes\footnote{Some of these results apply more generally, but for the purposes of the present paper we only concern ourselves with Grothendieck toposes over $\Set$.}, in the special cases that $(M,\tau)$ and/or $(M',\tau')$ are powder monoids or complete monoids. 

Given (Grothendieck toposes) $\Ecal$ and $\Fcal$, we shall write $\Geom(\Ecal,\Fcal)$ for the category of geometric morphisms $\Ecal \to \Fcal$, where a morphism $f \Rightarrow g$ is as usual a natural transformation $f^* \Rightarrow g^*$. We shall also write $\EssGeom(\Ecal,\Fcal)$ for the full subcategory of essential geometric morphisms $\Ecal \to \Fcal$.

First, we can use Corollary \ref{crly:Mlim} to give an intrinsic characterization of the hyperconnected morphism presenting a complete monoid.

\begin{prop}
\label{prop:intrinsic}
A topological monoid $(M,\tau)$ is complete if and only if the geometric morphism $h: \PSh(M) \to \Cont(M,\tau)$ is representably full and faithful on essential geometric morphisms in $\TOP$, in the sense that for any topos $\Fcal$, the functor
\[h \circ -: \EssGeom(\Fcal,\PSh(M)) \to \Geom(\Fcal,\Cont(M,\tau))\] 
is full and faithful.

Thus, if $\Ecal$ admits a hyperconnected morphism $h:\PSh(M) \to \Ecal$ which is representably full and faithful on essential geometric morphisms, the corresponding topological monoid $(L,\rho)$ representing $\Ecal$ has $L \cong M$.
\end{prop}
\begin{proof}
By taking $\Fcal = \Set$ and considering the canonical point of $\PSh(M)$, we see that the given condition is sufficient, since it forces the monoid of endomorphisms of the canonical point of $\Cont(M,\tau)$ to be isomorphic to that of $\PSh(M)$, which is precisely $M\op$. The same holds when we are given such a morphism $\PSh(M) \to \Ecal$.

Conversely, suppose we are given essential geometric morphisms $h,k: \Fcal \to \PSh(M)$ and a natural transformation $\alpha:h^* \Rightarrow k^*$. Any such natural transformation is determined by its component $\alpha_M:h^*(M) \to k^*(M)$. But since $(M,\tau)$ is complete, by Corollary \ref{crly:Mlim}, $M = \lim_{r \in \Rcal_{\tau}} M/r$, and $h^*$ and $k^*$ preserve all limits, whence $\alpha_M$ is determined uniquely by the components $\alpha_{M/r}$. The functor induced by $g$ sends $\alpha$ to $\alpha_{g^*}$; considering the components at the principal $(M,\tau)$-sets, we conclude that this functor is full and faithful, as claimed.
\end{proof}

Proposition \ref{prop:intrinsic} should be compared with the following two propositions:
\begin{prop}
\label{prop:inclff}
Inclusions of toposes are representably full and faithful in the $2$-category $\TOP$, in the sense that given an inclusion $g:\Fcal \to \Ecal$ and any topos $\Gcal$, the functor $g \circ -: \Geom(\Gcal,\Fcal) \to \Geom(\Gcal,\Ecal)$ is fully faithful.

In particular, $g$ is full and faithful on essential geometric morphisms in the sense of Proposition \ref{prop:intrinsic}, and when $g$ is an essential inclusion, we may restrict the codomain to deduce that $i \circ -: \EssGeom(\Gcal,\Fcal) \to \EssGeom(\Gcal,\Ecal)$ is full and faithful.
\end{prop}
\begin{proof}
Let $g:\Fcal \to \Ecal$ be a geometric inclusion, and let $h,k: \Gcal \rightrightarrows \Fcal$.

A geometric transformation $h \Rightarrow k$ consists of a natural transformation $h^* \Rightarrow k^*$. Let $\alpha$, $\beta$ be two such transformations. If $g \circ \beta = g \circ \alpha$, then for any object $C$ of $\Gcal$, letting $\epsilon_C$ denote the counit of $(i^* \dashv i_*)$ at $C$, which is an isomorphism, we have:
\begin{align*}
\alpha_C & = k^*\epsilon_{C} \circ \alpha_{g^*g_*(C)} \circ h^*\epsilon_C^{-1} = k^*\epsilon_{C} \circ (g \circ \alpha)_{g_*(C)} \circ h^*\epsilon_{C}^{-1}\\
& = k^*\epsilon_{C} \circ (g \circ \beta)_{g_*(C)} \circ h^*\epsilon_{C}^{-1} = k^*\epsilon_{C} \circ \beta_{g^*g_*(C)} \circ h^*\epsilon_{C}^{-1} = \beta_C,
\end{align*}
so $g \circ -$ is faithful.

Similarly, given $\alpha': h^*g^* \Rightarrow k^*g^*$, define $\alpha: h^* \Rightarrow k^*$ by letting its component at $C$ in $\Gcal$ be $k^*\epsilon_{C}^{-1} \circ \alpha'_{g_*(C)} \circ h^*\epsilon_{C}$. Then for each object $D$ in $\Fcal$ we have
\[(g \circ \alpha)_{D} = \alpha_{g^*(D)} = k^*\epsilon_{g^*(D)}^{-1} \circ \alpha'_{g_*g^*(D)} \circ h^*\epsilon_{g^*(D)} = \alpha'_D,\]
by naturality. So $g \circ -$ is full, as required.
\end{proof}

\begin{prop}
\label{prop:conncoff}
Connected geometric morphisms are representably cofull and cofaithful in the $2$-category $\TOP$, in the sense that given a connected morphism $c:\Fcal \to \Ecal$ and any topos $\Gcal$, the functor $- \circ c: \Geom(\Ecal,\Gcal) \to \Geom(\Fcal,\Gcal)$ is fully faithful.

This applies in particular to hyperconnected geometric morphisms.
\end{prop}
\begin{proof}
Let $c:\Fcal \to \Ecal$ be a connected geometric morphism. Then concretely, $- \circ c$ is simply the application of $c^*$ to the components of any given natural transformation. As such, since $c^*$ is full and faithful, $- \circ c$ is full and faithful.
\end{proof}

\begin{crly}
\label{crly:Cont}
The restriction of $\Cont(-)$ to the category of complete monoids, semigroup homomorphisms and conjugations (with direction reversed) is full and faithful on 2-cells.
\end{crly}
\begin{proof}
As mentioned above, the $2$-equivalence of \cite[Theorem 6.5]{TDMA} mapping a discrete monoid to its presheaf topos is (contravariantly) full and faithful on $2$-cells, so for essential geometric morphisms $f,f': \PSh(M) \to \PSh(M')$ induced by $\phi$, $\phi'$ respectively, each geometric transformation $\alpha: f \Rightarrow f'$ corresponds to a unique conjugation $\phi \Rightarrow \phi'$. Since $h'$ is full and faithful on essential geometric morphisms, the same is true of $2$-cells $h' \circ f \Rightarrow h' \circ f'$, by Proposition \ref{prop:intrinsic}. Passing across the square \eqref{eq:square}, since $h$ is cofull and cofaithful by Proposition \ref{prop:conncoff}, we obtain a further identification with the geometric transformations $g \Rightarrow g'$ (where $g$, $g'$ are also induced by $\phi$, $\phi'$ respectively), as required.
\end{proof}

\begin{rmk}
Explicitly, the conjugation corresponding to a transformation $\beta: g \Rightarrow g'$ is obtained as follows. First, compose with $h$ and take the limit of the components at the principal $(M',\tau')$-sets to obtain an $M$-set homomorphism $\alpha_{M'}: f^*(M') \to f'{}^*(M')$ (using the limit expression from Corollary \ref{crly:Mlim} and essentialness of $f,f'$ again), and by extension a natural transformation $\alpha: f^* \Rightarrow f'{}^*$. Then we take the mate $\overline{\alpha}:f'_! \Rightarrow f_!$, whose component at $M$ is the desired conjugation.
\end{rmk}

Corollary \ref{crly:Cont} can be understood as a strengthening of Proposition \ref{prop:lim}, since taking the domain monoid $M$ to be the trivial monoid and $\phi = \psi$ to be the unique monoid homomorphism to $(M',\tau')$, the conjugations are precisely the elements of $M'$.

Returning to properties of geometric morphisms, Proposition \ref{prop:inclff} will allow us to constrain the interactions between hyperconnected geometric morphisms and geometric inclusions in Section \ref{ssec:id}, but we also need to consider localic geometric morphisms.
\begin{prop}
\label{prop:locfaith}
Localic geometric morphisms are representably faithful in the $2$-category $\TOP$ of toposes, in the sense that given a localic geometric morphism $f:\Fcal \to \Ecal$ and any topos $\Gcal$, the functor $f \circ -: \Geom(\Gcal,\Fcal) \to \Geom(\Gcal,\Ecal)$ is faithful.
\end{prop}
\begin{proof}
A geometric morphism $g:\Fcal \to \Ecal$ is localic if and only if every object $Y$ of $\Fcal$ is a subquotient of one of the form $g^*(X)$ for $X$ in $\Ecal$, so there exists a diagram
\[\begin{tikzcd}
Y & \ar[l, two heads, "e"'] Z \ar[r, hook, "m"] & g^*(X).
\end{tikzcd}\]
Given geometric morphisms $h,k: \Gcal \rightrightarrows \Fcal$ and geometric transformations $\alpha,\beta: h \rightarrow k$ with $g \circ \alpha = g \circ \beta$, this is equivalent to the condition that $\alpha_{g^*(X)} = \beta_{g^*(X)}$ for all objects $X$ of $\Ecal$. Then considering naturality across the diagram above, since $h^*$ and $k^*$ both preserve epimorphisms and monomorphisms, we have:
\[\begin{tikzcd}
h^*g^*(X) \ar[r,"\alpha_{g^*(X)}", shift left] \ar[r,"\beta_{g^*(X)}"', shift right] &
k^*g^*(X) \\
h^*(Z) \ar[r,"\alpha_{Z}", shift left] \ar[r,"\beta_{Z}"', shift right] \ar[d, "h^*e"', two heads] \ar[u, "h^*m", hook] &
k^*g^*(X) \ar[d, "k^*e", two heads] \ar[u, "k^*m"', hook] \\
h^*(Y) \ar[r,"\alpha_{Y}", shift left] \ar[r,"\beta_{Y}"', shift right] & k^*(Y),
\end{tikzcd}\]
whence we see that $\alpha_{Z} = \beta_{Z}$ and then $\alpha_Y = \beta_Y$. Since $Y$ was a generic object of $\Fcal$, $\alpha = \beta$, as required.
\end{proof}

\begin{crly}
\label{crly:powderfaith}
Let $\tau$ be an action topology on $M$. Then $(M,\tau)$ is a powder monoid (equivalently, $\tau$ is a $T_0$ topology on $M$) if and only if the hyperconnected geometric morphism $\PSh(M) \to \Cont(M,\tau)$ is representably faithful on essential geometric morphisms.
\end{crly}
\begin{proof}
Let $(L,\rho)$ be the completion of $(M,\tau)$ and consider the continuous, dense monoid homomorphism $u:(M,\tau) \to (L,\rho)$. By Corollary \ref{crly:extend}, when $(M,\tau)$ is a powder monoid, $u$ is injective, so the induced geometric morphism $f:\PSh(M) \to \PSh(L)$ is a localic surjection. Considering the square \eqref{eq:square}, since $h'$ is full and faithful on essential geometric morphisms by Proposition \ref{prop:intrinsic} and $f$ is faithful on these, it follows that $gh$ and hence $h$ are both faithful on essential geometric morphisms, as claimed.

Conversely, the morphism $g$ induced by $u$ is an equivalence, so if $h$ is faithful on essential geometric morphisms, then so is $f$, since $h'$ is full and faithful on such. Thus, $u:M \to L$ must be injective, since we can recover it as the restriction of the functor $f \circ -: \EssGeom(\Set,\PSh(M)) \to \EssGeom(\Set,\PSh(L))$ to the endomorphisms of the canonical point of $\PSh(M)$. But $(L,\tau)$ is $T_0$, and any submonoid/subspace of a $T_0$ monoid must also be $T_0$, as required.
\end{proof}

Note that unlike in Proposition \ref{prop:intrinsic}, we cannot deduce that an arbitrary hyperconnected morphism $\PSh(M) \to \Ecal$ which is faithful on essential geometric morphisms expresses $\Ecal$ as $\Cont(M,\tau)$; the non-topological factor of Theorem \ref{thm:factor} may be non-trivial.

Recall from \cite[Section 6]{TDMA} that we can factorize any semigroup homomorphism $\phi:M \to M'$ into a monoid homomorphism followed by the inclusion of a subsemigroup of the form $\phi(1)M'\phi(1)$ into $M'$, and that this lifts to (a canonical representation of) the surjection-inclusion factorization of the essential geometric morphism corresponding to $\phi$. Accordingly, we separate the analysis of geometric morphisms coming from continuous semigroup homomorphisms into the analysis of inclusions of subsemigroups and continuous monoid homomorphisms, in Sections \ref{ssec:id} and \ref{ssec:monhom} respectively.

\subsection{Subsemigroups}
\label{ssec:id}

Throughout this section, $(M',\tau')$ is a right powder monoid and $e \in M'$ an idempotent. Consider the following topological observations about the ideals generated by idempotents.

\begin{lemma}
\label{lem:ideal}
The principal left ideal $M'e$ and the principal right ideal $eM'$ of $M'$ are closed in $(M',\tau')$.
\end{lemma}
\begin{proof}
We can characterize $M'e$ and $eM'$ as the subsets of $M'$ on those elements $p$ such that $pe = p$ and $p = ep$ respectively. 

Suppose $x$ is outside $M'e$. Since $(M',\tau')$ is zero-dimensional Hausdorff, we can find a basic clopen set $U$ with $x \in U$ and $xe$ in the complement of $U$. Then $\Ical_U^x$ is an open set containing $x$; if $p \in \Ical_U^x$ then since $e \notin x^*(U)$, we have $e \notin p^*(U)$ so $pe \neq p$. Thus we conclude that $M'e$ is contained in the complement of $\Ical_U^x$, and $M'e$ is closed.

Similarly, if $x$ is outside $eM'$, let $U$ be a basic clopen set containing $x$ but not $ex$. Then $e^*(M \backslash U)$ contains $x$, so we may consider the smaller neighbourhood $U \cap e^*(M \backslash U)$ of $x$. This excludes any element $p$ with $p = ep$, so that in particular $eM'$ is contained in the complement. Thus $eM'$ is closed.
\end{proof}

\begin{prop}
\label{prop:eMe}
Let $(M',\tau')$ be a powder monoid, $e$ an idempotent of $M'$ and $M := eM'e$. Let $\phi: M \hookrightarrow M'$ be the corresponding inclusion of semigroups, and let $\tau$ be the topology on $eM'e$ obtained by restricting $\tau'$. Then $(M,\tau)$ is a powder monoid, and any subsemigroup of this form is closed in $M'$.
\end{prop}
\begin{proof}
Being the coarsest topology on $M$ such that $\phi:M \to M'$ is continuous, $\tau$ is the coarsest topology such that (the hyperconnected part of) the morphism $h' \circ f$ in the square \eqref{eq:square} induced by $\phi$ factors through $\PSh(M) \to \Cont(M,\tau)$, by Theorem \ref{thm:factor}. Thus $\tau$ is an action topology on $M$. As a subspace of a Hausdorff space, $(M,\tau)$ is Hausdorff (we could alternatively have used Corollary \ref{crly:powderfaith} to deduce this). 

To see that $M$ is closed in $M'$, observe that it is the intersection of the ideals $eM'$ and $M'e$ which we showed to be closed in Lemma \ref{lem:ideal}.
\end{proof}

\begin{xmpl}
\label{xmpl:nopen}
On the other hand, $eM'e$ is not always open in $M'$. Indeed, consider the prodiscrete monoid constructed in Example \ref{xmpl:idemclosed}. The idempotent element $e = \infty$ is a zero element, so that the corresponding subsemigroup is simply $\{\infty\}$, which from the description of the topology on this monoid clearly fails to be open.
\end{xmpl}

Intuitively, we might expect the geometric morphism induced by the subsemigroup inclusion $M \hookrightarrow M'$ in Proposition \ref{prop:eMe} to be a geometric inclusion. To explain why this is the case in Example \ref{xmpl:nopen}, in the sense that the inclusion of $\{\infty\}$ induces a geometric inclusion, we show that this intuition is at least valid for complete monoids\footnote{We have not been able to demonstrate it for powder monoids more generally.}.

\begin{thm}
\label{thm:inccomplete}
Let $(M',\tau')$ be a complete monoid, let $M = eM'e$ for some idempotent $e \in M'$, and let $\phi:M \to M'$ be the subsemigroup inclusion. Then the restricted topology $\tau := \tau'|_M$ makes $(M,\tau)$ a complete topological monoid, and hence the induced geometric morphism $\Cont(M,\tau) \to \Cont(M',\tau')$ is a geometric inclusion.
\end{thm}
\begin{proof}
As usual, let $f:\PSh(M) \to \PSh(M')$ be the essential inclusion induced by $\phi$. Consider the hyperconnected-localic factorization of the following composite morphism $\PSh(M) \to \Cont(M',\tau')$:
\[\begin{tikzcd}
{\PSh(M)} \ar[r, "f"] \ar[d, "h"'] &
{\PSh(M')} \ar[d, "h'"] \\
{\Ecal} \ar[r, "g"] &
\Cont(M'{,}\tau').
\end{tikzcd}\]
Since the upper composite is an inclusion followed by a hyperconnected morphism, by \cite[Proposition A4.6.10]{Ele} the lower geometric morphism is an inclusion: the surjection-inclusion and hyperconnected-localic factorizations of the composite coincide.

Moreover, combining Proposition \ref{prop:intrinsic} with Proposition \ref{prop:inclff}, we have that the composite is full and faithful on essential geometric morphisms, and hence the hyperconnected part $\PSh(M) \to \Ecal$ also is. Thus the complete monoid representing $\Ecal$ has $M$ as its underlying monoid. That the corresponding topology is the restriction topology follows from Theorem \ref{thm:factor}, just as in the proof of Proposition \ref{prop:eMe}.
\end{proof}

\subsection{Monoid homomorphisms}
\label{ssec:monhom}

Now suppose $\phi:(M,\tau) \to (M',\tau')$ is a continuous \textit{monoid} homomorphism, so that the essential geometric morphism $f:\PSh(M) \to \PSh(M')$ it induces is a surjection, and hence examining the resulting square \eqref{eq:square}, so is the induced morphism $g: \Cont(M,\tau) \to \Cont(M',\tau')$. Combining this observation with Theorem \ref{thm:inccomplete}, we have that:
\begin{thm}
\label{thm:surjinc}
Let $\phi:(M,\tau) \to (M',\tau')$ be a continuous semigroup homomorphism between complete monoids inducing $g: \Cont(M,\tau)\to \Cont(M',\tau')$, and let $e:= \phi(1)$. Then the surjection--inclusion factorization of $g$ is canonically represented by the factorization of $\phi$ into a monoid homomorphism $M \to eM'e$ followed by an inclusion of subsemigroups $eM'e \hookrightarrow M'$, where $eM'e$ is equipped with the subspace topology.
\end{thm}

We can characterize which morphisms arise from continuous monoid homomorphisms, up to fixing sober representing monoids. Since sobriety has not been a focus of our account of monoids in this paper, we restrict further to powder monoids.

\begin{prop}
\label{prop:phisom}
Let $(M,\tau)$ and $(M',\tau')$ be powder monoids; let $T$, $T'$ be their respective Boolean algebras of clopen sets. A surjective geometric morphism  of the form ${g:\Cont(M,\tau)\to \Cont(M',\tau')}$ is induced by a continuous monoid homomorphism $\phi:(M,\tau) \to (M',\tau')$ if and only if $T' \cong g_*(T)$ in $\Cont(M',\tau')$.
\end{prop}
\begin{proof}
First suppose that $g$ is induced by some continuous monoid homomorphism $\phi$. Consider the internal Boolean algebras $T$ and $T'$ in the respective toposes. Since $\phi$ is a monoid homomorphism, $g$ commutes with the canonical points of $\Cont(M,\tau)$ and $\Cont(M',\tau')$ (which are induced by the unique monoid homomorphisms $1 \to M$ and $1 \to M'$ respectively). In particular, we have a canonical isomorphism $T' \cong g_*(T)$, as required.

Conversely, given such an isomorphism, both $T'$ and $g_*(T)$ represent $\Pcal\op \circ U \circ V$ by Scholium \ref{schl:Tseparator}, from which it follows that $g$ commutes with the canonical points. As such, $g^*(T')$ has an underlying set which can be identified with that of $T'$, and the counit of $g$ at $T$ provides a homomorphism of Boolean algebras $\phi^{-1}: g^*(T') \to T$ in $\PSh(M)$ which, since powder monoids are sober as spaces, uniquely defines a map $M \to M'$. The fact that $\phi^{-1}$ is an $M$-set homomorphism ultimately ensures that $\phi$ is a monoid homomorphism. Moreover, when we generate $g$ from a monoid homomorphism $\phi$, we find that this is the morphism we recover from $\phi^{-1}$.
\end{proof}

For an injective monoid homomorphism, we have a partial analogue of Theorem \ref{thm:inccomplete}:
\begin{lemma}
\label{lem:surjective}
Suppose $(M',\tau')$ is a powder monoid and $\phi:M \to M'$ is an injective monoid homomorphism. Then the subspace topology $\tau := \tau'|_M$ on $M$ makes $(M,\tau)$ a powder monoid.
\end{lemma}
\begin{proof}
Consider the square \eqref{eq:square} induced by $\phi$. By Proposition \ref{prop:locfaith}, the localic surjection $f$ is faithful on essential geometric morphisms, and by Corollary \ref{crly:powderfaith}, so is $h'$. It follows that $h$ must also be faithful on essential geometric morphisms, and hence that $(M,\tau)$ also is.
\end{proof}

As in Corollary \ref{crly:powderfaith}, we encounter the problem that we have no guarantee that the morphism $g:\Cont(M,\tau) \to \Cont(M',\tau')$ in Lemma \ref{lem:surjective} induced by an injective monoid homomorphism $\phi$ will have a trivial hyperconnected part. However, we can use whatever hyperconnected part there may be to produce a complete monoid, and when the codomain is also a complete monoid, this gives us a canonical factorization of $\phi$, as follows.

\begin{thm}
\label{thm:locextend}
Suppose $(M',\tau')$ is a complete monoid, and $\phi: M \to M'$ is an injective monoid homomorphism. Let $\tau$ be the restriction of $\tau'$ as in Lemma \ref{lem:surjective}. Then there is a complete monoid $(L,\rho)$ and a dense, continuous, injective monoid homomorphism $(M,\tau) \to (L,\rho)$ such that $\phi$ extends to a continuous injection $\psi: (L,\rho) \to (M',\tau')$ inducing a localic surjection $\Cont(L,\rho) \to \Cont(M',\tau')$.
\end{thm}
\begin{proof}
Of course, we define $(L,\rho)$ to be the complete monoid obtained from the hyperconnected part of the composite $h' \circ f$ in the square \eqref{eq:square} induced by $\phi$. Thus we have a diagram of geometric morphisms:
\begin{equation}
\label{eq:factors}
\begin{tikzcd}
{\PSh(M)} \ar[rr, "f"] \ar[dr, "u"'] \ar[dd, "h"'] & &
{\PSh(M')} \ar[dd, "h'"] \\
& {\PSh(L)} \ar[dd, "k", near start] \ar[ur, dashed] & \\
{\Cont(M{,}\tau)} \ar[rr, "g", near start] \ar[dr, "v"'] & &
{\Cont(M'{,}\tau')}, \\
& {\Cont(L{,}\rho)} \ar[ur, "w"'] &
\end{tikzcd}	
\end{equation}
in which $h$, $h'$, $k$ and $v$ are hyperconnected, while $f$ and $w$ are localic surjections and $f$ and $u$ are essential. It suffices for us to construct the inclusion $\psi: L \to M'$ to provide the (dashed) essential geometric morphism which restricts to $w$.

The homomorphism $\phi$ induces a morphism $M \to f^*(L') = f^*f_!(M)$, the unit of the adjunction $(f_! \dashv f^*)$, whose element-wise action coincides with $\phi$; we abuse notation and call this unit map $\phi$, too.

Define a mapping $t: \Rcal_{\tau'} \to \Rcal_{\tau}$ by pullback: for each $r' \in \Rcal_{\tau'}$, let $t(r')$ be the pullback of $f^*(r')$ along $\phi$. The resulting relation is such that the intermediate principal $M$-set in the epi-mono factorization of $M \to f^*(M') \too f^*(M'/r')$ is precisely $M/t(r')$, so we have $M \too M/t(r') \hookrightarrow f^*(M'/r')$. By construction, $f^*$ sends $(M',\tau')$-sets to $(M,\tau)$-sets lying in $\Cont(L,\rho)$ and this subcategory is closed under subobjects, so $M/t(r')$ is naturally an $(L,\rho)$-set. As such, we obtain maps $L \too M/t(r') \hookrightarrow f^*(M'/r')$ by factoring each $M \too M/t(r')$ through $L = \varprojlim_{r \in \Rcal_{v \circ h}} M/r$. These assemble into an $M$-set homomorphism
\[\psi: L \to f^*(M') = \varprojlim_{r \in \Rcal_{\tau'}} f^*(M'/r'),\]
which explicitly sends $\alpha \in L$ to $\psi(\alpha) = ([\phi(a_{t(r')})])_{r' \in \Rcal_{\tau'}}$. It remains to show that this map underlies a monoid homomorphism. As such, observe that for $a \in M$ and $r' \in \Rcal_{\tau'}$:
\begin{align*}
t(\phi(a)^*(r')) & =
t\left( \{(x,y) \in M' \times M' \mid (\phi(a)x,\phi(a)y) \in r' \} \right) \\ & =
\{(m,n) \in M \times M \mid (\phi(a)\phi(m),\phi(a)\phi(n)) \in r' \} \\ & =
a^* \left( \{(m,n) \in M \times M \mid (\phi(m),\phi(n)) \in r' \} \right) \\ & =
a^*(t(r')),
\end{align*}
where we have omitted each instance of $f^*$. Therefore, given $\alpha,\beta \in L$, we have
\begin{align*}
\psi(\alpha\beta) & =
([\phi(a_{t(r')}b_{a_{t(r')}^*(t(r'))})]) \\ & =
([\phi(a_{t(r')})\phi(b_{t(\phi(a_{t(r')})^*(r'))})]) \\ & =
\psi(\alpha)\psi(\beta),
\end{align*}
as required. Preservation of the unit follows from our assumption that $\phi$ was a monoid homomorphism. To demonstrate continuity, observe that if $U = \pi_{r'}^{-1}(\{[m']\})$ is a basic open set in $M'$, then $\psi^{-1}(U)$ is exactly the basic open set $\pi_{t(r')}^{-1}(\{[m]\})$ in $L$ if there is some $m \in M$ with $(\phi(m),m') \in r'$, and is empty otherwise.

Having shown that it exists, we immediately have that $\psi$ is the unique continuous monoid homomorphism making the desired triangle commute, since dense inclusions of Hausdorff spaces are epimorphisms in the category of such spaces.
\end{proof}

The construction of the factoring map $\psi$ in the above relies on the expression of $L$ as a limit. It is useful to have a more topological characterization, for which we need some further preliminary results.

\begin{lemma}
\label{lem:densects}
Let $\phi: (M,\tau) \to (M',\tau')$ be a continuous monoid homomorphism whose image is dense. Then the geometric morphism $g: \Cont(M,\tau) \to \Cont(M',\tau')$ induced by $\phi$ is hyperconnected.
\end{lemma}
\begin{proof}
First, observe that without loss of generality we may assume $\phi$ is a dense inclusion of monoids. Indeed, $\phi$ always factors as a surjective monoid homomorphism followed by a dense inclusion of monoids, and the former factor induces a hyperconnected essential morphism $f$ at the level of the presheaf toposes, whence by consideration of the square \eqref{eq:square}, $g$ must also be hyperconnected. As such, we identify $M$ with its image in $M'$.

Given an $M$-set homomorphism $s: g^*(X) \to g^*(Y)$, $x \in g^*(X)$ and $m' \in M'$, let $m \in M \cap \Ical_x^{m'} \cap \Ical_{s(x)}^{m'}$. Then we have $s(x \cdot m') = s(x \cdot m) = s(x) \cdot m = s(x) \cdot m'$, whence $s$ is an $M'$-set homomorphism, and so $g^*$ is full; it is always faithful when $\phi$ is a monoid homomorphism.

Moreover, the image of $g^*$ is closed under subobjects: given a sub-$M$-set $A \hookrightarrow g^*(Y)$ and $m' \in M'$, for each $y \in A$ we have some $m \in M \cap \Ical_y^{m'}$, whence $y \cdot m = y \cdot m' \in A$, and hence $A$ is a sub-$M'$-set. Altogether, this ensures that $g$ is hyperconnected, as claimed.
\end{proof}

\begin{prop}
\label{prop:denseisom}
Suppose $\phi:(M,\tau) \to (M',\tau')$ is a monoid homomorphism between topological monoids inducing an equivalence $\Cont(M,\tau) \simeq \Cont(M',\tau')$, that $(M,\tau)$ is a complete monoid and that $(M',\tau')$ is a powder monoid. Then $\phi$ is an isomorphism.
\end{prop}
\begin{proof}
Since an equivalence is full and faithful on geometric morphisms, in the square \eqref{eq:square} induced by $\phi$ is follows that $f$ is faithful on essential geometric morphisms and $h$ is full on those in the image of $f$. Since $\phi$ is a monoid homomorphism, it commutes with the canonical points of $\PSh(M)$ and $\PSh(M')$, whence the latter point is in the image of $f$, and hence ($h$ being faithful on essential geometric morphisms by assumption) $M'$ is a complete monoid, and $h$ is full and faithful on essential geometric morphisms. It follows that $f$ is also full and faithful on essential geometric morphisms, and so $\phi$ (the restriction of $f \circ -$ to the canonical point of $\PSh(M)$) is an isomorphism, as claimed.
\end{proof}

\begin{crly}
\label{crly:closure}
The monoid $(L,\rho)$ constructed in Theorem \ref{thm:locextend} is the closure of $(M,\tau)$ in $(M',\tau')$.
\end{crly}
\begin{proof}
Certainly $(M,\tau)$ is dense in $(L,\rho)$. Consider the dense-closed factorization of $\psi: (L,\rho) \to (M',\tau')$. If the dense part is non-trivial, by Lemma \ref{lem:densects} it produces a hyperconnected factor of the geometric morphism induced by $\psi$, and hence must be an equivalence. The intermediate monoid is a powder monoid by Lemma \ref{lem:surjective}, and thus the dense part is an isomorphism by Proposition \ref{prop:denseisom}.
\end{proof}

\begin{crly}
\label{crly:closed}
Any monoid which is a closed subsemigroup of a complete monoid is complete.
\end{crly}
\begin{proof}
Applying the factorization of Theorem \ref{thm:locextend} to the inclusion of a closed submonoid, the submonoid is dense in the complete intermediate monoid and so, being closed, must coincide with it, whence it is complete by Corollary \ref{crly:closure}. Combining this with Proposition \ref{prop:eMe}, the result follows.
\end{proof}

To summarize, we have that:
\begin{thm}
\label{thm:hypeloc}
Let $\phi:(M,\tau) \to (M',\tau')$ be a continuous semigroup homomorphism between complete monoids inducing $g: \Cont(M,\tau)\to \Cont(M',\tau')$. Then the hyperconnected--localic factorization of $g$ is canonically represented by the dense-closed factorization of $\phi$.
\end{thm}

\subsection{Morita equivalence}
\label{ssec:Morita}

In Sections \ref{sec:montop} and \ref{sec:surjpt} we saw how an arbitrary monoid can be reduced to a powder monoid and then extended to a complete monoid without changing its topos of actions (up to canonical equivalence). Proposition \ref{prop:denseisom} above demonstrates that a continuous monoid homomorphism $\phi$ between complete monoids induces an equivalence if and only if $\phi$ is an isomorphism, just as in the discrete case. Thus, as far as Morita equivalence for complete monoids \textit{via semigroup homomorphisms} goes, we are reduced to considering subsemigroups of the form $eM'e \hookrightarrow M'$.

First, observe that Morita equivalences of discrete monoids descend to complete topologies on those monoids. See also Conjecture \ref{conj:inclusion} below.
\begin{schl}
\label{schl:Morita3}
Let $(M',\tau')$ be a complete monoid, and suppose $e \in M'$ is an idempotent such that the inclusion $\iota: M := eM'e \hookrightarrow M'$ induces an equivalence $\PSh(M) \simeq \PSh(M')$. Let $\tau$ be the restriction topology on $M$ from Theorem \ref{thm:inccomplete}. Then $\iota$ induces an equivalence $\Cont(M,\tau) \simeq \Cont(M',\tau')$.
\end{schl}
\begin{proof}
As above, we use the notation of \eqref{eq:square}. We know from the proof of Theorem \ref{thm:inccomplete} that in this situation, $g: \Cont(M',\tau') \to \Cont(M,\tau)$ is the inclusion part of the hyperconnected-inclusion factorization of $h' \circ f$, so when $f$ is an equivalence, the inclusion part must be trivial. That is, $g$ is itself an equivalence.
\end{proof} 

In summary, we have:
\begin{thm}
\label{thm:Moritaphi}
Let $\phi:(M,\tau) \to (M',\tau')$ be a continuous semigroup homomorphism between complete monoids. Then if $\phi$ induces an equivalence $g:\Cont(M,\tau) \to \Cont(M',\tau')$, then $M \cong eM'e$ for some idempotent $e \in M'$, $\phi$ is the canonical inclusion of subsemigroups, and $\tau$ is the restriction of $\tau'$ along this inclusion.
\end{thm}
\begin{proof}
Factoring $\phi$ as a monoid homomorphism followed by an inclusion of semigroups, the induced geometric morphisms between toposes of continuous actions must also be equivalences, whence the former must be an isomorphism by Proposition \ref{prop:phisom}. The topology on the intermediate monoid is the restriction topology by \ref{thm:surjinc}.
\end{proof}

We would have liked to extend the $2$-equivalence between the $2$-category of discrete monoids and the $2$-category of their toposes of actions from \cite[Theorem 6.5]{TDMA}
by characterizing the geometric morphisms arising from continuous semigroup homomorphisms. However, this is not possible because, unlike in the discrete case, \textit{not all equivalences of toposes of topological monoid actions lie in the image of $\Cont(-)$}. Since all equivalences have indistinguishable categorical properties, but only some such equivalences are induced by semigroup homomorphisms, we cannot hope for an intrinsic characterization of geometric morphisms lying in the image of $\Cont(-)$.

\begin{xmpl}
\label{xmpl:Schanuel}
Consider once again the Schanuel topos of Example \ref{xmpl:notpro2}. Let $X$ be $\mathbb{N}$ or $\mathbb{R}$. In either case, the monoid $(\End_{\mathrm{mono}}(X),\tau_{\mathrm{fin}})$ representing the Schanuel topos has no non-identity idempotents, since $e^2(x) = e(x)$ implies $e(x) = x$ by injectivity, so any semigroup homomorphism in either direction must be a monoid homomorphism. Since the two representing monoids are complete and non-isomorphic, no such homomorphism can induce the equivalence of toposes
\[ (\End_{\mathrm{mono}}(\Nbb),\tau_{\mathrm{fin}}) \simeq
(\End_{\mathrm{mono}}(\Rbb),\tau_{\mathrm{fin}}). \]
\end{xmpl}

\begin{rmk}
Recall from \cite[Propositions 1.5 and 1.8]{MPaTP} that general geometric morphisms between toposes of discrete monoid actions $\PSh(M) \to \PSh(M')$ correspond to left-$M'$-right-$M$-sets whose $M'$-action is flat. The conclusion of the above is that, to fully understand Morita equivalences of toposes of topological monoid actions, a necessary next step is to investigate the class of biactions for topological monoids which produce geometric morphisms between the associated toposes. We leave this effort to future work.
\end{rmk}

\subsection{Reflective subcategories of the category of topological monoids}
\label{ssec:monads}

In spite of continuous semigroup homomorphisms not capturing the full richness of geometric morphisms, they do nonetheless produce well-behaved ($2$-)categories of monoids. We show in this section how the classes of monoids we have discussed so far form reflective subcategories of the category of monoids with topologies from which we began.

Let $\mathrm{MonT}_s$, $\mathrm{TMon}_s$, $T_0\mathrm{Mon}_s$, $\mathrm{PMon}_s$, $\mathrm{CMon}_s$ respectively be the $2$-categories of monoids equipped with topologies, topological monoids, $T_0$ topological monoids, \textit{right} powder monoids, and complete monoids, all equipped with continuous semigroup homomorphisms as their $1$-morphisms and conjugations as their $2$-morphisms. We have $2$-functors:
\begin{equation}
\label{eq:monadic}
\begin{tikzcd}
\mathrm{MonT}_s & \ar[l, "G_1"] \mathrm{TMon}_s & \ar[l, "G_2"] \mathrm{T_0Mon}_s & \ar[l, "G_3"] \mathrm{PMon}_s & \ar[l, "G_4"] \mathrm{CMon}_s;
\end{tikzcd}
\end{equation}
all of  these subcategories are full on $1$- and $2$-morphisms. In the following results, we demonstrate that all of  these functors have adjoints, which makes them reflective  as ($2$-)subcategories. Recall that for strict $2$-functors $F:\Ccal \to \Dcal$ and $G:\Dcal \to \Ccal$ to form a strict $2$-adjunction $(F \dashv G)$, we require there to be isomorphisms of categories
\[ \Hom_{\Dcal}(FX,Y) \cong \Hom_{\Ccal}(X,GY),\]
natural in $X$ and $Y$. Since the data of a conjugation consists of an element of the codomain monoid, which is not affected by any of the $G_i$, it suffices to prove that the $G_i$ have left adjoints as $1$-functors; preservation of the $2$-morphisms will then be automatic. As such, the functors constructed in the results below are informally referred to as adjoints.

Note that all of the following results also hold when we restrict to the subcategories of monoid homomorphisms, since all of the units of the adjunctions we construct are continuous monoid homomorphisms.

\begin{lemma}
\label{lem:G1}
The functor $G_1: \mathrm{TMon}_s \to \mathrm{MonT}_s$ has a left adjoint.
\end{lemma}
\begin{proof}
Given a monoid with a topology $(M,\tau)$, we can inductively define sub-topologies $\tau_i$ by letting $\tau_0 = \tau$ and $\tau_{i+1}$ consisting of those open subsets $U \in \tau_i$ such that $\mu^{-1}(U)$ is open in $\tau_i \times \tau_i$. Then letting $\tau_{\infty} := \bigcap_{i=0}^\infty \tau_i$, we claim that $(M,\tau_{\infty})$ is a topological monoid. Indeed, given $U \in \tau_{\infty}$, $\mu^{-1}(U) \in \tau_{\infty} \times \tau_{\infty}$ by construction. This is clearly the finest topology contained in $\tau$ with respect to which the multiplication on $M$ is continuous. We could alternatively define $\tau_{\infty}$ as the collection of $U \in \tau$ such that $\mu^{-k}(U)$ (which is a well-defined subset of $M^{k+1}$ by associativity of multiplication) is open in $\tau \times \cdots \times \tau$ ($k+1$ times) for every positive integer $k$.

The identity homomorphism $(M,\tau) \to (M,\tau_{\infty})$ is automatically continuous. Given any topological monoid $(M',\tau')$ and continuous semigroup homomorphism $\phi: (M,\tau) \to (M',\tau')$, since $\phi$ commutes with multiplication we have that for each $U' \in \tau'$,
\[\mu^{-k}(\phi^{-1}(U')) = \phi^{-1}(\mu'{}^{-k}(U'))\]
is open in $\tau$ by continuity of $\phi$, whence $\phi^{-1}(U')$ is a member of $\tau_{\infty}$. Thus $\phi$ factors (uniquely) through $(M,\tau) \to (M,\tau_{\infty})$, as required to make this map the unit of an adjunction.
\end{proof}

\begin{lemma}
\label{lem:G2}
The functor $G_2: T_0\mathrm{Mon}_s \to \mathrm{TMon}_s$ has a left adjoint.
\end{lemma}
\begin{proof}
If $(M,\tau)$ is a topological monoid, then the equivalence relation $\sim$ on $M$ identifying topologically indistinguishable elements is necessarily a two-sided congruence, since given $m_1 \sim m'_1$, $m_2 \sim m'_2$ and a neighbourhood $U$ of $m_1m_2$, we have that $\mu^{-1}(U)$ contains an open rectangle $U_1 \times U_2$ with $m_i \in U_i$, whence $(m'_1,m'_2) \in U_1 \times U_2$ and hence $m'_1m'_2 \in U$. Thus the quotient map $(M,\tau) \to (M/{\sim}, \tau)$ is a continuous semigroup homomorphism.

Given a continuous semigroup homomorphism $\phi$ from $(M,\tau)$ to a $T_0$ topological monoid $(M',\tau')$ and given $m \sim m'$, observe that $\phi(m)$ and $\phi(m')$ must be topologically indistinguishable and hence equal in $(M',\tau')$ by continuity. Thus $\phi$ factors through the quotient map above, as required.
\end{proof}

\begin{lemma}
\label{lem:G3}
The functor $G_3: \mathrm{PMon}_s \to T_0\mathrm{Mon}_s$ has a left adjoint. 
\end{lemma}
\begin{proof}
The construction of the powder monoid $(\tilde{M},\tilde{\tau})$ associated to $(M,\tau)$ in Theorem \ref{thm:Hausd} (via the construction of the action topology in Theorem \ref{thm:tau}) defines the value of the adjoint functor on objects and provides a candidate for the unit in the quotient homomorphism $(M,\tau) \to (\tilde{M},\tilde{\tau})$.

Let $\phi:(M,\tau) \to (M',\tau')$ be a continuous semigroup homomorphism with $(M',\tau')$ a powder monoid. Given $U' \in T'$ and $p \in M$, consider
\[\phi^{-1}(\Ical_{U'}^{\phi(p)}) = \{q \in M \mid \phi(q)^*(U') = \phi(p)^*(U')\},\]
which is clearly contained in
\begin{align*}
\Ical_{\phi^{-1}(U')}^p &= \{q \in M \mid q^*(\phi^{-1}(U')) = p^*(\phi^{-1}(U'))\}\\
&= \{q \in M \mid \phi^{-1}(\phi(q)^*(U')) = \phi^{-1}(\phi(p)^*(U'))\}.
\end{align*}
Since the former is open in $\tau$, so is the latter, and hence $\phi^{-1}(U') \in T$. Since any open in $\tau'$ is a union of basic opens in $T'$, we conclude that $\phi$ factors through $(M,\tilde{\tau})$, and hence through $(\tilde{M},\tilde{\tau})$ by the proof of Lemma \ref{lem:G2}, as required.
\end{proof}

For the fourth result, we recycle the proof of Theorem \ref{thm:locextend}.
\begin{schl}
\label{schl:G4}
The functor $G_4: \mathrm{CMon}_s \to \mathrm{PMon}_s$ has a left adjoint. 
\end{schl}
\begin{proof}
Given a powder monoid $(M,\tau)$, a complete monoid $(L',\rho')$ and a semigroup homomorphism $g:(M,\tau) \to (L',\rho')$, we must show that $g$ factors uniquely through the canonical monoid homomorphism $u: (M,\tau) \to (L,\rho)$, where the latter is the completion of $(M,\tau)$.

As in the proof of Theorem \ref{thm:locextend}, we may assume $\phi: M \to L'$ is a monoid homomorphism, since we may factor $\phi$ as a monoid homomorphism followed by an inclusion of subsemigroups, and the intermediate monoid is canonically a complete monoid with the restriction topology by Theorem \ref{thm:inccomplete}. We then construct the factoring $M$-set homomorphism $\psi: L \to f^*(L')$ just as in the proof of Theorem \ref{thm:locextend}, which did not depend on injectivity of $\phi$.
\end{proof}

Our results from previous sections demonstrate that the units of these four adjunctions all induce equivalences at the level of toposes of continuous actions of monoids.

Recalling the asymetry in the definition of powder monoids (see the comments after Definition \ref{dfn:powder}), we briefly consider the ($2$-)categories of left powder monoids and left complete monoids.

For this purpose, we employ the dual of the notation introduced in Section \ref{ssec:necessary}, writing ${}_x^p\Ical$ for the necessary clopen associated to an element $x$ in a left $M$-set $X$ and $p \in M$. Then we obtain a complementary result to Lemma \ref{lem:InA2}.
\begin{lemma}
\label{lem:lrclopen}
Let $U$ be a subset of $M$. Let $A = \Ical_U^p$ and $B = {}_U^p\Ical$. Then we have $\Ical_B^p = {}_A^p\Ical$. In particular, a two-sided powder monoid has a base of clopens expressible in this coincident form.
\end{lemma}
\begin{proof}
After expanding the definitions, we find that both subsets are equal to the set of $q \in M$ such that for all $w,z \in M$, $q \in w^*(U){}^*z$ if and only if $p \in w^*(U){}^*z$ (this is reminiscent of a double-coset construction).
\end{proof}

\begin{schl}
\label{schl:G3'}
The inclusion $G'_3$ of the sub-$2$-category $\mathrm{P'Mon}_s$ of \textbf{left} powder monoids into $T_0\mathrm{Mon}_s$ has a left adjoint. The idempotent monads $P$ and $P': T_0\mathrm{Mon}_s \to T_0\mathrm{Mon}_s$ induced by $G_3$ and $G'_3$ respectively commute, in the sense that $PP' = P'P$, whence the ($2$-)category $\mathrm{P''Mon}_s$ of two-sided powder monoids is also a reflective subcategory of $T_0\mathrm{Mon}_s$.
\end{schl}
\begin{proof}
The first part is clear by inspection of the proof of Lemma \ref{lem:G3}, which can be dualized without difficulty. To see that $PP' = P'P$, fix a ($T_0$) monoid $(M,\tau)$ and consider the left action topology associated to the right action topology associated to $\tau$. This is generated by those $U \in \tilde{\tau}$ such that for every $p \in M$, the subset ${}_U^p\Ical$ is in $\tilde{\tau}$ (and hence, being clopen, in $T$). It is necessary and sufficient to verify that for each fixed $p$ that $B := {}_U^p\Ical$ has $\Ical_B^p \in \tau$. By Lemma \ref{lem:lrclopen}, letting $A := \Ical_U^p$, we have that $\Ical_B^p = {}_A^p\Ical$, whence we have $A$ in the left action topology associated to $\tau$ and hence $U$ is in the right action topology associated to this. In summary, the topologies obtained by applying $P$ and $P'$ in either order are the same.
\end{proof}

\begin{schl}
\label{schl:G4'}
The inclusion $G'_4$ of the sub-$2$-category $\mathrm{C'Mon}_s$ of \textbf{left} complete monoids into $\mathrm{P'Mon}_s$ has a left adjoint.
\end{schl}
\begin{proof}
The proof proceeds exactly as for Scholium \ref{schl:G4}.
\end{proof}

With Scholiums \ref{schl:G3'} and \ref{schl:G4'} we may extend the diagram of monadic functors \eqref{eq:monadic} as follows,
\begin{equation}
\label{eq:monadic2}
\begin{tikzcd}
& \ar[dl, "G_3"] \mathrm{PMon}_s & & \ar[ll, "G_4"'] \mathrm{CMon}_s \\
\mathrm{T_0Mon}_s & & \ar[dl, "G_3"'] \ar[ul, "G'_3"] \mathrm{P''Mon}_s \\
& \ar[ul, "G'_3"'] \mathrm{P'Mon}_s & & \ar[ll, "G'_4"] \mathrm{C'Mon}_s. \\
\end{tikzcd}
\end{equation}
Clearly, there is more of this picture to fill in; see Conjecture \ref{conj:complete} below.

We could have included further reflective subcategories of $\mathrm{TMon}_s$ in this section, such as the category of zero-dimensional monoids, but we hope the examples we have included are sufficiently illustrative.

\begin{rmk}
One might wonder whether the underlying $1$-categories of these classes of monoids are monadic over $\Top$, given that this is the case for the category of topological monoids. In the first instance, we can observe that the subcategory of $T_0$ spaces is reflective in $\Top$, and that the category of $T_0$ monoids is `crudely' monadic (see \cite[Proposition 3.5.1]{TTT}) over the category of $T_0$ spaces, so this category is monadic over $\Top$. The same argument works for the zero-dimensional $T_0$ monoids.

On the other hand, powder monoids are not monadic over $\Top$. Indeed, one can check that a free zero-dimensional $T_0$ monoid (equivalently, a free monoid on a zero-dimensional $T_0$ space $X$) is a powder monoid, since the underlying space of such a monoid is the coproduct $\coprod_{n \in \Nbb} X^n$, and the topology on this space is generated by the basic clopen sets in the components $X^n$ which are products of clopens in $X$; it is easily checked that these are necessary clopens. As such, the `free powder monoid monad' coincides with the `free zero-dimensional $T_0$ monoid monad' on $\Top$, but we saw in Remark \ref{rmk:Q} that zero-dimensional $T_0$ monoids do not coincide with powder monoids, so powder monoids are not the algebras for this monad. It is unclear whether the category of complete monoids is monadic over $\Top$, but we anticipate that this will not be the case.

We thank the anonymous reviewer who pointed out the mistaken claim of monadicity in an earlier version of the present paper.
\end{rmk}

%% file: TTMA_Conclusion.tex
\section{Conclusion}
\label{sec:conclusion}

While we have gone a long way in establishing the properties of toposes of the form $\Cont(M,\tau)$ and their canonical representatives in this paper, it is clear that there are multiple avenues for future exploration of the subject.\footnote{Aspirationally, the author wishes to systematically explore them all, but realistically this will only be possible through a collaborative effort; offers for collaboration are welcomed.}

\subsection{Pathological powder monoids}
\label{ssec:moremonoid}

The reader may have noticed that we did not exhibit any examples illustrating the asymmetry in the definition of powder monoids. This is because our main classes of examples, prodiscrete monoids and powder groups, are both blind to this distinction, since their definitions are stable under dualizing. Similarly, any commutative right powder monoid is also a left powder monoid. These cases make constructing examples of right powder monoids which are not left powder monoids difficult. Nonetheless, we posit that:
\begin{conj}
\label{conj:powdery}
There exists a right powder monoid which is not a left powder monoid.
\end{conj}

Scholium \ref{schl:G3'} puts some limits on Conjecture \ref{conj:powdery}, since it says that any right powder monoid is at most one step away from also being a left powder monoid. In particular, we never get an infinite nested sequence of topologies on a monoid by repeatedly computing the associated right and left action topologies. We have not demonstrated the same results for complete monoids, but we expect them to hold:
\begin{conj}
\label{conj:complete}
The right completion of a left powder monoid or left complete monoid retains the respective property, and dually for left completions of right powder monoids or right complete monoids. However, we expect that there exists a right complete monoid which is not a left powder monoid.
\end{conj}

Another way of expressing Conjectures \ref{conj:powdery} and \ref{conj:complete} is to say that we expect the diagram of monadic full and faithful functors \eqref{eq:monadic2} to extend as follows:

\[\begin{tikzcd}
& & \ar[dl] \mathrm{CMon}_s & & \\
& \ar[dl] \mathrm{PMon}_s & &
\ar[ul] \ar[dl] \mathrm{CP'Mon}_s & \\
\mathrm{T_0Mon}_s & &
\ar[dl] \ar[ul] \mathrm{P''Mon}_s & &
\ar[ul] \ar[dl] \mathrm{C''Mon}_s,\\
& \ar[ul] \mathrm{P'Mon}_s & &
\ar[dl] \ar[ul] \mathrm{C'PMon}_s \\
& & \ar[ul] \mathrm{C'Mon}_s & & 
\end{tikzcd}\]
where the notation is the intuitive extension of that employed in \eqref{eq:monadic2} and each inclusion represented is non-trivial.

\subsection{Finitely generated complete monoids}
\label{ssec:fgcompmon}

Besides these conjectures characterizing pathological examples, there is plenty of ground still to cover in understanding these classes of monoids. What does a `generic' complete monoid look like, beyond what has been shown in this article? Is it possible to classify them?

For example, given an element $x$ in a complete monoid, we may consider the closure of the submonoid generated by $x$, which by Corollary \ref{crly:closed} is a complete monoid. One might consider this an instance of a `complete monoid generated by one element'\footnote{We include the quotation marks to emphasise that this submonoid is not generated by $x$ in an algebraic sense.}. We can identify such monoids as the canonical representatives of toposes admitting a hyperconnected morphism from $\PSh(\Nbb)$. By Corollary \ref{crly:prodisc} these are commutative prodiscrete monoids. Analogously, `finitely generated complete monoids' would correspond to complete monoids representing toposes admitting hyperconnected geometric morphisms from $\PSh(F_{n})$ for $F_n$ the free (discrete) monoid on some number $n$ of elements. By Proposition \ref{prop:prince2} they correspond to filters of right congruences on $F_n$, which we expect to have a tame classification. Is it possible to identify the `finitely presented complete monoids' amongst these? One could go on to investigate the properties of various ($2$-)categories of such monoids, taking advantage of results such as such as those in Section \ref{ssec:monads}. Future applications of this theory may rely on understanding these questions.

\subsection{Invariant properties}
\label{ssec:mitopmon}

In investigating complete monoids, it will be desirable to extend the results we obtained with Jens Hemelaer in \cite{MPaTP}, where we explored how properties of discrete monoids extend to properties of their toposes of actions and \textit{vice versa}. In this regard, we can already glean some positive results. Whilst we saw in Example \ref{xmpl:notpro2} that a complete monoid Morita equivalent topological group need not be a group, we have the next best result.
\begin{prop}
\label{prop:densegroup}
Let $(M,\tau)$ be a topological monoid. The following are equivalent:
\begin{enumerate}
	\item $\Cont(M,\tau)$ is an atomic topos;
	\item The completion of $(M,\tau)$ has a dense subgroup;
	\item The group of units in the completion of $(M,\tau)$ is dense;
	\item For each open relation $r \in \underline{\Rcal}_{\tau}$ and $m \in M$, there exists $m' \in M$ with $(mm',1) \in r$.
\end{enumerate}
\end{prop}
\begin{proof}
($3 \Rightarrow 2 \Rightarrow 1$) If the group of units of (the completion of) $(M,\tau)$ is dense, then clearly this provides a dense subgroup. If the $(G,\tau|_G)$ is a dense subgroup of (the completion of) $(M,\tau)$, then $\Cont(M,\tau)$ admits a hyperconnected morphism from $\PSh(G)$, whence the former is an atomic topos, by Remark \ref{rmk:atom}.

($1 \Rightarrow 3$) If $\Cont(M,\tau)$ is atomic, all of the supercompact objects are necessarily atoms. The opposite of $\underline{\Rcal}_{\tau}$ is easily verified to satisfy the amalgamation property and joint embedding property of \cite[Definition 3.3]{TGT}, and the canonical point of $(M,\tau)$ provides an $\underline{\Rcal}_{\tau}\op$-universal, $\underline{\Rcal}_{\tau}\op$-ultrahomogeneous object in $\mathrm{Ind}$-$\underline{\Rcal}_{\tau}\op$, namely the completion of $(M,\tau)$ itself. Thus by \cite[Theorem 3.5]{TGT}, there is a topological group $(G,\sigma)$ representing the topos (and having the same canonical point), and by the more detailed description of this construction in \cite[Proposition 5.7]{ATGT}, the group so constructed is precisely the group of units of $(M,\tau)$, and this group is dense in $(M,\tau)$.

($1 \Leftrightarrow 4$) Consider $\underline{\Rcal}_{\tau}$ when $\Cont(M,\tau)$ is atomic. Since all of the $M/r$ are atoms, all of the morphisms in this category are (strict) epimorphisms, which means that in particular the canonical monomorphisms $[m]: m^*(r) \to r$ are isomorphisms, providing $[m']: r \to m^*(r)$ such that $(mm',1) \in r$ and $(m'm,1) \in m^*(r)$, but the latter is implied by the former, so the former suffices. Conversely, if $4$ holds, then all of the morphisms of $\underline{\Rcal}_{\tau}$ are strict epimorphisms, whence the reductive topology coincides with the atomic topology, so $\Cont(M,\tau)$ is atomic as required.
\end{proof}

We anticipate a plethora of results of this nature, where a complete monoid generates a topos having a property $Q$ if and only if it has a dense submonoid having property $P$, where $P$ is the property corresponding to $Q$ for toposes of discrete monoid actions; the above is the case where $Q$ is the property of being atomic and $P$ is the property of being a group; see \cite[Theorem 2.4]{MPaTP}. In order to attain these results, some preliminary work is needed to accumulate the relevant factorization results for properties of geometric morphisms along hyperconnected geometric morphisms.

On the subject of geometric morphisms, we have two further conjectures. In the hope of improving Scholium \ref{schl:factors} to a more elegant result, we begin with the following:
\begin{conj}
\label{conj:characterization}
There exists an intrinsic characterization, independent of the representing monoid $M$, of those hyperconnected geometric morphisms with domain $\PSh(M)$ identifying toposes of the form $\Cont(M,\tau)$.
\end{conj}
To be more specific, observe that Proposition \ref{prop:intrinsic} provides an intrinsic sufficient condition for a hyperconnected morphism to express its codomain topos in terms of a topology on any monoid representing its domain topos; Conjecture \ref{conj:characterization} posits that it should be possible to refine this to a necessary and sufficient condition.

We also record our expectation that the converse of Scholium \ref{schl:Morita2} fails.
\begin{conj}
\label{conj:inclusion}
There exists a complete monoid $(M',\tau')$ and an idempotent $e \in M'$ such that the semigroup inclusion $M:= eM'e \hookrightarrow M'$ is \textit{not} a Morita equivalence, but the induced geometric inclusion $\Cont(M,\tau'|_{M}) \hookrightarrow \Cont(M',\tau')$ is an equivalence.
\end{conj}

\subsection{Topological semi-Galois theory}
\label{ssec:TSGT}

While we obtained many necessary conditions for a topos to be of the form $\Cont(M,\tau)$ in Section \ref{sec:properties}, and expounded many properties of the canonical site of principal actions for such a topos, we did not arrive at \textit{sufficient} conditions for a site to yield such toposes. Having such conditions will eventually be important for identifying when a topological monoid can be used to present a topos constructed by other means.

In a similar vein, we saw in Example \ref{xmpl:notpro2} how having a theory classified by the topos of actions of a given topological monoid to hand makes computing its completion much easier. We characterized toposes of the form $\Cont(M,\tau)$ in terms of their canonical points, which correspond to $\Set$-models of theories classified by $\Cont(M,\tau)$. As such, it is clear that another useful set of sufficient conditions to have to hand would be those specifying when a (geometric) theory $\Tbb$ admits a $\Set$-model corresponding to a point of $\Set[\Tbb]$ of the form required by Theorem \ref{thm:characterization}. In other words, which theories are classified by toposes of actions of topological monoids?

Combining these goals, the natural direction to proceed is to extend and generalize the main results in Caramello's Topological Galois Theory \cite{TGT}. The resulting \textit{topological semi-Galois theory} will also be an extension of the theory of semi-Galois categories studied in \cite{SGC}.

\begin{rmk}
Since the original version of this article was written, the author has completed their PhD thesis in which much progress was made towards the construction of topological semi-Galois theory \cite[Chapter 6]{Thesis}. For a site $(\Ccal,J)$ with the right properties to produce a supercompactly generated topos (a so-called `principal site') where $\Ccal$ carries a compatible factorization system, the existence of a point of $\Sh(\Ccal,J)$ of the form required by Theorem \ref{thm:characterization} is equivalent to the existence of an object with certain properties in $\Ind(\Ccal\op)$. This yields several sufficient conditions for the existence of a presentation of $\Sh(\Ccal,J)$ as a topos of continuous actions of a topological monoid. However, the question of which theories are classified by toposes of monoid actions remains open.
\end{rmk}

\subsection{Actions on topological spaces}

A proof of, or counterexample to, Conjecture \ref{conj:powdery} will establish the extent of the symmetry in the type of Morita equivalence studied in this paper. Whichever way this result falls, however, we have shown in this article that the category of right actions of a topological monoid on discrete spaces is a very coarse invariant of such a monoid. Moreover, anyone interested in actions of topological monoids is more likely to wish to examine their actions on more general classes of topological space. A solution to this, which is viable in any Grothendieck topos $\Ecal$, is to first consider the topos $[M\op,\Ecal]$, which is constructed as a pullback (the lower square) in $\TOP$,
\begin{equation}
\label{eq:pbEM}
\begin{tikzcd}
\Ecal \ar[r] \ar[d] \ar[dr, phantom, "\lrcorner", very near start] & \Set \ar[d] \\
{[M\op{,}\Ecal]} \ar[r] \ar[d] \ar[dr, phantom, "\lrcorner", very near start] & {[M\op{,}\Set]} \ar[d] \\
\Ecal \ar[r] & \Set,
\end{tikzcd}	
\end{equation}
since $M$ induces an internal monoid in $\Ecal$ by its image under the inverse image functor of the global sections morphism of $\Ecal$. Taking $\Ecal$ to be the topos of sheaves on a space $X$, we can view the objects of $[M\op,\Ecal]$ as right actions of $M$ on spaces which are discrete fibrations over $X$; taking $\Ecal$ to be a more general topos of spaces, we similarly get actions of $M$ on such spaces.

In each case, we can construct the subcategory of $[M\op,\Ecal]$ on the actions which are continuous with respect to a topology $\tau$ on $M$. In the best cases, this will produce a topos hyperconnected under $[M\op,\Ecal]$, and the analysis can proceed analogously to that of the present article, taking advantage of the $\Ecal$-valued point in \eqref{eq:pbEM}. If this can be done with sufficient generality, one will be able to address a host of interesting Morita-equivalence problems in this way.

\subsection{Space-enriched categories}
\label{ssec:topcat}

Another direction to generalize is to consider topologies on small categories with more than one object. Let $\Ccal$ be a category with set of objects $C_0$, set of morphisms $C_1$, identity map $i:C_0 \to C_1$, domain and codomain maps $d,c: C_1 \to C_0$, and composition $m:C_2\to C_1$, where $C_2$ is the pullback:
\[\begin{tikzcd}
C_2 \ar[r] \ar[d] \ar[dr, phantom, "\lrcorner", very near start] & C_1 \ar[d, "d"] \\
C_1 \ar[r, "c"'] & C_0.
\end{tikzcd}\]
A presheaf on $\Ccal$ can be expressed as an object $a:F_0 \to C_0$ of $\Set/C_0$, equipped with a morphism $b:F_1 \to F_0$ where $F_1$ is the pullback,
\[\begin{tikzcd}
F_1 \ar[r, "\pi_2"] \ar[d, "\pi_1"'] \ar[dr, phantom, "\lrcorner", very near start] & F_0 \ar[d, "a"] \\
C_1 \ar[r, "c"'] & C_0,
\end{tikzcd}\]
satisfying $a \circ b = d \circ \pi_2$, $b \circ (\id_{F_0} \times_{C_0} i) = \id_{F_0}$ and $b \circ (b \times_{C_0} \id_{C_1}) = b \circ (\id_{F_0} \times_{C_0} m)$.

Suppose we equip $C_0$ with the discrete topology and $C_1$ with a topology $\tau$ making the structure maps continuous; we could for instance equip the hom-sets of $C$ with topologies compatible with composition. We might then say that a presheaf is \textit{continuous} with respect to $\tau$ if $b:F_1 \to F_0$ is a continuous map when $F_0$ is equipped with the discrete topology and $F_1$ is equipped with the pullback of the topology on $C_1$. Yet again, we can consider the full subcategory of $\PSh(\Ccal)$ on these presheaves, and we expect it to be coreflective. In good cases, we will have the analogue of Proposition \ref{prop:hyper}, and the analysis can proceed as in this paper, leading to a class of genuine topological categories representing these toposes.

A special case of this construction which may be of interest to readers of the present paper is when the category in question is taken to be the left-cancellative category constructed from an inverse semigroup (denoted $L(S)$ in \cite{MEiS}). This category inherits a topology from the inverse semigroup, and the category of continuous presheaves for this topological category forms an invariant of such a semigroup which is typically different from that obtained by considering actions of the associated monoid.

\subsection{Localic monoids and constructiveness}
\label{ssec:localic}

Topos theorists tend to try to work constructively wherever possible, since doing so ensures that all results can be applied over an arbitrary topos. In this light, our frequent reliance on complementation in the underlying sets of our monoids is quite restrictive, since \textit{a priori} it means our results are applicable only over Boolean toposes, and we have not formally demonstrated here that they apply even to this level of generality.

From a constructive perspective, more suitable objects of study than topological monoids would be \textbf{localic monoids}, which are monoids in the category of locales over a given base topos, typically $\Set$. Early on in the research for this paper, Steve Vickers suggested that we consider pursuing this direction. However, while the category of actions of a localic monoid on sets (again viewed as discrete spaces) is easy to define, it is much harder to show that such a category is a topos. In Johnstone \cite[Example B3.4.14(b)]{Ele}, we see that the more powerful results of \textit{descent theory} are required to show that categories of actions of localic groups are toposes. While descent theory is an important tool, it is far more abstract than the comonadicity theorem we used in Corollary \ref{crly:topos}, making concrete characterization results for these toposes more challenging to prove.

While we did not treat localic monoids in this article, the present work is a valuable tool in that analysis. Indeed, the functor sending a locale to its topological space of points preserves limits, so that it provides a canonical `forgetful' functor from a category of actions of a localic monoid to a category of actions of a topological monoid. We anticipate that, just as in Section \ref{sec:properties}, this functor can be used to constrain the properties of a category of actions of a localic monoid. 

We should mention that another obstacle in our study of localic monoids is again a lack of easily tractable examples, especially examples of localic monoids (or even localic groups) which one can show are \textit{not} Morita equivalent to topological monoids in their actions on discrete spaces. While the construction of the localic group $\mathrm{Perm}(A)$ of permutations of a locale $A$, described by Wraith in \cite{LocGrp}, is used as a basis for the Localic Galois Theory of Dubuc \cite{LGT}, the latter author provides no specific examples of instances of these. We expect that the construction of such examples will further illuminate the appropriate approach to studying categories of actions of localic monoids.